\documentclass[11pt]{amsart}

\usepackage{amsrefs}
\usepackage{amsmath}
\usepackage{amssymb}
\usepackage{hyperref}
\usepackage{stackrel}
\usepackage{xcolor}
\usepackage{enumerate}
\usepackage{verbatim}
\usepackage{cancel}
\usepackage{setspace}
\allowdisplaybreaks

\newtheorem{teo}{Theorem}[section]
\newtheorem{lem}[teo]{Lemma}
\newtheorem{defi}[teo]{Definition}

\newtheorem{cor}[teo]{Corollary}
\newtheorem{pro}[teo]{Proposition}

\newtheorem{obs}[teo]{Remark}

\DeclareMathOperator{\dist}{dist}
\DeclareMathOperator{\Hess}{Hess}
\DeclareMathOperator{\spann}{span}
\DeclareMathOperator{\epi}{epi}
\DeclareMathOperator*{\esssup}{ess\,sup}

\newcommand{\sfe}{S_{F_e}}
\newcommand{\ssubset}{\subset\joinrel\subset}
\newcommand{\Kgzero}{K_{\mathfrak{g}^{\ast} \backslash \left\{ 0 \right\} }}

\begin{document}
	
	\title[Mollifier smoothing of left-invariant $C^0$-Finsler structures]{Mollifier smoothing of left-invariant strongly convex $C^0$-Finsler structures on Lie groups and convergence of extremals}
	
	\author{Ryuichi Fukuoka}
	\address{Department of Mathematics, State University of Maring\'a,
		87020-900, Maring\'a, PR, Brazil \\ email: rfukuoka@uem.br}
	
	\author{Anderson Macedo Setti}
	\address{Department of Exact and Earth Sciences, Federal University of Mato Grosso, 78060-900, Cuiabá, MT, Brazil, email: anderson.setti@ufmt.br}
	
	\date{\today}
	
	\begin{abstract}
		Let $M$ be a smooth manifold and $TM$ its tangent bundle. 
A $C^0$-Finsler structure of $M$ is a continuous function $F:TM \rightarrow \mathbb{R}$ such that $F$ restricted to each tangent space $T_xM$ of $M$ is an asymmetric norm.
$F$ is strongly convex if $F\vert_{T_xM}$ is a strongly convex asymmetric norm for every $x \in M$. 
Let $G$ be a Lie group endowed with a left-invariant strongly convex $C^0$-Finsler structure $F$.
	We introduce a smoothing $F_{\varepsilon}$ of $F$, which is a left-invariant version of the mollifier smoothing presented previously by the same authors.
	We study extremals $x(t)$ on $(G,F)$ using the Pontryagin maximum principle.
	Given $(x_0,\alpha_0)$ in the cotangent bundle $T^\ast G$ of $G$, we prove that there exist a unique Pontryagin extremal $t\in \mathbb{R} \mapsto (x(t), \alpha(t))$ such that $(x(0),\alpha(0))=(x_0,\alpha_0)$.
	Moreover, if $t \in \mathbb{R} \mapsto (x_\varepsilon(t), \alpha_{\varepsilon}(t))$ is the unique Pontryagin extremal on $(G,F_\varepsilon)$ such that $(x_\varepsilon(0), \alpha_{\varepsilon}(0))=(x_0, \alpha_0)$, then we prove that $(x_{\varepsilon}(t),\alpha_\varepsilon(t))$ converges uniformly to $(x(t),\alpha(t))$ on compact intervals of $\mathbb{R}$.
	\end{abstract}
	
	
	
	\maketitle
	
	\section{Introduction}
		
	Let $M$ be a smooth manifold, $T_xM$ be its tangent space at $x\in M$, $T^{\ast}_xM$ be its cotangent space at $x$, $TM = \{(x,y): x\in M, y \in T_xM\}$ be its tangent bundle and $T^\ast M = \{(x,\alpha): x \in M, \alpha \in T^\ast_x M\}$ be its cotangent bundle. 
	A {\em $C^0$-Finsler structure} of $M$ is a continuous function $F:TM\rightarrow [0,\infty)$ such that $F(x,\cdot):T_xM\rightarrow [0,\infty)$ is an asymmetric norm for every $x\in M$ (see Definition \ref{Definicao de norma assimetrica}).
	
	We will use the term ``Finsler geometry'' according to its more usual definition. 
	Its basic theory follows closely the way that Riemannian geometry was developed, mainly because its objects are smooth and differential calculus is applicable directly (see \cite{BaoChernShen}).
	On the other hand, we need additional techniques when dealing with some geometrical structures on smooth manifolds that generalizes Riemannian and Finsler structures.
	
	In this work, we are interested to study minimizing paths on $C^0$-Finsler manifolds.
	In Riemannian (and Finsler) geometry, we have geodesic fields on tangent bundles (slit tangent bundles) as a basic tool.
	Although general $C^0$-Finsler and $C^0$-sub-Finsler manifolds don't admit such type of vector fields, we can restrict ourselves to structures $F$ that satisfies some kind of horizontal smoothness.
	We can define a control system where its vector fields are continuously differentiable and represent $F$.
	Instead of using the Euler-Lagrange equations on the energy functional in order to obtain geodesics, we can apply the Pontryagin maximum principle (``PMP'' in short) on this control system in order to obtain extremals of $(M,F)$. 
	This is an approach that is used in works such as \cites{Agrachev-Barilari-Boscain, Agrachev-Gamkrelidze-feedback-1, Agrachev-Lee, Ardentov-LeDonne-Sachkov, Ardentov-Lokutsievskiy-Sachkov, Barilari-Boscain-LeDonne-Sigalotti, Berestovskii-Zubareva-Engel, Gribanova, Hakavuori-Infinite-geodesics, Lokutsievskiy-Heisenberg, Sachkov-bang-bang-Engel}, and an overview of them can be found in  \cites{PrudencioFukuoka, Fukuoka-Rodrigues}.
	
	Let us summarize this work.	 
	
	A $C^0$-Finsler structure being {\em of Pontryagin type} is the way 
	we define the family of continuously differentiable vector fields $\{X_u\}$ of the control theory that is useful for our purposes (see Definition \ref{Variedade de Finsler do Tipo Pontryagin} and \cite{Fukuoka-Rodrigues}).
	This family is given by unit vector fields defined on an open subset $U$ of $M$, and the main idea is that $\{X_u\}$ satisfy the minimum requirements of the PMP for the problem of minimizing paths parametrized by arc length.
	As a consequence of applying the PMP to this control system, we get the extended geodesic field $\mathcal{E}$, which is a multivalued vector field on the {\em slit cotangent bundle} $T^\ast U\backslash 0 = \{(x,\alpha): x\in U, \alpha \in T^\ast_x U \backslash \{0\}\}$.
	Its integral curves $(x(t), \alpha(t))$ are the Pontryagin extremals of $(U,F)$ and $x(t)$ are extremals of $(U,F)$. 
		Every minimizing curve of $(M,F)$ parametrized by arc length with its trace on $U$ is an extremal of $(U,F)$.
	In the particular case where $F$ is a Finsler structure, the fiberwise Legendre transform from $TM \backslash 0 \rightarrow T^\ast M \backslash 0$ is a smooth map that identifies the geodesic spray on $TM \backslash 0$ and the extended geodesic field on $T^\ast M \backslash 0$.
	In other words, if we consider the identification given by the Legendre transform, then the geodesic spray on $T M \backslash 0$ is included in the theory of extended geodesic fields on $T^\ast M \backslash 0$. 
	The proof of all these facts can be found in \cite{Fukuoka-Rodrigues}.
		
	Lie groups $G$ endowed with left-invariant $C^0$-Finsler structures $F$ are of Pontryagin type. 
	The control system is given by the unit left-invariant vector fields.
	In this work, we study Lie groups endowed with left-invariant strongly convex $C^0$-Finsler structures.
	In this case, $\mathcal{E}$ is a locally Lipschitz vector field defined on $T^\ast G\backslash 0$ (see \cite{Fukuoka-Rodrigues}), what ensures that for a given $t_0 \in \mathbb{R}$ and $(x_0, \alpha_0) \in T^\ast G \backslash 0$, there exist a unique maximal Pontryagin extremal $(x(t),\alpha(t))$ satisfying $(x(t_0), \alpha(t_0)) = (x_0, \alpha_0)$.
	We prove here that these maximal Pontryagin extremals are defined in $\mathbb{R}$ (see Theorem \ref{theorem_extremal de Pontryagin em grupos}).
	
	Let $\mathfrak{g}$ be the Lie algebra of $G$. Left-invariant $C^0$-Finsler structures $F$ on $G$ are determined by their restriction $F_e$ to $\mathfrak{g}$ as it happens for left-invariant Riemannian structures.
	In the case where $F$ is a left-invariant strongly convex $C^0$-Finsler structure on $G$, we have that $F_e$ is a strongly convex asymmetric norm (see Definition \ref{Fortemente convexa - definicao artigo R-H}).
	We define a mollifier smoothing $F_\varepsilon$ of $F$, which is a left-invariant version of the mollifier smoothing on $C^0$-Finsler structures defined on \cite{Fukuoka-Setti}.
	In order to do so, we first define the mollifier smoothing $\tilde{F}_\varepsilon$ of $F_e$, and then we spread $\tilde{F}_\varepsilon$ through left-invariance on $G$, obtaining $F_\varepsilon$.
	Given $(x_0,\alpha_0) \in T^\ast G\backslash 0$ and $t_0 \in \mathbb{R}$, consider the Pontryagin extremals $t \in \mathbb{R} \mapsto (x(t),\alpha(t))$ and $t \in \mathbb{R} \mapsto (x_\varepsilon(t), \alpha_\varepsilon(t))$ of $(G,F)$ and $(G,F_\varepsilon)$ respectively satisfying $(x(t_0),\alpha(t_0)) = (x_\varepsilon(t_0), \alpha_\varepsilon(t_0)) = (x_0, \alpha_0)$.
	The main result of this work states that given a compact interval $[a,b] \subset \mathbb{R}$,
	$(x_\varepsilon(t), \alpha_\varepsilon(t))$ converges uniformly to $(x(t), \alpha(t))$ on $[a,b]$ when $\varepsilon$ converges to zero (see Theorem \ref{theorem_principal no fibrado cotangente}).
	
	Let us make comments about some subjects related to this work. 
	
	If $(M,F)$ is a Finsler manifold, and we consider $(x_0, y_0) \in TM$, then there exist an interval $I \subset \mathbb{R}$ and a unique maximal geodesic $x:I \rightarrow M$ such that $x(0)=x_0$ and $\dot{x}(0) = y_0$.
	We can ask if Theorem \ref{theorem_extremal de Pontryagin em grupos} can be used to prove this result for Lie groups endowed with left-invariant strongly convex $C^0$-Finsler structures.
	The answer is that the existence and uniqueness result of this initial value problem, that holds for Finsler manifolds, doesn't hold in our case.
	As a counterexample, we can  
	consider the Lie group $G=\{(x^1, x^2) \in \mathbb{R}^2: x^2>0\}$, $(x^1, x^2)\cdot (z^1, z^2) = (x^2 z^1 + x^1, x^2 z^2)$, which is isomorphic to the connected component of the identity element $e$ of the affine group $\text{Aff}(\mathbb{R})$. 
	Let $F$ be the left-invariant strongly convex $C^0$-Finsler structure given in Section 9.2 of \cite{Fukuoka-Rodrigues}, which extremals can be explicitly calculated using the PMP (see also \cite{Gribanova}).
	We have that given the initial condition $(x_0,y_0) = (x_0,(0,1)) \in T G \cong G \times \mathbb{R}^2$, there exist infinitely many geodesics $x(t) \in G$ satisfying $x(0) = x_0$ and $\dot{x}(0) = y_0$, and even local uniqueness of this initial value problem doesn't hold.
	On the other hand, for every $(x_0,\alpha_0) \in T^{\ast}G$, there exists a unique Pontryagin extremal $(x(t), \alpha(t))$ such that $(x(0),\alpha(0))=(x_0,\alpha_0)$.
	This example show that, in general, it is more natural to study extremals of $C^0$-Finsler manifolds using Hamiltonian formalism and the PMP on $T^{\ast}M\backslash 0$ than a structure like geodesic fields on $TM \backslash 0$.
	
	We can define a distance function $d_F$ on a connected $C^0$-Finsler manifold $(M,F)$, where $d_F(x_1,x_2)$ is the infimum of the lengths of piecewise continuously differentiable paths that connect $x_1$ to $x_2$ (in this order).
	Notice that, in general, $d_F$ is asymmetric.
	Although given $(x_0,y_0) \in TM$ we don't have the local existence and uniqueness of an extremal $x(t)$ for the initial value problem $x(0)=x_0$ and $\dot{x}(0)=y_0$, the boundary value problem is much more well-behaved, even for general $C^0$-Finsler structures.
	We have that if $x_1, x_2 \in (M,F)$ are such that $d_F(x_1,x_2)$ is small enough, then there exist a minimizing path connecting $x_1$ and $x_2$ (see \cite{Mennucci-asymmetric-distances} for the original result and Remark 2.6 of \cite{Fukuoka-Rodrigues} for an adaptation of this result for $C^0$-Finsler manifolds).
	This result is only about existence, it doesn't state anything about uniqueness.
	The technique of the proof of this boundary value problem is metric geometry, which is another important area in studying $C^0$-Finsler manifolds.
	For another instance of the use of metric geometry in $C^0$-Finsler geometry, we can calculate minimizing paths explicitly in some specific $C^0$-Finsler manifolds, even when PMP isn't applicable (see \cite{Fukuoka-large-family}).

	In \cite{Fukuoka-Setti}, the authors introduced a mollifier smoothing for general $C^0$-Finsler structures, which is done in two steps:
	The first step is a vertical smoothing on each tangent space.
	In order to assure that this smoothing is a Minkowski norm (and therefore strongly convex), an additional term is introduced in the bump function.
	The second part is a horizontal smoothing, which is done on coordinate open subsets $\{U_\beta\}$ of $M$, weighted with respect to a partition of the unity subordinated to $\{U_\beta\}$. 
	If $F$ is a Finsler structure, the authors prove that the geometrical objects of $F_\varepsilon$ such as several connections of the Finsler geometry and the flag curvatures converge uniformly to the corresponding objects of $F$ on compact subsets when $\varepsilon \rightarrow 0$ (see Section 8 of \cite{Fukuoka-Setti}). 
	In this work, the mollifier smoothing on $F_e$ is similar to the vertical mollifier smoothing defined in \cite{Fukuoka-Setti}.
	The only difference is that here we suppose that the asymmetric norm $F_e$ is strongly convex, and the smoothing is a Minkowski norm without the additional term. 
	On the other hand, the horizontal part of the mollifier smoothing defined in \cite{Fukuoka-Setti} is replaced by an extension of the smoothing of $F_e$ through the Lie group by left-invariance. 
	Considering these differences, it is natural to ask whether the convergence of geometrical objects that hold in the previous work are also valid here. 
The answer is positive, but due to the length of this work, we will leave the proof for a future work.

	Besides Theorems \ref{theorem_extremal de Pontryagin em grupos} and \ref{theorem_principal no fibrado cotangente}, the main contribution of this manuscript is to show that the mollifier smoothing of $C^0$-Finsler structures and the uniform convergence of Pontryagin extremals works extremely well together, especially because the smoothing occurs on the tangent bundle of $G$ and the convergence occurs on the cotangent bundle of $G$.
	This is an evidence that the mollifier smoothing and the extended geodesic field are important objects in studying this kind of geometry. 
	We also emphasize the importance of developing the strong convexity hypothesis properly in order to make it useful for this kind of transition between the cotangent and the tangent bundle.
	
	Now we explain how this work is organized.
	In Section \ref{section_preliminaries} we present the prerequisites necessary for this work.
	Although most of the subjects are known in the literature, some few results are new according to our best knowledge, and they are indicated along the section.
	In Section \ref{Strongly convex functions and strongly convex asymmetric norms}, we study strong convexity of asymmetric norms.
	We give three equivalent characterizations of strong convexity in the general case and one additional characterizations for the smooth case (Minkowski norms).
	We give special attention to the modulus of convexity, which will be important afterward.  
	In Section \ref{section_Vertical part of a Pontryagin extremal}, we present the Lie algebra level version $\tilde{\mathcal{E}}$ of  extended geodesic fields when $M=G$ is a Lie group.
	We also present its version for strongly convex asymmetric norm $F_e$, which is given by $\mathfrak{a} \in \mathfrak{g}^\ast \backslash 0 \mapsto \mathfrak{a}([u(\mathfrak{a}),\cdot ]) \in \mathfrak{g}^\ast$, where $\mathfrak{g}^\ast$ is the dual space of $\mathfrak{g}$ and $u$ maps $\mathfrak{a}$ to the unique vector on the unit sphere $S_{F_e}[0,1]$ of $\mathfrak{g}$ that maximizes $\mathfrak{a}$.
	As a consequence, we prove that the maximal Pontryagin extremals on Lie groups endowed with left-invariant strongly convex $C^0$-Finsler structures are defined in $\mathbb{R}$ (see Theorem \ref{theorem_extremal de Pontryagin em grupos}). 
	In Section \ref{Smoothing of a strongly convex asymmetric norm} we define the mollifier smoothing $F_\varepsilon$ on left-invariant strongly convex $C^0$-Finsler structures $F$. 
	It is a simplified version of the vertical smoothing presented in \cite{Fukuoka-Setti} in terms of calculations.
	The main result states that $F$ and $F_\varepsilon$ are strongly convex with respect to the same norm if $\varepsilon$ is sufficiently small (see Theorem \ref{F_epsilon2 til são fortemente convexas com a mesma constante}).
	In Section \ref{Smoothing F} we prove that $F_\varepsilon$ converges uniformly to $F$ on compact subset of $TG$ when $\varepsilon$ converges to zero.
	In Section \ref{section_convergences and estimates u and E}, we prove some estimates for $\tilde{\mathcal{E}}$ and $u: \mathfrak{g}^\ast \backslash 0 \rightarrow S_{F_e}[0,1]$ of $(G,F)$.
	We also prove that if $\tilde{\mathcal{E}}_\varepsilon$ and $u_\varepsilon$ are the corresponding maps of $(G,F_\varepsilon)$, then they converge uniformly to $\tilde{\mathcal{E}}$ and $u$ respectively on compact subsets of $\mathfrak{g}^\ast \backslash 0$.
Finally, in Section \ref{a epsilon convege para a e x epsilon converge para x}, we prove that given $(x_0, \alpha_0) \in T^\ast G \backslash 0$, if $(x(t), \alpha(t))$ and $(x_\varepsilon(t), \alpha_\varepsilon(t))$ are Pontryagin extremals of $(G,F)$ and $(G,F_\varepsilon)$ respectively such that $(x(0), \alpha(0)) = (x_\varepsilon(0), \alpha_\varepsilon(0)) = (x_0, \alpha_0)$, then $(x_\varepsilon(t), \alpha_\varepsilon(t))$ converge uniformly to $(x(t), \alpha(t))$ on compact subsets $[a,b]$ of $\mathbb{R}$ as $\varepsilon \rightarrow 0$ (see Theorem \ref{theorem_principal no fibrado cotangente}).
A version of this theorem on $\mathfrak{g}\backslash \{0\} \times \mathfrak{g}^\ast \backslash \{0\}$ is also proven (see Theorem \ref{theorem_convergencia uniforme nas algebras}).
	
	\begin{obs}
	The Einstein summation convention is used throughout this work.
	\end{obs}

\section{Preliminaries}
\label{section_preliminaries}

In this section, we present the preliminaries that are necessary for the development of this work. 
Throughout this manuscript, $\mathbb{V}$ is an $n$-dimensional vector space over $\mathbb{R}$ endowed with an inner product $\left< \cdot, \cdot \right>$.
We denote its dual space by $\mathbb{V}^\ast$, and the isomorphism $v \mapsto \left< v, \cdot \right>$ between the vector spaces $\mathbb{V}$ and $\mathbb{V}^\ast$ induces an inner product $\left< \cdot, \cdot\right>_{\ast}$ on $\mathbb{V}^\ast$.
We denote the norms corresponding to $\left< \cdot, \cdot \right>$ and $\left< \cdot, \cdot \right>_{\ast}$ by $\Vert \cdot\Vert$ and $\Vert \cdot\Vert_\ast$ respectively.
A particular instance of this case is when $\mathbb{V}$ is the Lie algebra $\mathfrak{g}$ of a Lie group $G$, which will always be  endowed with an inner product $\left< \cdot, \cdot\right>$.

\subsection{Convex functions and asymmetric norms}
\label{subsection_asymmetric norms}

In this subsection we present the prerequisites on convex functions and asymmetric norms.
We define strongly convex asymmetric norms and we prove that Minkowski norms are strongly convex.
For asymmetric norms and convex functions, see \cites{Cobzas, Rockafellar, SunYuan}.
For strongly convex asymmetric norms, see \cite{Fukuoka-Rodrigues}.
Finally, for Minkowski norms, see \cite{BaoChernShen}.

\begin{defi} \label{Fortemente convexa - definição clássica} Let $U$ be a non-empty convex subset of $\mathbb{V}$ and $f: U \rightarrow \mathbb{R}$ be a function.
	\begin{enumerate}
	\item We say that $f$ is {\rm convex} if
		\begin{eqnarray*}
			f(ty+(1-t)z) \leq tf(y)+(1-t)f(z),
		\end{eqnarray*}
		for all $t \in [0,1]$ and any $y, z \in U.$
	\item We say that $f$ is {\rm strictly convex} if
	\begin{eqnarray*}
		f(ty+(1-t)z) < tf(y)+(1-t)f(z),
	\end{eqnarray*}
	for all $t \in (0,1)$ and any $y, z \in U$ with $y \neq z.$
	\item We say that $f$ is {\rm strongly} (or {\rm uniformly}) {\rm convex with modulus} $\gamma,$ when there is a constant $\gamma>0$ where for any $y, z \in U$ and all $t \in [0,1],$ we have
	\begin{equation*}
		f(ty+(1-t)z) \leq tf(y)+(1-t)f(z)-\frac{1}{2}\gamma t(1-t) \Vert y-z \Vert^2.
	\end{equation*}
	\end{enumerate}
\end{defi}

\begin{defi}
Consider $f: S \rightarrow \mathbb{R}$, where $S$ is a non-empty subset of $\mathbb{V}$.
The {\em epigraph} $\epi f$ of $f$ is a subset of $\mathbb{V} \times \mathbb{R}$ defined by
\[
\epi f=\{(x,r)\in \mathbb{V} \times \mathbb{R}: \,\, f(x) \leq r,\,\, x\in S,\,\, r\in \mathbb{R}\}.
\]
\end{defi}

\begin{obs}
If $f:X \rightarrow Y$ is a function and $y\in  Y$, we denote the inverse image $f^{-1}(y)$ by $\left\{ f=y \right\}$.
\end{obs}

\begin{teo}
\label{teo_convexo e epigrafo}
Let $S\subset \mathbb{V}$ be a non-empty convex set and $f:S \rightarrow \mathbb{R}$.
Then $f$ is a convex function iff $\epi f$ is a convex subset of $\mathbb{V} \times \mathbb{R}$. 
In particular, if $r\in \mathbb{R}$ and $f$ is a convex function, then 
\begin{align}
\nonumber
\epi f \cap (\mathbb{V} \times \{r\}) & =\{(x,r) \in \mathbb{V} \times \mathbb{R}:f(x)\leq r, x\in S\} \\ 
& = \{x \in \mathbb{V}:f(x) \leq r, x \in S\} \times \{r\} \nonumber \\
&\cong \{x \in \mathbb{V}:f(x) \leq r, x \in S\}\nonumber 
\end{align}
are convex sets.
\end{teo}

\begin{teo}
\label{teo_composicao de convexas}
Let $f:\mathbb{V} \rightarrow \mathbb{R}$ be a convex function and $\varphi: f(\mathbb{V}) \rightarrow \mathbb{R}$ be a non-decreasing convex function. 
Then $\varphi \circ f$ is a convex function.
\end{teo}

\begin{defi} \label{Definicao de norma assimetrica} An {\rm asymmetric norm} on $\mathbb{V}$ is a function $F:\mathbb{V} \rightarrow [0,\infty)$ that satisfies the following conditions:
	\begin{enumerate}
		\item $F(y) = 0$ only if $y = 0;$
		\item $F(\mu y) = \mu F(y)$ for every $\mu \in [0, \infty)$ and $y \in \mathbb{V}$ (positive homogeneity);
		\item $F(y + z) \leq F(y) + F(z)$ for every $y, z \in \mathbb{V}$ (triangle inequality).
	\end{enumerate}
\end{defi}

It is straightforward that every norm is an asymmetric norm.

\begin{pro} \label{F e Lipschitz} 
An asymmetric norm $F:(\mathbb{V},\Vert \cdot\Vert) \rightarrow [0,\infty)$ is a Lipschitz map. 
\end{pro}

\begin{proof}
Let $\mathcal{B} = \{e_1, \ldots, e_n\}$ be a basis of $\mathbb{V}$ and $\Vert \cdot \Vert_1$ be the sum norm given by $\Vert y^i e_i\Vert_1 = \sum_{i=1}^n \vert y^i \vert$. 
Set $\tilde{C} = \max_{i=1, \ldots, n} \{F(e_i),F(-e_i)\}$.
Then every $y \in \mathbb{V}$ is written as $y^i \tilde{e}_i$, where $\tilde{e}_i \in \{e_i, -e_i\}$ and $y^i \geq 0$, and we have that $F(y) \leq \tilde{C} \Vert y\Vert_1$ due to the positive homogeneity and triangle inequality of $F$.
Finally, the inequality $F(y) \leq \tilde{C}\Vert y\Vert_1$ and the equivalence of the norms $\Vert \cdot \Vert$ and $\Vert \cdot \Vert_1$ implies that there exists $C>0$ such that $\vert F(y) - F(z)\vert \leq C \Vert y-z\Vert$ for every $y,z\in \mathbb{V}$, what settles the proof. 
\end{proof}

\begin{defi}
	Let $F:\mathbb V \rightarrow [0,\infty)$ be an asymmetric norm.
	The {\em open ball} centered at $y$ and radius $r$ is the subset
	\[
	B_F(y,r) = \{z \in \mathbb V : F(z-y) < r\}.
	\]
	The {\em closed ball} centered at $y$ and radius $r$ is the subset
	\[
	B_F[y,r] = \{z \in \mathbb V : F(z-y) \leq r\}.
	\]
	The {\em sphere} centered at $y$ and radius $r$ is the subset
	\[
	S_F[y,r] = \{z \in \mathbb V : F(z-y) = r\}.
	\]
\end{defi}

\begin{obs} \label{remark_asymmetric norm equivalent} Let $F,\hat{F}:\mathbb{V} \rightarrow [0,\infty)$ be asymmetric norms. Then $F$ and $\hat{F}$ are {\em equivalent asymmetric norms}, that is, there exist positive constants $c_1$ and $c_2$ such that 
	\begin{eqnarray}
	\label{eqnarray_normas assimetricas equivalentes}
		c_1\hat{F}(y) \leq F(y) \leq c_2\hat{F}(y),
	\end{eqnarray}
for every $y \in \mathbb{V}$.
Indeed, $F$ and $\hat{F}$ are continuous by  Proposition \ref{F e Lipschitz}, and it is enough to set $c_1 = \min\limits_{y \in S_{\Vert \cdot\Vert}[0,1]} \frac{F(y)}{\hat{F}(y)}$ and $c_2 = \max\limits_{y \in S_{\Vert \cdot\Vert}[0,1]} \frac{F(y)}{\hat{F}(y)}$ due to the positive homogeneity of $F$ and $\hat{F}$.
\end{obs}

We now introduce the definition of subdifferential and subgradient (see \cite{Rockafellar}).
We will restrict and adapt the theory for our necessities.
Let $f: \mathbb{V} \rightarrow \mathbb{R}$ be a convex function. 
A vector $\alpha_{y}{}^{\sharp}$ is said to be a {\it subgradient of $f$ at} $y \in \mathbb{V}$ if
\begin{eqnarray} \label{desigualdade subgradiente 1}
	f(z) \geq f(y) + \langle \alpha_{y}{}^{\sharp}, z-y \rangle, \quad \forall \, z \in \mathbb{V}.
\end{eqnarray}
The set of all subgradients of $f$ at $y$ is called the {\it subdifferential of $f$ at $y$} and it is denoted by $\partial f(y).$ $\partial f (y)$ is also called the set of functional supports of $f$ at $y$.
Since $f(y)\in \mathbb{R}$ for every $y\in \mathbb{V}$, we have that $\partial f(y)$ is a non-empty set for every $y \in \mathbb{R}$. 

Using the Riesz-Fr\'echet Representation Theorem we can identify the subgradient $\alpha_{y}{}^{\sharp}$ of $f$ at $y$ with the linear functional $\alpha_{y} = \langle \alpha_{y}{}^{\sharp}, \cdot \rangle.$ Thus (\ref{desigualdade subgradiente 1}) can be rewritten as
\begin{eqnarray} \label{desigualdade subgradiente 2}
	f(z) \geq f(y) + \alpha_{y}(z-y), \quad \forall \, z \in \mathbb{V}.
\end{eqnarray}
From now on, the elements of $\partial f(y)$ will be linear functionals $\alpha_y$ in $\mathbb{V}^\ast$ (compare to \cite{Fukuoka-Rodrigues}).
In particular, if $f$ is differentiable at $y$, then $\partial f(y) = \{df_y\}$.

Strict convexity of asymmetric norms is defined analogously to the strict convexity of norms.

\begin{defi}
We say that an asymmetric norm $F$ is strictly convex if 
\begin{align}
F((1-t)y+tz) < (1-t)F(y) + t F(z)
\label{align_F estritamente convexo na esfera}
\end{align} 
for every $y,z \in S_F[0,1]$, $y\neq z$, and $t\in (0,1)$.
\end{defi}

\begin{defi} \label{Fortemente convexa - definicao artigo R-H} Let $\check{F}$ be an asymmetric norm on $\mathbb{V}.$ We say that an asymmetric norm $F$ is {\rm strongly convex with respect to} $\check{F}$ if
	\begin{equation*}
		F^2(z) \geq F^2(y) + \alpha(z-y) + \check{F}^2(z-y)
	\end{equation*}
for every $y,z \in \mathbb{V}$ and $\alpha \in \partial F^2(y)$.
\end{defi}

\begin{obs}
Frequently the asymmetric norm $\check{F}$ of Definition \ref{Fortemente convexa - definicao artigo R-H} isn't important because all asymmetric norms are equivalent (see \eqref{eqnarray_normas assimetricas equivalentes}) and given any other asymmetric norm $\hat{F}$ on $\mathbb{V}$, $F$ will be  
strongly convex with respect to a multiple of $\hat{F}$.
\end{obs}

The relation of Definition \ref{Fortemente convexa - definicao artigo R-H} with the classical one (see Definition \ref{Fortemente convexa - definição clássica}) will be presented in Theorem \ref{Fortemente convexa - Equivalencia}.

Now we present a smooth version of strongly convex asymmetric norms (see Proposition \ref{proposition_norma limitante inferior}).

\begin{defi}
	\label{norma de Minkowski}
	A function $F:\mathbb V \rightarrow [0, \infty)$ is a {\rm Minkowski norm} if 
	\begin{enumerate}
		\item $F$ is smooth in $\mathbb V\backslash\{0\}$;
		\item $F(\mu y) = \mu F(y)$ for every $\mu \in [0,\infty)$ and $y\in \mathbb V$;
		\item If $(y^1, \ldots, y^n)$ is a coordinate system of $\mathbb V$ with respect to a basis of $\mathbb V$, then the  $n\times n$ hessian matrix
		\[
		\left( g_{ij(y)} \right) := \left( \left[\frac{1}{2} F^2 (y) \right]_{y^iy^j} \right), \hspace{5mm} i,j=1, \ldots, n,
		\] 
		is positive definite for every $y \in \mathbb V\backslash \{0\}$, where the subscript $y^iy^j$ stands for the partial derivatives with respect to $y^i$ and $y^j$.
	\end{enumerate}
\end{defi}

\begin{pro}
\label{pro_Minkowski Euler}
Let $F$ be a Minkowski norm and $y \in \mathbb{V} \backslash \{0\}$.
Then
\begin{enumerate}
\item $y^i F_{y_i}(y) = F(y)$;
\item $y^j F_{y_iy_j}(y) = 0$.
\end{enumerate}
\end{pro}

\begin{pro}
Every Minkowski norm is an asymmetric norm. 
\end{pro}

\subsection{The maximizer $u(\alpha)$ of $\alpha \in \mathbb{V}^\ast\backslash \{0\}$ on $S_F[0,1]$}

\label{subsection_inverse of Legendre transform}

In this subsection, $\mathbb{V}$ is endowed with a strictly convex asymmetric norm $F$. 
The strict convexity of $F$ implies that for every $\alpha \in \mathbb{V}^{\ast}\backslash \{0\}$, there exist a unique $u(\alpha) \in S_F[0,1]$ that maximizes $\alpha $ on $S_F[0,1]$.
In this subsection we study the map $\alpha \mapsto u(\alpha)$.
The references for this subsection are \cites{Rockafellar, Fukuoka-Rodrigues}.

\begin{defi} Let $F$ be an asymmetric norm on $\mathbb{V}.$ The {\rm dual asymmetric norm} $F_\ast: \mathbb{V}^\ast \rightarrow [0,\infty)$ of $F$ is defined by
	\begin{eqnarray*}
		F_\ast(\alpha) = \sup_{y \in S_F[0, 1]} \alpha(y).
	\end{eqnarray*}
\end{defi}

\begin{pro}
If $F$ is a strictly convex asymmetric norm on $\mathbb{V}$, then $F_{\ast}{}^2$ is differentiable. 
\end{pro}

\begin{lem} \label{Formula F}  Let $F:\mathbb{V} \rightarrow  [0,\infty)$ be a strictly convex asymmetric norm. 
Then $F(dF_\ast^{\ 2}(\alpha)) = 2F_\ast (\alpha)$ for every $\alpha \in \mathbb{V}^\ast.$
\end{lem}

\begin{pro} \label{Formula u alpha}  If $F:\mathbb{V} \rightarrow [0,\infty)$ is a strictly convex asymmetric norm and $\alpha \in \mathbb{V}^\ast,$ then
	\begin{eqnarray*}
		u(\alpha)= \frac{dF_\ast^{\ 2}(\alpha)}{F(dF_\ast^{\ 2}(\alpha))} = \frac{dF_\ast^{\ 2}(\alpha)}{2F_\ast(\alpha)}
	\end{eqnarray*}
	is the unique point in $S_F[0,1]$ that maximizes $\alpha$ in $S_{F}[0,1],$ where $dF_\ast^{\ 2}: \mathbb{V}^\ast \rightarrow \mathbb{V}$ is identified with the differential $dF_\ast^{\ 2}: \mathbb{V}^\ast\times \mathbb{V}^\ast\rightarrow \mathbb{R}.$
\end{pro}

\begin{teo}
\label{theorem_Lipschitz constant df2ast}
If $F$ is a strongly convex asymmetric norm on $\mathbb{V}$ with respect to $\sqrt{\frac{\gamma}{2}}\Vert \cdot\Vert$, then the mapping $dF_\ast{}^2: \mathbb{V}^\ast \rightarrow \mathbb{V}$ is Lipschitz, with Lipschitz constant $\frac{4}{\gamma}$.
\end{teo}


Now we prove that $\alpha \mapsto u(\alpha)$, when restricted to compact subsets of $\mathbb{V}^\ast\backslash \{0\}$, is Lipschitz.
This result wasn't found elsewhere, and its proof is included here for convenience. 

\begin{lem} \label{u_alpha é lipschitz} Let $F:\mathbb{V} \rightarrow [0,\infty)$ be a strongly convex asymmetric norm with respect to $\sqrt{\frac{\gamma}{2}} \Vert \cdot \Vert.$ The application that associates every $\alpha \in \mathbb{V}^\ast \setminus \{0\}$ to the vector $u(\alpha) \in \mathbb{V}$ is Lipschitz over compact subset of $\mathbb{V}^\ast \setminus \{0\}.$
\end{lem}

\begin{proof} 
	Let $K_{\mathbb{V}^\ast \setminus \{0\}}$ a compact subset of $\mathbb{V}^\ast \setminus \{0\}.$ We must show that there exists a constant $\mathcal{K}>0,$ such that
	\begin{eqnarray*}
		\Vert u(\alpha) - u(\beta) \Vert \leq \mathcal{K} \Vert \alpha - \beta \Vert_\ast, \, \forall \, \alpha, \beta \in K_{\mathbb{V}^\ast \setminus \{0\}}.
	\end{eqnarray*}
	
	Let
	\begin{eqnarray}
		C_1 & = & \max\left\{ \frac{1}{F_\ast(\alpha)} : \alpha \in K_{\mathbb{V}^\ast \setminus \{0\}} \right\} > 0, \label{Constante C1} \\
		C_2 & = & \max\left\{ F_\ast(\alpha) : \alpha \in K_{\mathbb{V}^\ast \setminus \{0\}} \right\} > 0, \label{Constante C2} \\
		C_3 & = & \max\left\{\big\Vert dF_\ast^{\ 2}(\alpha) \big\Vert : \alpha \in K_{\mathbb{V}^\ast \setminus \{0\}} \right\} > 0 \label{Constante C3}, \\
		L_1 & = & \max\left\{ F_\ast(\alpha) : \alpha \in S_{\Vert \cdot \Vert_\ast}[0,1]  \right\}. \label{Constante L1}
	\end{eqnarray}
	Note that $L_1$ is a Lipschitz constant of $F_\ast$. 
	Then, given $\alpha, \beta \in K_{\mathbb{V}^\ast \setminus \{0\}},$ we have
	\begin{eqnarray*}
		\Vert u(\alpha) - u(\beta)  \Vert & = & \Bigg\Vert  \frac{dF_\ast^{\ 2}(\alpha)}{2F_\ast(\alpha)} -  \frac{dF_\ast^{\ 2}(\beta)}{2F_\ast(\beta)}  \Bigg\Vert  \\
		& = & \frac{1}{2}\Bigg\Vert  \frac{F_\ast(\beta)dF_\ast^{\ 2}(\alpha)-F_\ast(\alpha)dF_\ast^{\ 2}(\beta)}{F_\ast(\alpha)F_\ast(\beta)} \Bigg\Vert  \\
		& \leq & \frac{1}{2}C_1^2 \big\Vert  F_\ast(\beta)dF_\ast^{\ 2}(\alpha)-F_\ast(\alpha)dF_\ast^{\ 2}(\beta) \big\Vert  \\
		& \leq & \frac{1}{2}C_1^2 \big\Vert  F_\ast(\beta)dF_\ast^{\ 2}(\alpha)-F_\ast(\alpha)dF_\ast^{\ 2}(\alpha) \big\Vert \\
		& + & \frac{1}{2}C_1^2 \big\Vert F_\ast(\alpha)dF_\ast^{\ 2}(\alpha)-F_\ast(\alpha)dF_\ast^{\ 2}(\beta) \big\Vert \\
		& = & \frac{1}{2}C_1^2  \big\Vert dF_\ast^{\ 2}(\alpha) \big\Vert \big\vert   F_\ast(\alpha)-F_\ast(\beta) \big\vert \\  
		& + & \frac{1}{2}C_1^2  F_\ast(\alpha) \big\Vert dF_\ast^{\ 2}(\alpha)-dF_\ast^{\ 2}(\beta) \big\Vert \\
		& \leq & \frac{1}{2}C_1^2 C_3   \vert F_\ast(\alpha)-F_\ast(\beta) \vert \\  
		& + & \frac{1}{2}C_1^2 C_2 \Vert dF_\ast^{\ 2}(\alpha)-dF_\ast^{\ 2}(\beta) \Vert \\
		& \leq & \frac{1}{2} \left(C_1^2 C_3 L_1 + C_1^2 C_2 \frac{4}{\gamma}\right) \Vert \alpha-\beta \Vert_\ast,
	\end{eqnarray*}
where the last inequality is due to Theorem \ref{theorem_Lipschitz constant df2ast}.	Setting
	\begin{eqnarray*}
		\mathcal{K} = \frac{1}{2} \left(C_1^2 C_3 L_1 + C_1^2 C_2 \frac{4}{\gamma}\right),
	\end{eqnarray*} 
	we have that 
	\begin{eqnarray}
	\label{eqnarray_u Lipschitz}
		\Vert u(\alpha) - u(\beta)  \Vert & \leq & \mathcal{K} \Vert \alpha-\beta \Vert_\ast,
	\end{eqnarray}
	for every $\alpha, \beta \in K_{{\mathbb{V}^\ast \setminus \{0\}}},$ what settles the proof.
\end{proof}


\subsection{Almost radial growth rates of $F^2$  for asymmetric norms $F$ on $\mathbb{V}$}
\label{subsection_almost radial}

In Lemma 4 of \cite{Fukuoka-Setti}, we have already calculated the almost radial growth rates of $F$.
In this subsection we obtain a slight variation of that result.
Consider an asymmetric norm $F$ on $\mathbb{V}$.

Define
\begin{equation} 
	\label{constante r}
	r_M  =  \max_{y \in S_{\|\cdot\|}[0,1]} F(y) = \max_{y \in B_{\|\cdot\|}[0,1]} F(y)
\end{equation}
and
\begin{equation} 
	\label{constante rho}
	r_m  =  \min_{y \in S_{\|\cdot\|}[0,1]} F(y).
\end{equation}
Then
\begin{equation*} 
	\max_{y \in S_{\|\cdot\|}[0,1]} F(y) = \min_{y \in S_{\|\cdot\|}\left[ 0,\frac{r_M}{r_m} \right]} F(y).
\end{equation*}

Let
\begin{equation}
	\label{Conjunto A}
\mathcal A := \left\{y \in \mathbb{R}^n: \|y\| \in \left[ \frac{1}{2}, \frac{2r_M}{r_m} \right] \right\}.
\end{equation}


\begin{lem} \label{proposicao F em Rn}
	Let $F$ be an asymmetric norm on ${\mathbb R^n}$ and $v \in \mathcal A.$ Set
	\[
	\varepsilon \in \left(0,\frac{r_m \Vert v\Vert}{4nr_M}\right],
	\]
	where $r_M$ and $r_m$ are defined by (\ref{constante r}) and (\ref{constante rho}) respectively. Then
	\[
	\frac{r_m^2}{32}t < F^2\left(w+t\frac{v}{\|v\|}\right)-F^2(w) < \left( \frac{2r_mr_M^3+12r_M^4}{r_m^2}\right)t  
	\]
	for every $w\in B_{\|\cdot\|}[v,\varepsilon]$ and $t \in (0,1)$, and the first inequality holds also for $t > 0$.
\end{lem}

\begin{proof}
Let us prove the first inequality. In Lemma 4 of \cite{Fukuoka-Setti}, we proved the inequality
	\begin{align}
	\label{align_desigualdade quase radial trabalho anterior}
	F\left(w+t\frac{v}{\|v\|}\right)-F(w) > \frac{r_m}{8}t
	\end{align}
under the same hypotheses, with the exception that in that result we supposed $v\neq 0$ instead of $v\in \mathcal{A}$.  
Then

\begin{eqnarray}
& & F^2\left(w+t\frac{v}{\Vert v\Vert} \right) - F^2(w) \nonumber \\ 
= & & \left( F\left(w+t\frac{v}{\Vert v\Vert} \right) + F(w) \right) \left( F\left(w+t\frac{v}{\Vert v\Vert} \right) - F(w) \right) \nonumber \\
\geq & & F(w) \left( F\left(w+t\frac{v}{\Vert v\Vert} \right)-F(w)\right) \nonumber \\
= & & \Vert w\Vert \frac{F(w)}{\Vert w\Vert}  \left( F\left(w+t\frac{v}{\Vert v\Vert} \right)-F(w)\right) \nonumber \\
> & & \frac{1}{2}\Vert v\Vert r_m\left( \frac{r_m}{8}t \right)\geq \frac{r_m^2}{32}t, \nonumber
\end{eqnarray}
where the strict inequality is due to $\Vert w\Vert \geq \frac{1}{2}\Vert v\Vert$, \eqref{constante rho} and \eqref{align_desigualdade quase radial trabalho anterior}.

Now we prove the second inequality.
Fix $v \in \mathcal{A}.$ Given $w\in B_{\|\cdot\|}[v,\varepsilon],$ we have that
	\begin{eqnarray*}
		\Vert w \Vert - \Vert v \Vert \leq \Vert w-v \Vert \leq \varepsilon \leq   \frac{r_m \Vert v\Vert}{4nr_M}.
	\end{eqnarray*} 
	Using $\Vert v \Vert \leq \frac{2r_M}{r_m},$ we get
	\begin{eqnarray*}
		\Vert w \Vert &\leq& \frac{r_m+4r_M}{2r_m}.
	\end{eqnarray*}
	Then
	\begin{eqnarray*}
		F^2\left(w+tv\right)-F^2(w) & \leq & F^2(w)+ 2tF(w)F(v) + t^2F^2(v)-F^2(w) \\
		& = & 2tF(w)F(v) + t^2F^2(v) \\
		& < & 2tF(w)F(v) + tF^2(v) \\
		& = & t\left(2F(w)F(v) + F^2(v) \right) \\
		& \leq & t \left[2 \cdot \frac{r_mr_M+4r_M^2}{2r_m} \cdot \frac{2r_M^2}{r_m} + \left(\frac{2r_M^2}{r_m}\right)^2 \right] \\
		& = & t \left( \frac{2r_mr_M^3+12r_M^4}{r_m^2}\right),
	\end{eqnarray*}
	since $t \in (0,1).$
	Therefore,
	 \begin{eqnarray*}
	 	F^2\left(w+tv\right)-F^2(w) &  < &  \left( \frac{2r_mr_M^3+12r_M^4}{r_m^2}\right)t 
	 \end{eqnarray*}
	 for every $v\in \mathcal{A},$ $w\in B_{\|\cdot\|}[v,\varepsilon]$ and $t\in (0,1).$ 
\end{proof}

\subsection{The extended geodesic field}
\label{Campo geodesico estendido}

In this subsection we define Pontryagin type $C^0$-Finsler manifolds and extended geodesic fields 
(see Sections 3 and 4 of \cite{Fukuoka-Rodrigues}).
In short, a Pontryagin type $C^0$-Finsler manifold has its structure determined by a control system of unit continuously differentiable vector fields and the extended geodesic field are determined applying the PMP in this control system. 


\begin{defi} \label{Definicao de estrura de Finsler de classe C0} A {\rm $C^0$-Finsler structure} of $M$ is a continuous function $F: TM \rightarrow	[0,\infty)$ such that its restriction to each tangent space is an asymmetric norm. 
The pair $(M, F)$ is a {\rm $C^0$-Finsler manifold}.
If the restriction of $F$ to every tangent space of $M$ is a strongly convex asymmetric norm, then $F$ is called a {\rm strongly convex $C^0$-Finsler structure}.
\end{defi}

We denote the open ball, the closed ball and the sphere in $(T_xM,F\vert_{T_xM})$ with center at $a \in T_xM$ and radius $r>0$ by $B_F(x,a,r)$, $B_F[x,a,r]$ and $S_F[x,a,r]$ respectively.

\begin{defi} \label{Estrutura de Finsler}
	A {\rm Finsler structure} of $M$ is a function $F:TM \rightarrow [0, \infty)$ 
	with the following properties:
	\begin{enumerate}
		\item $F$ is smooth on the {\em slit tangent bundle}  $TM\setminus 0 = \{(x,y) \in TM \, : \, y \neq 0\}.$
		\item For every $x \in M,$ $F(x, \cdot): T_xM \rightarrow [0, \infty)$ is a Minkowski norm.
	\end{enumerate}  
The pair $(M,F)$ is a {\em Finsler manifold}.
\end{defi}

\begin{obs}
Since every Minkowski norm is an asymmetric norm, we have that a Finsler structure is a $C^0$-Finsler structure.
\end{obs}

\begin{obs} 
A $C^0$-Finsler structure $F:TM \rightarrow  [0,\infty)$ induces the fiberwise dual $C^0$-Finsler structure $F_\ast: T^\ast M \rightarrow  [0,\infty)$, that is, $F_\ast(x,\cdot):T^\ast_xM \rightarrow  [0,\infty)$ is defined as 
\[
F_\ast(x,\alpha)=\max_{y\in S_F[x,0,1]}\alpha(y).
\]
\end{obs}


\begin{defi} \label{Variedade de Finsler do Tipo Pontryagin} A $C^0$-Finsler manifold $(M, F)$ is of {\rm Pontryagin type at} $p\in M$ if there exist a neighborhood $U$ of $p,$ a coordinate system $\phi = (x^1, \ldots, x^n):U  \rightarrow \mathbb{R}^n,$ with	the respective natural coordinate system $d\phi = (x^1, \ldots, x^n, y^1, \ldots, y^n): TU \rightarrow \mathbb{R}^{2n}$ on the tangent bundle, and a family of $C^1$ unit vector fields
	\begin{eqnarray*}
		\left\{ x \mapsto X_u(x) = (y^1(x^1, \ldots x^n, u), \ldots, y^n(x^1, \ldots, x^n, u)): u \in S^{n-1} \right\}
	\end{eqnarray*}
	on $U$ parametrized by $u \in S^{n-1}$ such that
	\begin{enumerate}
		\item $u \mapsto X_u(x)$ is a homeomorphism from $S^{n-1} \subset \mathbb{R}^n$ onto $S_F[x,0,1]$ for every $x\in U;$
		\item $(x, u) \mapsto (y^1(x^1, \ldots, x^n, u), \ldots,  y^n(x^1, \ldots, x^n, u))$ is continuous;
		\item $(x, u) \mapsto \left( \frac{\partial y^1}{\partial x^i}(x^1, \ldots, x^n, u), \ldots, \frac{\partial y^n}{\partial x^i}(x^1, \ldots, x^n, u) \right)$
		is continuous for every $i = 1, \ldots,n.$		
	\end{enumerate}
	We say that $F$ is of {\rm Pontryagin type} on $M$ if it is of Pontryagin type for every $p \in M$. The pair $(M, F)$ is called {\em Pontryagin type $C^0$-Finsler manifold}.
	
\end{defi}

\begin{obs}
\label{remark_tipo Pontryagin nao depende de sistema de coordenadas}
Definition \ref{Variedade de Finsler do Tipo Pontryagin} does not depend on the choice of the coordinate system (see Remark 3.2 of \cite{Fukuoka-Rodrigues}).
\end{obs}

Now we present the definition of extended geodesic field $\mathcal{E}$ on $T^\ast U$ of a Pontryagin type $C^0$-Finsler manifold.

The {\em Hamiltonian} $H_u: T^\ast U \rightarrow \mathbb{R}$ {\em corresponding to the control $u$} is defined by $H_u(x,\alpha) = \alpha(X_u(x))$.
Denote the canonical symplectic form on $T^\ast U$ by $\omega$, which is the exterior derivative of the tautological $1$-form on $T^\ast U$.
Let $\vec{H}_u$ be a vector field on $T^\ast U$ given by $d(H_u)(x,\alpha) = \omega(\cdot, \vec{H}_u(x,\alpha))$.
For each $(x,\alpha) \in T^\ast U \backslash 0$, define 
\[
\mathcal{C}(x, \alpha) = \left\{ u \in S^{n-1}: \alpha(X_u(x)) = \max\limits_{v\in S^{n-1}}\alpha(X_v(x))\right\}.
\]
The {\em extended geodesic field} is the multivalued vector field defined on $T^\ast U \backslash 0$ given by
\begin{equation*}
\mathcal{E}(x,\alpha) = \{\vec{H}_u(x,\alpha): u \in \mathcal{C}(x,\alpha)\}.
\end{equation*}

Represent $X_u$ in a coordinate system $(x^1, \ldots, x^n)$ of $U$ by 
$$X_u(x) = f^i(x, u) \frac{\partial}{\partial x^i} = (f^1(x, u), \ldots, f^n(x, u)).$$
In the coordinate system $(x^1,\ldots, x^n,\alpha_1, \ldots, \alpha_n)$ of $T^\ast U$ corresponding to $(x^1, \ldots, x^n)$, $\mathcal{E}$ is given by 

\begin{eqnarray*}
	\mathcal{E}(x, \alpha) = \left\{ f^i(x,u) \frac{\partial}{\partial x^i} - \alpha_j \frac{\partial f^j}{\partial x^i}(x,u) \frac{\partial}{\partial \alpha_i}: u \in \mathcal{C}(x, \alpha) \right\}.
\end{eqnarray*}
The extended geodesic field is obtained as a consequence of PMP applied to the control system given by Definition \ref{Variedade de Finsler do Tipo Pontryagin}.
An absolutely continuous curve  $\gamma: [a, b] \rightarrow T^\ast U$ such that $\gamma'(t) \in \mathcal{E}(\gamma(t))$ for almost every $t \in [a, b]$ is called a {\it Pontryagin extremal} of $(U, F).$
If $\gamma(t) = (x(t),\alpha(t))$ is a Pontryagin extremal of $(U,F)$, then
\begin{eqnarray} \label{É constante pelo PMP}
	\mathcal{M}(x(t), \alpha(t)) := \max_{u \in S^{n-1}}	\alpha(t)\big(X_u(x(t))\big) = F_\ast(x(t), \alpha(t))
\end{eqnarray}
is constant.
We also have that if $x(t)$ is a length minimizer on $U,$ then there exists $\alpha(t)$ such that $(x(t), \alpha(t))$ is a Pontryagin extremal of $(U, F)$.

\begin{obs}
Suppose that we have a control system on an open subset $U$ of $M$ (not necessarily a coordinate open subset) such that every $u \in S^{n-1}$ corresponds to a continuously differentiable vector field $X_u$ on $U$.
Suppose that Item (1) of Definition \ref{Variedade de Finsler do Tipo Pontryagin} is satisfied for every $x \in U$.
In addition, suppose that for every $p\in U$, there exist a coordinate system on a neighborhood $U_p$ of $p$ such that Items (2) and (3) of Definition  \ref{Variedade de Finsler do Tipo Pontryagin} are satisfied on $U_p$.
Then Items (2) and (3) of Definition \ref{Variedade de Finsler do Tipo Pontryagin} are satisfied with respect to any coordinate system on $U$ due to 
Remark \ref{remark_tipo Pontryagin nao depende de sistema de coordenadas}, 
and $\mathcal{E}$ can be defined on $U$ independently of the choice of coordinates. 
Finally, we have that for any minimizing path  $x(t):[a,b] \rightarrow U$, there exist an absolutely continuous $\alpha(t)$ such that $(x(t), \alpha(t))$ is a Pontryagin extremal (see Corollary 12.1 of \cite{Sachkov}). 
\end{obs}


Now we present this theory for Lie groups (see \cite{PrudencioFukuoka}).
Let $G$ be a Lie group. We denote the neutral element of $G$ by $e$. Recall that its Lie algebra and its dual space are denoted by $\mathfrak{g}$ and $\mathfrak{g}^\ast$ respectively.   
For every $x \in G$ we denote the left translation $x^\prime \in G \mapsto x.x^\prime \in G$ by $L_x$. 
Let $F : TG \rightarrow  [0,\infty)$ be a left-invariant $C^0$-Finsler structure, that is,
\begin{eqnarray*}
	F(x\cdot x', (dL_x)_{x'}(y)) = F(x',y),
\end{eqnarray*}
for every $x, x' \in G$ and $y \in T_{x'}G.$ Denote $F(e, \cdot)$ por $F_e.$ 
{\em We have that Lie groups endowed with a left-invariant $C^0$-Finsler structure is of Pontryagin type, where $u \in S^{n-1}$ is replaced by $\mathfrak{u} \in S_{F_e}[0,1]$ and $X_{\mathfrak{u}}$ are left-invariant vector fields such that $X_{\mathfrak{u}}(e) = \mathfrak{u}$ (see \cite{PrudencioFukuoka} and Section 8 of \cite{Fukuoka-Rodrigues}).}

The next result, which proof can be found in Theorem 8 of \cite{PrudencioFukuoka}, shows that the differential inclusion $\gamma'(t) \in \mathcal{E}(\gamma(t))$ can be represented on $S_{F_e}[0,1]\times\mathfrak{g^\ast}.$


\begin{teo} \label{Teorema Jessica e Ryuichi} Let $(G, F)$ be a Lie group endowed with a left invariant $C^0$-Finsler structure $F.$ Let $t \in I \mapsto (x(t), \alpha(t))$ be a Pontryagin extremal of $(G, F)$ corresponding to the control $\mathfrak{u}(t)$, and $\mathfrak{a}(t)$ be the pullback $d\big(L^\ast_{x(t)}\big)_e(\alpha(t))$ of $\alpha(t)$ to $\mathfrak{g}^\ast.$ Then $\mathfrak{a}(t)$ is an absolutely continuous function and $\mathfrak{u} (t)\in \mathcal{C}(\mathfrak{a}(t)) := \mathcal{C}(e,\mathfrak{a}(t)) \subset S_{F_e}[0,1]$ is a measurable control satisfying
	\begin{eqnarray} \label{Sistema de controle}
		\dot{\mathfrak{a}}(t) = - \mbox{\rm ad}^\ast(\mathfrak{u}(t))(\mathfrak{a}(t)) = \mathfrak{a}(t) ([\mathfrak{u}(t), \cdot]),
	\end{eqnarray}	
	where ${\rm ad}^\ast$ is the infinitesimal coadjoint representation on $\mathfrak{g}^\ast$.
	Reciprocally let $\mathfrak{a}(t) \in \mathfrak{g}^\ast\setminus\{0\}$ be an absolutely continuous function and $\mathfrak{u}(t) \in \mathcal{C}(\mathfrak{a}(t))$ be a measurable control such that Eq. (\ref{Sistema de controle}) is satisfied. Then, given $t_0 \in I$ and $x_0 \in G$, there exists a unique Pontryagin extremal $(x(t), \alpha(t))$ corresponding to the control $ \mathfrak{u}(t)$ such that $\mathfrak{a}(t) = d\big(L^\ast_{x(t)}\big)_e(\alpha(t))$ and $x(t_0) = x_0.$
	
\end{teo}

\begin{defi} 
A solution $(\mathfrak{u}(t), \mathfrak{a}(t))$ of $\dot{\mathfrak{a}} = -\mbox{\rm ad}^\ast(\mathfrak{u})(\mathfrak{a}),$ where $\mathfrak{u}(t) \in \mathcal{C}(\mathfrak{a}(t))$ is a measurable function, is also called a {\rm Pontryagin extremal} of $(G, F)$ and $\mathfrak{a}(t)$ is called the {\rm vertical part} of a Pontryagin extremal.
\end{defi}


\subsection{Mollifier smoothing of convex functions}
\label{subsection_mollifier smoothing of convex functions}

In this section we will present the mollifier smoothing of convex functions, but before that, we will define the standard mollifier function and the mollifier smoothing. 

We will restrict ourselves to the mollifier smoothing of continuous functions in Euclidean spaces, but this smoothing can be done for locally integrable functions. For more details, see Appendix C of \cite{Evans}.


Let $U$ be an open subset of $\mathbb{R}^n,$ $\varepsilon>0$ and $U_{\varepsilon}=\{y\in U : \dist(y,\partial U)>\varepsilon\},$ where $\partial U$ is
the boundary of $U.$ Let $\Vert \cdot \Vert$ be the Euclidean norm on $\mathbb{R}^n.$


\begin{defi} \label{definição de mollifier} 
	
	\
	
\begin{enumerate}
	
\item Define $\eta \in \mathit{C}^{\infty}(\mathbb{R}^n)$ by
		\begin{eqnarray*}
			\eta(y) & :=& \begin{cases}
				
				C\exp\left(\frac{1}{ \|y\|^2-1}\right),  \quad & {\rm{if  \quad }}\|y\|<1; \\
				
				0, & {\rm{if}} \quad \|y\|\geq 1,
				
			\end{cases}
		\end{eqnarray*}
		where the constant $C> 0$ is chosen so that $\int_{\mathbb{R}^n}\eta \ dx=1;$

\item  For  each $\varepsilon>0,$ define
\begin{eqnarray*}
	\eta_{\varepsilon}(y) & := & \frac{1}{\varepsilon^n}\eta\left(\frac{y}{\varepsilon}\right).
\end{eqnarray*}

We call $\eta$ of {\rm standard mollifier.}
\end{enumerate}
\end{defi} 


Note that $\eta_{\varepsilon}$ is $\mathit{C}^{\infty},$ 
\begin{eqnarray*}
	\int_{\mathbb{R}^n}\eta_{\varepsilon} \ dy = \int_{\mathbb{R}^n}\eta \ dy=1  \quad {\rm and} \quad {\rm{supp}}(\eta_{\varepsilon})=B_{\Vert \cdot \Vert}[0,\varepsilon],
\end{eqnarray*}
where $\rm{supp}(\eta_{\varepsilon})$ is the support of  $\eta_{\varepsilon}.$


\begin{defi} If $f:U \rightarrow \mathbb{R}$ is a continuous function, its {\rm mollifier smoothing} $\eta_\varepsilon \ast f$ is the convolution of $\eta_\varepsilon$ and $f$ in $U_\varepsilon.$ That is,
	\begin{eqnarray*}
		(\eta_{\varepsilon}*f)(y)  := \int_U\eta_{\varepsilon}(y-z)f(z)dz = \int_{B_{\Vert \cdot \Vert}[0,\varepsilon]}\eta_\varepsilon(z)f(y-z)dz,
	\end{eqnarray*}
	for every $y \in U_\varepsilon.$
\end{defi}


Before presenting the properties of $\eta_\varepsilon \ast f$, we will first introduce the definition of a multi-index.

\begin{defi}
	
	\
\begin{enumerate}
	\item A vector of the form $\nu=(\nu_1,\ldots,\nu_n),$ where each component $\nu_i$ is a non-negative integer, is called a {\rm multi-index} of order
		$$|\nu|=\nu_1+\cdots+\nu_n;$$
		
	\item Given a  multi-index $\nu,$ define
		$$D^\nu f(y):= \frac{\partial^{|\nu|}f(y)}{\partial (y^1)^{\nu_1}\cdots\partial (y^n)^{\nu_n}},$$
		where $f$ is a real valued function defined in an open subset of $\mathbb{R}^n$.
\end{enumerate}
\end{defi}


\begin{teo} \label{propriedades da aplicação mollifier} (Properties of mollifiers). Let $f:U\subset \mathbb{R}^n \rightarrow \mathbb{R}$ be a  continuous function, then the following statements are true:
	\begin{enumerate}
		\item \label{item_suavizacao suave}$(\eta_{\varepsilon}*f)\in \mathit{C}^{\infty}(U_\varepsilon)$;
		\item $(\eta_{\varepsilon}*f) \rightarrow f$ as $\varepsilon \rightarrow 0$ (pointwise);
		\item $(\eta_{\varepsilon}*f) \rightarrow f$ uniformly on compact subsets of $U.$
	\end{enumerate}
	\begin{proof} See Theorem 6 of Appendix C of \cite{Evans}. 
	\end{proof}
\end{teo}


\begin{obs} \label{derivada mollifier} 
	As a consequence of the proof of Item \eqref{item_suavizacao suave} of Theorem \ref{propriedades da aplicação mollifier} (see  \cite{Evans}),  we have that for every $y \in U_\varepsilon$ and for every multi-index $\nu=(\nu_1,\ldots,\nu_n)$
	\begin{eqnarray*}
		D^\nu(\eta_{\varepsilon}*f)(y) =  \int D^\nu\eta_{\varepsilon}(y-z)f(z)dz = (D^\nu\eta_{\varepsilon}*f)(y). 
	\end{eqnarray*}
	Moreover, when $f$ is $\mathit{C}^\infty,$ it follows that
	\begin{eqnarray*}
		D^\nu(\eta_{\varepsilon}*f)(y) =  \int\eta_{\varepsilon}(z)D^\nu f(y-z)dz =  (\eta_{\varepsilon}*D^\nu f)(y).
	\end{eqnarray*}
\end{obs}


\begin{lem} \label{suavizacao mollifier e convexa}
	Let $f:\mathbb{R}^n \rightarrow \mathbb{R}$ be a convex function. Then the mollifier smoothing of $f$ is also a convex function.
\end{lem}
\begin{proof}
	See Lemma 2 of \cite{Fukuoka-Setti}.
\end{proof}


\begin{lem} 
	\label{condicao para curvatura} 
	If $f:{\mathbb R} \rightarrow {\mathbb R}$ is a convex function, then $(\eta_{\varepsilon}*f){''}(y)\geq 0$ for every $y \in \mathbb R$. Moreover, $(\eta_{\varepsilon}*f)''(y)=0$ if and only if $f$ is affine in $(y-\varepsilon,y+\varepsilon).$ 
\end{lem}

\begin{proof}
	This lemma is proved in Lemma 3 of \cite{Fukuoka-Setti}.
\end{proof}

Although we placed the next result in this section, we didn't find this result elsewhere, and therefore its proof is presented here for the sake of completeness. 

\begin{lem} \label{Suavização de função fortemente convexa} If $f:\mathbb{R}^n \rightarrow \mathbb{R}$ is a strongly convex function with module $\gamma$, then $\eta_\varepsilon \ast f$ is also strongly convex with modulus $\gamma.$
\end{lem}

\begin{proof} We have
	$$f(ty+(1-t)z) \leq tf(y)+(1-t)f(z) -\frac{1}{2} \gamma t(1-t) \Vert y-z \Vert^2,$$ 
for every $y, z \in \mathbb{R}.$ Then
\begin{eqnarray*}
	& & (\eta_\varepsilon \ast f)(ty+(1-t)z) \\
	& = & \int \eta_\varepsilon(w) f(ty+(1-t)z - w) dw \\
	& = & \int \eta_\varepsilon(w) f(t(y-w)+(1-t)(z-w)) dw \\
	& \leq &  \int \eta_\varepsilon(w) \left[ tf(y-w) + (1-t)f(z-w) -\frac{1}{2} \gamma t(1-t) \Vert y-z \Vert^2 \right] dw \\
	& = & t(\eta_\varepsilon \ast f)(y) + (1-t)(\eta_\varepsilon \ast f)(z) - \frac{1}{2} \gamma t(1-t) \Vert y-z \Vert^2.
\end{eqnarray*}
This proves that $\eta_\varepsilon \ast f$ is strongly convex with modulus $\gamma.$	
\end{proof}


\subsection{Distance between parallel hyperplanes}

We end this section with a straightforward geometric result which is presented here for future reference.

\begin{pro}
Let $(\mathbb{V},\left< \cdot, \cdot \right>)$ be a finite dimensional vector space endowed with an inner product and consider $k_1,k_2 \in \mathbb{R}$ and $\alpha \in \mathbb{V}^{\ast}\backslash \{0\}$.
Then
\begin{equation}
\label{Distancia entre hiperplanos}
	\dist(\left\{\alpha = k_1\right\}, \left\{ \alpha = k_2\right\}) = \frac{\vert k_1 - k_2 \vert }{\Vert \alpha \Vert_{\ast} }.
\end{equation}
\end{pro}

\section{Strongly convex asymmetric norms and strongly convex functions}
\label{Strongly convex functions and strongly convex asymmetric norms}

We begin recalling the following well known result about continuously differentiable strongly convex functions.



\begin{teo} \label{Equivalencias - Fortemente convexa (Definicao Classica)} Let $U \subset \mathbb{V}$ be a non-empty open convex set and $f: U \rightarrow \mathbb{R}$ be a continuously differentiable function on $U$. Let $(y^1, \ldots, y^n)$ be a coordinate system of $\mathbb{V}$ relative to a basis of
		$\mathbb{V}.$ Then the following properties are equivalent:
	\begin{enumerate}
		\item The function $f$ is strongly convex in $U$ with modulus $\gamma > 0.$
		\item For any $y, z \in U$ 
		$$ f(z) \geq f(y) + d f_y(z-y) + \frac{1}{2}\gamma \Vert z-y\Vert^2; $$
		\item For any $y, z \in U$
		$$ (df_z - df_y) (z-y) \geq \gamma \Vert z-y\Vert^2. $$		
		\item \label{item_convexidade forte caso suave} When $f$ is twice continuously differentiable in $U$, the above properties are also equivalent to:
		
		The Hessian $\Hess f$ of $f$ is strongly (or uniformly) positive definite at every point of $U$. More precisely, 
		$$(\Hess f)_y (v, v) := f_{y^iy^j}(y)v^i v^j \geq  \gamma \Vert v \Vert^2,  \forall y \in U, \forall v \in \mathbb{V}.$$
	\end{enumerate}
\end{teo}

\begin{proof} See Chapter 1 of \cite{SunYuan}.
\end{proof}

In Theorem \ref{Fortemente convexa - Equivalencia} and Proposition \ref{proposition_norma limitante inferior}, we prove versions of Theorem \ref{Equivalencias - Fortemente convexa (Definicao Classica)} for strongly convex asymmetric norms. 
In Proposition \ref{proposition_para mostrar que S_F_e está contida em uma esfera euclidiana}, we prove that $F$ is a strongly convex asymmetric norm, iff there exist $\mathcal{R}>0$ such that $S_F[0,1]$ is ``more curved'' than a Euclidean sphere of radius $\mathcal{R}$.


\begin{teo} \label{Fortemente convexa - Equivalencia} Let $F$ be an asymmetric norm on $\mathbb{V}$ and $\gamma$ be a positive constant. Then
	$F^2$ is strongly convex with module $\gamma$ (according to Definition \ref{Fortemente convexa - definição clássica}) if, and only if,
	$F$ is strongly convex with respect to $\sqrt{\frac{\gamma}{2}}\Vert \cdot \Vert$ (according to Definition \ref{Fortemente convexa - definicao artigo R-H}).
\end{teo}

\begin{proof} \rm
	$(\Rightarrow)$ Suppose that $F^2$ is strongly convex with modulus $\gamma.$ Then
	\begin{equation*}
		F^2(ty+(1-t)z) \leq tF^2(y)+(1-t)F^2(z)-\frac{1}{2}\gamma t(1-t) \Vert y-z \Vert^2,
	\end{equation*}
	for all $t\in[0,1]$ and any $y,z \in \mathbb{V}.$ 
		
	Let be $y \in \mathbb{V}$ and $\alpha_y \in \partial F^2(y)$ be fixed, but arbitrary. 		
	Fix $t=t_0 \in (0,1).$ Given $z \in \mathbb{V},$ we get from  \eqref{desigualdade subgradiente 2} that
	\begin{eqnarray*}
	F^2(t_0z+(1-t_0)y) & \geq & F^2(y) + t_0 \alpha_y (z-y).
	\end{eqnarray*}
	By the Classical Definition of strongly convex function, we have
	\begin{eqnarray*}
		t_0 F^2(z)+(1-t_0)F^2(y) - \frac{1}{2}\gamma t_0(1-t_0) \Vert z-y \Vert^2 \geq F^2(t_0z+(1-t_0)y).
	\end{eqnarray*}
	From the last two inequalities, we obtain
	\begin{eqnarray*}
		t_0 F^2(z)+(1-t_0)F^2(y) - \frac{1}{2}\gamma t_0(1-t_0) \Vert z-y \Vert^2 \geq F^2(y) + t_0 \alpha_y (z-y),
	\end{eqnarray*}
	what gives
	\begin{eqnarray*}
		F^2(z) \geq F^2(y)+\alpha_y(z-y) + \frac{1}{2}\gamma(1-t_0) \Vert z-y \Vert^2.
	\end{eqnarray*}
	Note that the above inequality holds for every $t_0\in (0,1).$ Considering $t_0 \to 0$ we get
	\begin{eqnarray*}
		F^2(z) \geq F^2(y)+\alpha_y(z-y) + \frac{1}{2} \gamma \Vert z-y \Vert^2,
	\end{eqnarray*}
	which proves that $F$ is a strongly convex asymmetric norm with respect to $\sqrt{\frac{\gamma}{2}}\Vert \cdot \Vert.$

	$(\Leftarrow)$ Suppose that $F$ is strongly convex with respect to $\sqrt{\frac{\gamma}{2}}\Vert \cdot \Vert.$ Then
	\begin{equation*}
		F^2(w) \geq F^2(p) + \alpha_p(w-p) + \frac{1}{2} \gamma\Vert w-p \Vert^2,
	\end{equation*}
	for every $p,w \in \mathbb{V}$ and $\alpha_p \in \partial F^2(p).$	
		
	Given $t \in [0,1]$ and $y,z \in \mathbb{V},$ we have
	\begin{eqnarray}
	\label{eqnarray_t vezes F quadrado}
		tF^2(y) & \geq & tF^2(p) + t\alpha_p(y-p) + t \frac{1}{2} \gamma \Vert y-p \Vert^2
	\end{eqnarray}
	and
	\begin{eqnarray}
	\label{eqnarray_1 menos t vezes F quadrado}
		& & (1-t)F^2(z) \geq (1-t)F^2(p) + (1-t)\alpha_p(z-p) + (1-t) \frac{1}{2} \gamma \Vert z-p \Vert^2. 
	\end{eqnarray}
	Consider $p=ty+(1-t)z.$ Then $y-p = (1-t)(y-z)$, $z-p=t(z-y)$, and adding \eqref{eqnarray_t vezes F quadrado} and \eqref{eqnarray_1 menos t vezes F quadrado}, it is straightforward that
	\begin{eqnarray*}
		F^2(ty+(1-t)z) \leq tF^2(y)+(1-t)F^2(z)-\frac{1}{2}\gamma t(1-t)\Vert y-z \Vert^2.
	\end{eqnarray*}
	This completes the proof of the result.
\end{proof}

\begin{obs} If $F$ is a strongly convex asymmetric norm with respect to $\sqrt{\frac{\gamma}{2}}\Vert \cdot \Vert$ then $F$ is strictly convex. In fact, by Theorem \ref{Fortemente convexa - Equivalencia}, $F^2$ is strongly convex with modulus $\gamma$, 
what implies that $F^2((1-t)y+tz) < (1-t)F^2(y) + t F^2(z)$ for every $y,z \in \mathbb{V}$, $y\neq z$, and $t\in (0,1)$.
In particular, 
\begin{align*}
F((1-t)y+tz) < (1-t)F(y) + t F(z)
\end{align*} 
for every $y,z \in S_F[0,1]$, $y\neq z$, and $t\in (0,1)$, that is,  $F$ is strictly convex (see \eqref{align_F estritamente convexo na esfera}).
\end{obs}

\begin{pro}
\label{proposition_norma limitante inferior}
Let $F$ be an asymmetric norm on $(\mathbb{V},\Vert \cdot \Vert)$ which is smooth on $\mathbb{V} \backslash \{0\}$. If $F$ is Minkowski norm on $\mathbb{V}$, then there exist $\gamma > 0$ such that 
	\begin{align}
	\label{align_norma limitante inferior}
g_{ij(y)} v^i v^j := \frac{1}{2}\frac{\partial^2 F^2}{\partial y^i \partial y^j}(y)v^iv^j \geq \frac{1}{2} \gamma \Vert v\Vert^2 
	\end{align}
for every $y \in \mathbb{V} \backslash\{0\}$ and $v \in \mathbb{V}$.
Moreover, $F$ is strongly convex with respect to every norm $\sqrt{\frac{\gamma}{2}}\Vert \cdot \Vert$ that satisfies \eqref{align_norma limitante inferior}.
Reciprocally, if $F$ is a strongly convex asymmetric norm with respect to $\sqrt{\frac{\gamma}{2}}\Vert \cdot \Vert$, then $F$ is a Minkowski norm such that \eqref{align_norma limitante inferior} is satisfied.
\end{pro}

\begin{proof}
	Before we go into the proof of this result, we make some remarks on Taylor expansions.
	Let $f:[a,x] \rightarrow \mathbb{R}$ be a $C^k$-function and suppose that its $k$-th derivative $f^{(k)}$ is differentiable on $(a,x)$ (not necessarily $C^1$).
	Then the Taylor expansion with Lagrange remainder holds, that is, there exist a $c\in (a,x)$ such that 
	\begin{align}
		\label{align Lagrange unidimensional}
		f(x) = \sum_{l=0}^k \frac{f^{(l)}(a)}{l!}(x-a)^{l} + \frac{f^{(k+1)}(c)}{(k+1)!}(x-a)^{k+1} 
	\end{align}
	holds. 
	This result is usually stated with $f$ being of class $C^{k+1}$ on $[a,b]$, but this isn't necessary if we see the proof of this result.
	In fact, we only need that $f^{(k)}$ is differentiable on $(a,x)$ when applying the Rolle's Theorem on the function 
	\[
	t \mapsto f(x) - \sum_{l=0}^k \frac{f^{(l)}(t)}{l!}(x-t)^l - \left(\frac{x-t}{x-a}\right)^{k+1}r_k(x),
	\]
	where $r_k(x)$ is the remainder term
	\[
	r_k(x) = f(x) - \sum_{l=0}^k \frac{f^{(l)}(a)}{l!}(x-a)^{l}.
	\]
	
	Let $F:\mathbb{V} \rightarrow  [0,\infty)$ be a Minkowski norm. 
	Parametrize $\mathbb{V}$ by coordinates $(y^1, \ldots, y^n)$.
	Then $F^2$ is $C^\infty$ on $\mathbb{V}\backslash \{0\}$ and $C^1$ on $\mathbb{V}$ due to the positive homogeneity of $F$.
	We have that $g_{ij(y)} dy^i \otimes dy^j =\frac{1}{2}\frac{\partial^{2} F^2}{\partial y^i \partial y^j}(y)dy^i \otimes dy^j$ is a family of inner products on $\mathbb{V}$ parametrized by $ y\in \mathbb{V} \backslash \{0\}$ such that $g_{ij(y)}=g_{ij(\lambda y)}$ for every $\lambda > 0$ and $y \in \mathbb{V}\backslash \{0\}$ (see \cite{BaoChernShen}).
	Therefore, the family of inner products above is equal to the family $\{g_{ij(y)} dy^i \otimes dy^j\}_{y \in S_{\Vert \cdot\Vert} [0,1]}$. 
	Due to the compactness of $S^{n-1}$ and the continuity of $g_{ij(y)}$, there exist $\gamma > 0$ such that
	\begin{align}
		\label{align norma limitante inferior} 
		g_{ij(y)}v^i v^j \geq \frac{1}{2}\gamma \Vert v\Vert^2
	\end{align} 
	for every $y \in S_{\Vert \cdot\Vert}[0,1]$ (and therefore $y\in \mathbb{V}\backslash\{0\}$) and $v\in \mathbb{V}$.
	
	Let $y_0 , y_1 \in \mathbb{V}$ be distinct points and consider $\gamma:[0,1] \rightarrow \mathbb{V}$ given by $\gamma(t)=(1-t).y_0 + t.y_1$.
	If $0 \not\in \gamma((0,1))$, then $F^2(\gamma(t))$ is smooth in $(0,1)$ and $C^1$ in $[0,1]$. 
	Then 
	\begin{align}
		\label{align Lagrange F quadrado}
		F^2(y_1) & = F^2(y_0) + d(F^2)_{y_0}(y_1-y_0)+g_{ij}(c)(y^i_1 - y^i_0)(y^j_1 - y^j_0) \\
		& \geq F^2(y_0) + d(F^2)_{y_0}(y_1-y_0)+\frac{1}{2}\gamma \Vert y_1 - y_0\Vert^2, \label{align fortemente convexo caso suave}
	\end{align}
	where \eqref{align Lagrange F quadrado} is \eqref{align Lagrange unidimensional} applied to $F^2(\gamma(t))$ and  \eqref{align fortemente convexo caso suave} is due to \eqref{align norma limitante inferior}.  
	On the other hand, if $\gamma(t_0)=0$ with $t_0 \in (0,1)$, then we consider \eqref{align fortemente convexo caso suave} applied to $F^2(\gamma(t))\vert_{[0,t_0]}$ and $F^2(\gamma(t))\vert_{[t_0,1]}$, what implies
	\begin{align}
	\label{align_desigualdade minkowski fortemente convexa 1}
	F^2(y_1) \geq \frac{1}{2}\gamma\Vert y_1 \Vert^2
	\end{align}
and
	\begin{align}
	0 & \geq F^2(y_0) - d(F^2)_{y_0}(y_0)+\frac{1}{2}\gamma\Vert y_0\Vert^2 \nonumber \\
	& = F^2(y_0) - d(F^2)_{y_0}(y_1) + d(F^2)_{y_0}(y_1 - y_0) +\frac{1}{2}\gamma\Vert y_0\Vert^2 \label{align_segundo segmento}
	\end{align}
because $d(F^2)_0 \equiv 0$.
Notice that $y_1=-\frac{\Vert y_1 \Vert}{\Vert y_0 \Vert} y_0$ because $y_0$ and $y_1$ lies in the opposite sides of a line that contains the origin.
Then 
	\begin{align}
	\label{align_lados opostos}
	-d(F^2)_{y_0}(y_1)=\frac{\Vert y_1\Vert}{\Vert y_0\Vert} d(F^2)_{y_0}\left( y_0 \right).
	\end{align}	
It is immediate from the positive homogeneity of $F$ and from \eqref{align_desigualdade minkowski fortemente convexa 1} that
	\begin{align}
	\label{align_positivamente homogenea}
d(F^2)_{y_0}(y_0)=2F^2(y_0) \geq \gamma \Vert y_0\Vert^2.	
	\end{align}
Plugging \eqref{align_lados opostos} and \eqref{align_positivamente homogenea} into \eqref{align_segundo segmento}, we have
	
	\begin{align}
	0 & \geq F^2(y_0) + \gamma \Vert y_0 \Vert \Vert y_1\Vert  + d(F^2)_{y_0}(y_1 - y_0) +\frac{1}{2}\gamma\Vert y_0\Vert^2 \label{align_desigualdade minkowski fortemente convexa 2}
	\end{align}
respectively.	
Summing up \eqref{align_desigualdade minkowski fortemente convexa 1} and \eqref{align_desigualdade minkowski fortemente convexa 2}, we have that
	\begin{align*}
		F^2(y_1) & \geq F^2(y_0) +  d(F^2)_{y_0}(y_1-y_0)+ \frac{1}{2}\gamma\Vert y_1 - y_0\Vert^2
	\end{align*}
because $y_1$ is a positive multiple of $-y_0$, what proves the strong convexity of $F$ with respect to $\sqrt{\frac{\gamma}{2}}\Vert \cdot\Vert$.

The reciprocal is direct consequence of Theorem \ref{Equivalencias - Fortemente convexa (Definicao Classica)} applied to $F^2$ restricted to a convex neighborhood of $y$ in $\mathbb{V}\backslash \{0\}$.
\end{proof}

\begin{obs}
Proposition \ref{proposition_norma limitante inferior} isn't a consequence of Item \eqref{item_convexidade forte caso suave} of Theorem \ref{Equivalencias - Fortemente convexa (Definicao Classica)}.
In fact, if $F$ is a Minkowski norm, then $F^2$ is of class $C^2$ at $0\in \mathbb{V}$ iff $F$ comes from an inner product (see \cite{BaoChernShen}).
\end{obs}


Before proving Proposition \ref{proposition_para mostrar que S_F_e está contida em uma esfera euclidiana}, we state a preliminary result.

\begin{lem}
\label{lemma_unicidade subgradiente no raio}
Let $f: \mathbb{V} \rightarrow \mathbb{R}$ be a convex function such that $f(y)>0$ for every $y\neq 0$.
Let $r\geq 1$ and suppose that $f$ is a positively homogeneous of degree $r$, that is, $f(\lambda y)= \lambda^r f(y)$ for every $\lambda > 0$. 
Consider $y \in \mathbb{V}\backslash \{0\}$, $\alpha \in \partial f(y)$ and $C\in \mathbb{R}$. 
If $C\alpha \in \partial f(y)$, then $C=1$. 
\end{lem}
\begin{proof}
Suppose that $C\alpha \in \partial f(y)$. Then
\[
f(\lambda y) \geq f(y) + (\lambda-1)C\alpha (y) \Leftrightarrow \left( \lambda^r - 1\right)f(y)\geq (\lambda - 1)C\alpha(y).
\]
If $\lambda > 1$, we have 
\[
\frac{\lambda^r - 1}{\lambda - 1}f(y) \geq C\alpha(y) \Rightarrow \lim_{\lambda \rightarrow 1^+}\frac{\lambda^r - 1}{\lambda - 1}f(y) = r.f(y) \geq C\alpha(y)
\]
due to L'Hospital rule.
Analogously, if $\lambda \in (0,1)$, we have
\[
\frac{\lambda^r - 1}{\lambda - 1}f(y) \leq C\alpha(y)\Rightarrow r.f(y)\leq C\alpha(y).
\]
Therefore, $C=\frac{r.f(y)}{\alpha(y)}$ is the unique value such that $C\alpha \in \partial f(y)$. But $\alpha \in \partial f(y)$, what implies $C=1$ and settles the lemma.
\end{proof}

\begin{pro} \label{proposition_para mostrar que S_F_e está contida em uma esfera euclidiana} Let $F$ be an asymmetric norm. 
Then $F$ is strongly convex with respect to $\sqrt{\frac{\gamma}{2}}\Vert \cdot \Vert$ iff there exists $\mathcal{R}>0$ satisfying the following condition: for every $y \in S_{F}[0,1]$ and $\alpha_y \in \partial F^2 (y)$ there exists $\tilde{y} \in \mathbb{V}$ such that
	\begin{eqnarray*}
		S_{F}[0,1] \subset B_{\Vert \cdot \Vert}[\tilde{y}, \mathcal{R}], \quad y \in S_{\Vert \cdot \Vert}[\tilde{y}, \mathcal{R}]
	\end{eqnarray*}
and $T_y(S_{\Vert \cdot\Vert}[\tilde{y},\mathcal{R}]) = \{\alpha_{y} = \alpha_y(y)\}$.
\end{pro}
\begin{proof}
	($\Rightarrow$) Fix $y \in S_F[0,1]$ and $\alpha_y \in \partial F^2(y)$. We claim that there exist $y^\prime$ and $\mathcal{R}_y > 0$ such that $S_F[0,1] \subset B_{\Vert \cdot\Vert}[y^\prime, \mathcal{R}_y]$, $y \in S_{\Vert \cdot\Vert}[y^\prime,\mathcal{R}_y]$ and $T_y(S_{\Vert \cdot\Vert}[y^\prime , \mathcal{R}_y]) = \{\alpha_y =\alpha_y(y)\}$. 
	
	As $F$ is a strongly convex asymmetric norm with respect to $\sqrt{\frac{\gamma}{2}}\Vert \cdot \Vert,$ then
	\begin{eqnarray} \label{Desigualdade auxiliar 1}
		F^2(z) \geq F^2(y)+ \alpha_y(z-y)+\frac{1}{2}\gamma \Vert z-y \Vert^2,
	\end{eqnarray}
	for every $z \in \mathbb{V}$. 
	
	Let $\{e_1, \ldots, e_n\}$ be an orthonormal basis of $\mathbb{V}$ and consider $z\in \mathbb{V}$. 
	Set $y = y^ie_i$, $z=z^ie_i$ and $(\alpha_y)_i = \alpha_y(e_i)$, $i \in \{1, \ldots, n\}$. Then
	\begin{eqnarray*}
		& & F^2(y) + \alpha_y(z-y) + \frac{1}{2}\gamma \Vert z-y \Vert^2 \\
		& = & F^2(y) + (\alpha_y)_i(z^i-y^i) + \frac{1}{2}\gamma \sum_{i=1}^n (z^i-y^i)^2 \\
		& = & F^2(y) + \sum_{i=1}^n \left[\sqrt{\frac{\gamma}{2}}(z^i-y^i) + \frac{(\alpha_y)_i}{\sqrt{2\gamma}} \right]^2 - \sum_{i=1}^n \frac{(\alpha_y)_i^{\ 2} }{2\gamma} \\
		& \leq & F^2(z).
	\end{eqnarray*}
	For every $z \in S_{F}[0,1],$ we have $F^2(z) = F(z) = 1 = F(y) = F^2(y),$ what implies 
	\begin{eqnarray*}
		& & \sum_{i=1}^n \left[\sqrt{\frac{\gamma}{2}}(z^i-y^i) + \frac{(\alpha_y)_i}{\sqrt{2\gamma}} \right]^2 \leq  \sum_{i=1}^n \frac{(\alpha_y)_i^{\ 2} }{2\gamma} \nonumber\\
		& \Leftrightarrow & \frac{\gamma}{2} \sum_{i=1}^n \left[(z^i-y^i) + \frac{(\alpha_y)_i}{\gamma} \right]^2  \leq \sum_{i=1}^n \frac{(\alpha_y)_i^{\ 2} }{2\gamma} \nonumber\\
		& \Leftrightarrow & \sqrt{\sum_{i=1}^n \left[z^i-\left(y^i - \frac{(\alpha_y)_i}{\gamma}\right) \right]^2}  \leq \frac{1}{\gamma} \sqrt{\sum_{i=1}^n (\alpha_y)_i^{\ 2}} \nonumber\\
		& \Leftrightarrow & \Bigg\Vert z-\left(y - \frac{((\alpha_y)_1, \ldots, (\alpha_y)_n)}{\gamma}\right) \Bigg\Vert   \leq \frac{1}{\gamma} \sqrt{\sum_{i=1}^n (\alpha_y)_i^{\ 2}}. \label{Desigualdade para mostrar que S_F_e está contida em uma esfera euclidiana}
	\end{eqnarray*}
Therefore if we set
	\begin{align}
	\nonumber
	y^\prime = y - \frac{((\alpha_y)_1, \ldots, (\alpha_y)_n)}{\gamma} \,\text{  and  }\,\mathcal{R}_y = \frac{1}{\gamma} \sqrt{\sum_{i=1}^n (\alpha_y)_i^{\ 2}},
	\end{align}	 
it is immediate that  $S_F[0,1] \subset B_{\Vert \cdot\Vert}[y^\prime, \mathcal{R}_y]$ and $y \in S_{\Vert \cdot\Vert}[y^\prime,\mathcal{R}_y]$.
Notice also that $S_{\Vert \cdot \Vert}[y^\prime, \mathcal{R}_y]$ is the inverse image of a regular value of the smooth function $f$ defined by
\[
z \mapsto \sum_{i=1}^n \left[(z^i - y^i) + \frac{(\alpha_y)_i}{\gamma} \right]^2.
\]
Therefore, $T_y(S_{\Vert \cdot \Vert}[y^\prime, \mathcal{R}_y]) = \left\{ \alpha_y = \alpha_y(y)\right\}$ because $\nabla f(y)$ is proportional to $((\alpha_y)_1, \ldots, (\alpha_y)_n)$, what implies that $T_y(S_{\Vert \cdot \Vert}[y^\prime, \mathcal{R}_y])$ is parallel to $\ker (\alpha_y)$, what settles the claim.

	Now we will determine $\mathcal{R}>0$ such that for every $y \in S_F[0,1]$ and $\alpha_y \in \partial F^2 (y)$, there exist a $\tilde{y}$ satisfying $S_F[0,1] \subset B_{\Vert \cdot\Vert}[\tilde{y},\mathcal{R}]$, $y \in S_{\Vert \cdot \Vert}[\tilde{y},\mathcal{R}]$ and $T_y(S_{\Vert \cdot\Vert}[\tilde{y},\mathcal{R}]) = \left\{ \alpha_y = \alpha_y(y)\right\}$. 
	In order to do so, we will prove that there exist $\mathcal{R}\geq \mathcal{R}_y$ for every $y\in S_F[0,1]$.
	Considering $z=y+e_i$ in (\ref{Desigualdade auxiliar 1}), we have that
	\begin{eqnarray*}
		\alpha_y(e_i) & \leq & F^2(y+e_i),
	\end{eqnarray*}
what implies
	\begin{align}
	\nonumber
	\mathcal{R}_y = \frac{1}{\gamma} \sqrt{\sum_{i=1}^n (\alpha_y)_i^{\ 2}} \leq \frac{1}{\gamma} \sqrt{\sum_{i=1}^n F^4(y+e_i)} \leq \max_{\substack{y\in S_F[0,1]\\ i=1,\ldots, n}} \frac{1}{\gamma} \sqrt{\sum_{i=1}^n F^4(y+e_i)},
	\end{align}
where the maximum exists due to the continuity of $F$ and the compactness of $S_F[0,1] \times \{1, \ldots, n\}$.
Therefore, we can define $\mathcal{R}$ as this maximum.
	Finally notice that if $y \in S_F[0,1]$ and $\alpha_y \in \partial F^2(y)$, then there exist a unique $\tilde{y} \in \mathbb{V}$ such that $S_{\Vert \cdot \Vert}[\tilde{y},\mathcal{R}]$ contains $y$, $T_y (S_{\Vert \cdot \Vert}[\tilde{y},\mathcal{R}]) = \{\alpha_y = \alpha_y(y)\}$ and $B_{\Vert \cdot \Vert}[y^\prime,\mathcal{R}_y] \subset B_{\Vert \cdot \Vert}[\tilde{y},\mathcal{R}]$, what also implies $S_F[0,1] \subset B_{\Vert \cdot \Vert}[\tilde{y},\mathcal{R}]$ and settles this part of the proof.
	
	($\Leftarrow$) Now suppose that there exists $\mathcal{R}>0$ satisfying the following condition: for every $y \in S_F[0,1]$ and $\alpha_y \in \partial F^2(y)$, there exists $\tilde{y} \in \mathbb{V}$ such that $S_{F}[0,1] \subset B_{\Vert \cdot \Vert}[\tilde{y}, \mathcal{R}]$, $y \in S_{\Vert \cdot \Vert}[\tilde{y}, \mathcal{R}]$ and $T_y(S_{\Vert \cdot\Vert}[\tilde{y},\mathcal{R}]) = \{\alpha_{y} = \alpha_y(y)\}$.
	Then $S_{\left\Vert \cdot \right\Vert}[\tilde{y},\mathcal{R}]$ is the unit sphere of a Minkowski norm $\tilde{F}$ due to Theorem 2 of \cite{Fukuoka-Setti}, what implies that there exist $\gamma > 0$ such that
	\[
	\tilde{F}^2(z) \geq \tilde{F}^2(y) + d(\tilde{F}^2)_y (z-y) + \frac{1}{2}\gamma \Vert z-y\Vert^2,
	\]
where $d(\tilde{F}^2)_y$ is proportional to $\alpha_y$, due to Proposition \ref{proposition_norma limitante inferior}. In addition, we have that $S_F[0,1] \subset B_{\Vert \cdot \Vert}[\tilde{y},\mathcal{R}]$, what implies $F^2 \geq \tilde{F}^2$. 
	Therefore,
	\[
	F^2(z) \geq F^2(y) + d(\tilde{F}^2)_y (z-y) + \frac{1}{2}\gamma \Vert z-y\Vert^2
	\]
due to $F^2(y) = \tilde{F}^2(y) = 1$. 
Finally, Lemma \ref{lemma_unicidade subgradiente no raio} implies $d(\tilde{F}^2)_y = \alpha_y$, what settles the strong convexity of $F$ with respect to $\sqrt{\frac{\gamma}{2}} \Vert\cdot\Vert$.
\end{proof}

\begin{obs}
For the sake of simplicity, we will simply say that $F$ is strongly convex if its modulus of convexity isn't important. 
\end{obs} 


\section{Pontryagin extremals}
\label{section_Vertical part of a Pontryagin extremal}


Let $F$ be a left-invariant $C^0$-Finsler structure on a Lie group $G$.
Theorem \ref{Teorema Jessica e Ryuichi} gives a representation of Pontryagin extremals on $S_{F_e}[0,1] \times \mathfrak{g}^\ast$.
In this section, we study further properties of Pontryagin extremals and extended geodesic fields. 
In Corollary \ref{corollary_Jessica Ryuichi}, we adapt Theorem \ref{Teorema Jessica e Ryuichi} for the case where $F$ is strongly convex.
Still in the strongly convex case, Theorem \ref{theorem_extremal de Pontryagin em grupos} states that given $(x_0, \alpha_0) \in T^\ast G \backslash 0$, there exist a unique Pontryagin extremal $t \in \mathbb{R} \mapsto (x(t),\alpha(t))$ on $T^\ast G \backslash 0$ such that $(x(0), \alpha(0)) = (x_0 , \alpha_0)$.
Theorem \ref{a(t) esta definida para toda a reta} is the ``Lie algebra version'' of Theorem \ref{theorem_extremal de Pontryagin em grupos}.

{\em
In this section, $G$ be a Lie group endowed with a left-invariant $C^0$-Finsler structure $F$ and 
its Lie algebra $\mathfrak{g}$ is endowed with an inner product $\left< \cdot, \cdot \right>$.
}
 
If $F_e$ is strictly convex asymmetric norm and $\mathfrak{a} \in \mathfrak{g}^\ast \backslash \{0\}$, then there exist a unique $u(\mathfrak{a})$ that maximizes $\mathfrak{a}$ in $\sfe[0,1]$, that is $\mathcal{C}(\mathfrak{a}) = \{u(\mathfrak{a})\}$ (see Proposition \ref{Formula u alpha}). 
Then the rule $\tilde{\mathcal{E}}$ that associates every element $\mathfrak{a} \in \mathfrak{g}^\ast \setminus \{0\}$ to the set
\begin{eqnarray*}
	\tilde{\mathcal{E}}(\mathfrak{a}) = \left\{ -\mbox{ad}^\ast(u)(\mathfrak{a}) : u \in \mathcal{C}(\mathfrak{a}) \right\},
\end{eqnarray*}
can be considered as the application 
\begin{eqnarray*}
	\tilde{\mathcal{E}}(\mathfrak{a}) = -\mbox{ad}^\ast(u(\mathfrak{a}))(\mathfrak{a}) = \mathfrak{a}([u(\mathfrak{a}), \cdot]) \in \mathfrak{g}^\ast.
\end{eqnarray*}


\begin{teo} \label{E til e Lipschitz} Let $G$ be a Lie group endowed with a strongly convex left invariant $C^0$-Finsler structure $F$. Let $K_{\mathfrak{g}^\ast \setminus \{0\}}$ be a compact subset of $\mathfrak{g}^\ast \setminus \{0\}.$
	Then the map $\tilde{\mathcal{E}}$ is Lipschitz on $K_{\mathfrak{g}^\ast \setminus \{0\}}.$
\end{teo}

	\begin{proof}
	Let $\{e_1, \ldots, e_n\}$ be an orthonormal basis of $\mathfrak{g}$.
	
	Let
	\begin{eqnarray}
		\hat{C} & = &  \max \left\{ \big\vert c_{ij}^{\ \ k} \big\vert : \, i, j, k \in \{1, \ldots, n\} \right\}   \mbox{ and} \label{eqnarray_cijk}\\
		\tilde{C} & = & \max\left\{ \Vert \mathfrak{a} \Vert_\ast : \, \mathfrak{a} \in K_{\mathfrak{g}^\ast \setminus \{0\}} \right\},\nonumber
	\end{eqnarray} 
	where $c_{ij}^{\ \ k}$ are the Lie algebra structure constants.
	
	
	Given $\mathfrak{a}, \mathfrak{b} \in K_{\mathfrak{g}^\ast \setminus \{0\}}$, we have
	\begin{eqnarray}
		& & \big\Vert \tilde{\mathcal{E}}(\mathfrak{a}) - \tilde{\mathcal{E}}(\mathfrak{b}) \big\Vert_\ast
		 = \big\Vert \mathfrak{a}[ u(\mathfrak{a}), \cdot] - \mathfrak{b} [u(\mathfrak{b}), \cdot]  \big\Vert_\ast \nonumber \\
		& \leq & \big\Vert \mathfrak{a}[u(\mathfrak{a}), \cdot] - \mathfrak{a}[u(\mathfrak{b}), \cdot]  \big\Vert_\ast + \big\Vert \mathfrak{a}[u(\mathfrak{b}), \cdot] - \mathfrak{b}[u(\mathfrak{b}), \cdot]  \big\Vert_\ast \nonumber \\
		& = & \big\Vert \mathfrak{a}[u(\mathfrak{a})-u(\mathfrak{b}), \cdot] \big\Vert_\ast + \big\Vert (\mathfrak{a}-\mathfrak{b})[u(\mathfrak{b}), \cdot] \big\Vert_\ast \nonumber \\		 		
		& = & \max_{y \in S_{\Vert \cdot \Vert}[0,1]} \big\vert \mathfrak{a}[u(\mathfrak{a})-u(\mathfrak{b}), y^je_j] \big\vert + \max_{y \in S_{\Vert \cdot \Vert}[0,1]} \big\vert (\mathfrak{a}-\mathfrak{b})[u(\mathfrak{b}), y^je_j] \big\vert \nonumber \\	
		& = & \max_{y \in S_{\Vert \cdot \Vert}[0,1]} \big\vert y^j \big\vert \big\vert \mathfrak{a}[u(\mathfrak{a})-u(\mathfrak{b}), e_j] \big\vert +  \max_{y \in S_{\Vert \cdot \Vert}[0,1]} \big\vert y^j \big\vert \big\vert (\mathfrak{a}-\mathfrak{b})[u(\mathfrak{b}), e_j] \big\vert \nonumber \\
		& \leq & \sum_{j=1}^n \left( \big\vert \mathfrak{a}[u(\mathfrak{a})-u(\mathfrak{b}), e_j] \big\vert + \big\vert (\mathfrak{a}-\mathfrak{b})[u(\mathfrak{b}), e_j] \big\vert \right) \nonumber \\		
		& \leq & \sum_{j=1}^n \left( \Vert \mathfrak{a} \Vert_\ast \big\Vert [u(\mathfrak{a})-u(\mathfrak{b}), e_j] \big\Vert + \Vert \mathfrak{a}-\mathfrak{b} \Vert_\ast \big\Vert [u(\mathfrak{b}), e_j] \big\Vert \right)  \nonumber \\	
		& = & \sum_{j=1}^n \left( \Vert \mathfrak{a} \Vert_\ast \big\Vert [(u^{i}(\mathfrak{a})-u^{i}(\mathfrak{b}))e_i, e_j] \big\Vert + \Vert \mathfrak{a}-\mathfrak{b} \Vert_\ast \big\Vert [u^{i}(\mathfrak{b})e_i, e_j] \big\Vert \right) \nonumber \\	
		& \leq & \sum_{j=1}^n  \left(\tilde{C} \big\vert u^{i}(\mathfrak{a})-u^{i}(\mathfrak{b}) \big\vert  \big\Vert [e_i, e_j] \big\Vert + \big\vert u^{i}(\mathfrak{b})\big\vert \Vert \mathfrak{a}-\mathfrak{b} \Vert_\ast \big\Vert [e_i, e_j] \big\Vert \right) \nonumber \\
		& \leq & \sum_{i,j=1}^n  \tilde{C} \big\Vert u(\mathfrak{a})-u(\mathfrak{b}) \big\Vert  \left\Vert \sum_{k=1}^n c_{ij}{}^k e_k \right\Vert \nonumber \\
		& & + \sum_{i,j=1}^n \big\Vert u(\mathfrak{b})\big\Vert \Vert \mathfrak{a}-\mathfrak{b} \Vert_\ast \left\Vert \sum_{k=1}^n c_{ij}{}^k e_k \right\Vert \label{eqnarray_estimativa Lipschitz u}\\
		& \leq & \left(\tilde{C}\mathcal{K}+ \frac{1}{r_m}\right)\hat{C}n^3 \Vert \mathfrak{a} - \mathfrak{b}\Vert_{\ast}, \nonumber
	\end{eqnarray}
where the last inequality is due to \eqref{eqnarray_u Lipschitz}, \eqref{constante rho}, \eqref{eqnarray_cijk} and
\begin{align}
\max_{\mathfrak{b} \in K_{\mathfrak{g}^\ast \setminus \{0\}}} \Vert u(\mathfrak{b})\Vert \leq & \max_{y \in \sfe[0,1]} \Vert y\Vert = \frac{1}{r_m}.\label{align_maximo em SF eh 1 sobre rm}
\end{align}
Therefore $\left( \tilde{C}\mathcal{K} + \frac{1}{r_m}\right)\hat{C}n^3$ is a Lipschitz constant of $\tilde{\mathcal{E}}$.
\end{proof}


\begin{pro} \label{condições sobre a(t) e sua derivada} Let $(G, F)$ be a Lie group endowed with a left-invariant $C^0$-Finsler structure $F.$ Let $(\mathfrak{u}(t),\mathfrak{a}(t))$ be a Pontryagin extremal of $(G, F),$ with $t \in I$, where $I$ is any interval containing the origin. The following conditions are true:
	\begin{enumerate}
		\item There exists a $C>0$ such that $\Vert \mathfrak{a} [u(\mathfrak{a}), \cdot]\Vert_\ast <C$ for every $\mathfrak{a} \in $ $S_{(F_e)_\ast}[0,(F_e)_\ast(\mathfrak{a}(0))].$ 
		\item $\mathfrak{\dot{a}}(t) $ is bounded in $I.$
		\item $\mathfrak{a}(t)$ is Lipschitz in $I$.
		\item $\mathfrak{\dot{a}}(t)$ is Lipschitz in $I$, if $F_e$ is strongly convex asymmetric norm.
	\end{enumerate}
\end{pro}
\begin{proof} 	
	\item[(1)]
	We have
	\begin{align*}
	\big\Vert \mathfrak{a}[u(\mathfrak{a}), \cdot] \big\Vert_\ast =  \max_{y \in S_{\Vert \cdot \Vert}[0, 1]} \big\vert \mathfrak{a}[u(\mathfrak{a}), y] \big\vert
	\end{align*}	
	by definition.
	Notice that $\mathfrak{a}$, $y$ and $u(\mathfrak{a})$ are contained in compact subsets. 
	Then the continuity of $[\cdot, \cdot]$ settles this item.
	
	\
	
	\item[(2)] It follows directly from Item (1), since $\mathfrak{\dot{a}}(t) = \mathfrak{{a}}(t)[\mathfrak{u}(t), \cdot]$ and $\mathfrak{a}(t) \in S_{(F_e)_\ast}[0, (F_e)_\ast(\mathfrak{a}(0))]$ by (\ref{É constante pelo PMP}).
	
	\	
	
	\item[(3)] Follows directly from Item (2).
	
	
	\item[(4)] It follows from Item (3) and Theorem \ref{E til e Lipschitz}, since $\mathfrak{\dot{a}}(t) = \tilde{\mathcal{E}}(\mathfrak{a}(t)).$
\end{proof}

In the next theorem we prove that the vertical part of a Pontryagin extremal is defined for all $t\in \mathbb{R}.$


\begin{teo} \label{a(t) esta definida para toda a reta} Let $G$ be a Lie group endowed with a left-invariant strongly convex $C^0$-Finsler structure $F$ and $\mathfrak{a}_0 \in \mathfrak{g}^\ast \backslash \{0\}$. 
Then the initial value problem 
	\begin{eqnarray} \label{EDO a(t)}
		\mathfrak{\dot{a}} = \mathfrak{a}[u(\mathfrak{a}), \cdot],\,\,\,\,\,\,\,\, \mathfrak{a}(0) = \mathfrak{a}_0
	\end{eqnarray}
	admits a unique solution $\mathfrak{a}(t)$, which is defined for all $t \in \mathbb{R}.$
\end{teo}
\begin{proof}
The initial value problem  \eqref{EDO a(t)} admits a unique solution because $\mathfrak{a}\mapsto \mathfrak{a}[u(\mathfrak{a}), \cdot]$ is locally Lipschitz on $\mathfrak{g}^\ast \backslash \{0\}$ due to Theorem \ref{E til e Lipschitz}.

	In order to fix ideas, suppose that
	\begin{eqnarray} \label{Solucao a(t)}
		\mathfrak{a}(t):(a,b) \rightarrow \mathfrak{g}^\ast\setminus\{0\}
	\end{eqnarray} 
	is the maximal solution of (\ref{EDO a(t)}), with $a\in [-\infty, 0)$ and $b \in (0, \infty]$.
	The other cases are analogous.
	
	We claim for a contradiction assuming that $b < \infty.$ Let $(s_i)_{i \in \mathbb{N}}$ be a sequence in $(a, b)$ converging to $b.$ As $(s_i)$ is a Cauchy sequence and $\mathfrak{a}(t)$ is Lipschitz by Proposition \ref{condições sobre a(t) e sua derivada}, then $(\mathfrak{a}(s_i))_{i \in \mathbb{N}}$ is a Cauchy sequence, what implies that  $\lim\limits_{i \rightarrow \infty} \mathfrak{a}(s_i)$ exists. 
	Denote $p_0 =\lim\limits_{t \rightarrow b^{-}} \mathfrak{a}(t).$                
	
	Consider the initial value problem
	\begin{eqnarray} \label{PVI com a(c2)=p0}
		\mathfrak{\dot{a}} = \mathfrak{a}[u(\mathfrak{a}), \cdot], \quad \mathfrak{a}(b) = p_0.
	\end{eqnarray}
	By the Picard-Lindel\"of Theorem, there is a unique solution 
	$$\mathfrak{b}(t):[b, b+h) \rightarrow \mathfrak{g}^\ast\setminus\{0\}$$
	for the initial value problem (\ref{PVI com a(c2)=p0}), with $h>0.$
	
	Let 
	\begin{equation} \label{Solucao c(t)}
		\mathfrak{c(t)} = \begin{cases}
			\mathfrak{a}(t), \quad t \in (a, b) \\
			\mathfrak{b}(t), \quad t \in [b, b+h).
		\end{cases}
	\end{equation}
	Then
	\begin{eqnarray*}
		\frac{d}{d t}\Bigg\vert_{t=b^{-}}\mathfrak{c}(t) = \lim_{t \rightarrow b^{-}} \mathfrak{\dot{a}}(t) = \lim_{t \rightarrow b^{-}} \tilde{\mathcal{E}}(\mathfrak{a}(t)) = \tilde{\mathcal{E}}\left(\lim_{t \rightarrow b^{-}}\mathfrak{a}(t)\right) = \tilde{\mathcal{E}}(p_0)
	\end{eqnarray*}
	and
	\begin{eqnarray*}
		\frac{d}{d t}\Bigg\vert_{t=b^{+}}\mathfrak{c}(t) = \lim_{t \rightarrow b^{+}} \mathfrak{\dot{b}}(t) = \lim_{t \rightarrow b^{+}} \tilde{\mathcal{E}}(\mathfrak{b}(t)) = \tilde{\mathcal{E}}\left(\lim_{t \rightarrow b^{+}}\mathfrak{b}(t)\right) = \tilde{\mathcal{E}}(p_0),
	\end{eqnarray*}
	where the last equality holds because $\mathfrak{c}(t)$ is continuous. It follows that $\mathfrak{c}(t)$ is a continuation of $\mathfrak{a}(t)$ beyond $b,$ which is a contradiction, because the solution (\ref{Solucao a(t)}) is maximal. Thus, we must have $b= \infty.$ In the same way we can verify that $a = -\infty$ and $\mathfrak{a}(t)$ is defined for all $t \in \mathbb{R}$, what settles the proof.
\end{proof}

Now we will prove the Lie group version of Theorem \ref{a(t) esta definida para toda a reta} in Theorem \ref{theorem_extremal de Pontryagin em grupos}, but we will present some preliminary results before.

The following proposition is a particular case of Theorem 8.1 of \cite{Fukuoka-Rodrigues}.
\begin{pro}
\label{proposition_campo geodesico estendido localmente Lipschitz grupo}
Let $G$ be a Lie group endowed with a left-invariant strongly convex $C^0$-Finsler structure $F$.
Then the extended geodesic field $\mathcal{E}(x,\alpha)$ of $(G,F)$ is a locally Lipschitz vector field on $T^\ast G\backslash 0$.
\end{pro}

Now we will adapt Theorem \ref{Teorema Jessica e Ryuichi} for the case where $F$ is strongly convex.

Let $t \in (a,b) \rightarrow (x(t), \alpha(t)) \in T^\ast G \backslash 0$ be the unique Pontryagin extremal such that $(x(0), \alpha(0)) = (x_0, \alpha_0)$ defined on $(a,b)$ (see Proposition \ref{proposition_campo geodesico estendido localmente Lipschitz grupo}).
Then $(u(\mathfrak{a}(t)), \mathfrak{a}(t))$, where $\mathfrak{a}(t) = d\big(L^\ast_{x(t)}\big)_e(\alpha(t))$, is the Pontryagin extremal on $S_{F_e}[0,1] \times \mathfrak{g}^\ast \backslash \{0\}$ corresponding to $(x(t), \alpha(t))$  due to Theorem \ref{Teorema Jessica e Ryuichi}.
In particular, $(x(t), \alpha(t))$ is the Pontryagin extremal corresponding to the control $u(\mathfrak{a}(t))$.
Notice that $(u(\mathfrak{a}(t)), \mathfrak{a}(t))$ is determined by $\mathfrak{a}(t)$, which is the unique solution of the initial value problem 
\[
\dot{\mathfrak{a}} = \tilde{\mathcal{E}}(\mathfrak{a}), \quad \mathfrak{a}(0) = d\big(L^\ast_{x_0}\big)_e(\alpha_0),
\]
(recall that $\tilde{\mathcal{E}}$ is locally Lipschitz due to Theorem \ref{E til e Lipschitz}).
Therefore, if $F$ is strongly convex, then the correspondence in Theorem \ref{Teorema Jessica e Ryuichi} can be simplified to a mapping from the Pontryagin extremal $(x(t), \alpha(t))$ satisfying $(x(0),$ $\alpha(0)) = (x_0, \alpha_0)$ to the vertical part of the Pontryagin extremal $\mathfrak{a}(t)$ satisfying $\mathfrak{a}(0) = d\big(L^\ast_{x_0}\big)_e(\alpha_0)$.

For the inverse map, let $t \in (a,b) \rightarrow \mathfrak{a}(t) \in \mathfrak{g}^\ast \backslash \{0\}$ be the vertical part of the Pontryagin extremal satisfying $\mathfrak{a}(0) = d\big(L^\ast_{x_0}\big)_e(\alpha_0)$.
Then $(u(\mathfrak{a}(t)),\mathfrak{a}(t))$ is a Pontryagin extremal defined on $S_{F_e}[0,1] \times \mathfrak{g}^\ast \backslash \{0\}$ and Theorem \ref{Teorema Jessica e Ryuichi} states that there exist a unique Pontryagin extremal $(x(t),\alpha(t)) \in T^\ast G\backslash 0$ corresponding to the control $u(\mathfrak{a}(t))$ such that $\mathfrak{a}(t) = d\big(L^\ast_{x(t)}\big)_e(\alpha(t))$ and $x(t_0) = x_0$, which is the original Pontryagin extremal defined on $T^\ast G \backslash 0$.

Therefore, we have proved the following corollary of Theorem \ref{Teorema Jessica e Ryuichi}.

\begin{cor}
\label{corollary_Jessica Ryuichi}
Let $(G, F)$ be a Lie group endowed with a left-invariant strongly convex $C^0$-Finsler structure $F$.
Consider $t_0 \in \mathbb{R}$ and $x_0 \in G$.
Then the correspondences between Pontryagin extremals presented in Theorem \ref{Teorema Jessica e Ryuichi} are given by 
\begin{equation}
\label{equation_correspondencia Pontryagin grupo-algebra}
\begin{array}{c}
(x(t), \alpha(t))\text{ solution of }\left\{
\begin{array}{l}
(\dot{x},\dot{\alpha}) = \mathcal{E}(x,\alpha) \\
(x(t_0), \alpha(t_0)) = (x_0, \alpha_0) 
\end{array}
\right. \\
\updownarrow \\
\mathfrak{a}(t) \text{ solution of }\left\{
\begin{array}{l}
\dot{\mathfrak{a}} = \tilde{\mathcal{E}}{(\mathfrak{a})} \\
\mathfrak{a}(0) =  d\big(L^\ast_{x_0}\big)_e(\alpha_0).
\end{array}
\right.
\end{array}
\end{equation}
Moreover, we have that
\begin{equation*}
\mathfrak{a}(t) = d\big(L^\ast_{x(t)}\big)_e(\alpha(t))
\end{equation*} 
and $(x(t),\alpha(t))$ is the Pontryagin extremal corresponding to the control $\mathfrak{u}(t) = u(\mathfrak{a}(t))$.
\end{cor}

Now we are in position to prove the Lie group version of Theorem \ref{a(t) esta definida para toda a reta}.

\begin{teo}
\label{theorem_extremal de Pontryagin em grupos}
Let $G$ be a Lie group endowed with a left-invariant strongly convex $C^0$-Finsler structure $F$ and $(x_0, \alpha_0) \in T^\ast G \backslash 0$.
Then there exist a unique Pontryagin extremal $(x(t),\alpha(t))$ of $(G,F)$ defined on $\mathbb{R}$ such that $(x(0), \alpha(0)) = (x_0, \alpha_0)$.
\end{teo}

\begin{proof}
Let $t \in (a,b) \mapsto (\tilde{x}(t), \tilde{\alpha}(t))$ be a Pontryagin extremal such that $(\tilde{x}(0), \tilde{\alpha}(0))=(x_0, \alpha_0)$ and $t \in (a,b) \mapsto \tilde{\mathfrak{a}}(t)$ be the vertical part of the Pontryagin extremal corresponding to $(\tilde{x}(t), \tilde{\alpha}(t))$ given by \eqref{equation_correspondencia Pontryagin grupo-algebra}. 
Theorem \ref{a(t) esta definida para toda a reta} implies that there exists a unique extension $\mathfrak{a}(t)$ of $\tilde{\mathfrak{a}}(t)$ defined on $\mathbb{R}$.
Then the Pontryagin extremal $t \in \mathbb{R} \mapsto (x(t), \alpha(t))$ corresponding to $t \in \mathbb{R} \mapsto \mathfrak{a}(t)$ is the unique extension of $(\tilde{x}(t), \tilde{\alpha}(t))$ to $\mathbb{R}$. 
\end{proof}


\section{Mollifier smoothing of strongly convex asymmetric norms}

\label{Smoothing of a strongly convex asymmetric norm}

In this section, we define the mollifier smoothing $\tilde{F}_\varepsilon$ of a strongly convex asymmetric norm $F$, which is a version of the vertical mollifier smoothing defined in \cite{Fukuoka-Setti}.
We prove that for a sufficiently small $\varepsilon$, $\{F_\varepsilon\}$ and $F$ are strongly convex with respect to the same norm (see Theorem \ref{F_epsilon2 til são fortemente convexas com a mesma constante}). 

Let $F$ be a strongly convex asymmetric norm with respect to $\sqrt{\frac{\gamma}{2}}\Vert \cdot \Vert$ on $\mathbb{R}^n$. The mollifier smoothing of $F^2$ is defined by
\[
\left(\eta_\varepsilon \ast F^2\right) (y) = \int \eta_\varepsilon(z)F^2(y-z)dz
\]
for $\varepsilon \in (0, \tau),$ where 
\begin{eqnarray} \label{Constante tau}
	\tau \in \left( 0, \frac{r_m}{8nr_M} \right), 
\end{eqnarray}
and $r_M$ and $r_m$ are defined as in (\ref{constante r}) and (\ref{constante rho}), respectively.
Notice that $F^{-1}([r_M/2,2r_M]) \subset \mathcal A,$ where $\mathcal{A}$ is defined as (\ref{Conjunto A}).

The following lemma will be used to construct the sphere of the mollifier smoothing. 


\begin{lem} 
\label{lema propriedade de Fepsilon2} 
$\left(\eta_\varepsilon \ast F^2\right)$ has the following properties for every $\varepsilon \in \left(0, \tau\right):$
\begin{enumerate}
\item $\Hess \left(\eta_\varepsilon \ast F^2\right)_{y}(v,v) \geq \gamma \Vert v \Vert^2$ for every $y\in \mathrm{int} \mathcal A$ and $v\in \mathbb{V}$;
\item The {\em radial derivative} of $\left(\eta_\varepsilon \ast F^2\right)$ on $\mathcal A$ satisfies
\[
\frac{r_m^2}{32} < \frac{y^i}{\Vert y\Vert}\left(\eta_\varepsilon \ast F^2\right)_{y^i}(y)<\left( \frac{2r_mr_M^3+12r_M^4}{r_m^2}\right),\,\, \forall \, y \in \mathcal A;
\]
\item $\left(\eta_\varepsilon \ast F^2\right)(y)<r_M^2$ for every $y \in \mathbb{R}^n$ with $\|y\| \leq \frac{1}{2};$
\item $\left(\eta_\varepsilon \ast F^2\right)(y)>r_M^2$ for every $y \in \mathbb{R}^n$ with $\|y\| \geq \frac{2r_M}{r_m}.$
\end{enumerate}
In particular we have that $\left\{ \eta_\varepsilon \ast F^2 = r_M^2\right\}$ is a smooth hypersurface in $\mathbb{R}^n$ such that its radial projection onto  $S_{\|\cdot\|}[0,1]$ is a diffeomorphism.
\end{lem}

\begin{proof}

\

(1) 
By Lemma \ref{Suavização de função fortemente convexa}, we have that $\left(\eta_\varepsilon \ast F^2\right)$ is strongly convex with the same modulus $\gamma$ of $F^2$
and this item is consequence of Item (4) of Theorem \ref{Equivalencias - Fortemente convexa (Definicao Classica)}.

\

(2) Consider $y \in \mathcal{A}.$ Then
\begin{eqnarray*}
	\frac{r_m^2}{32}t & < & \left(\eta_\varepsilon \ast F^2\right) \left(y+t\frac{y}{\|y\|}\right)-\left(\eta_\varepsilon \ast F^2\right) (y) \\
	& = & \int_{B_{\|\cdot\|}[0,\varepsilon]} \eta_\varepsilon(z)\left[ F^2\left(y+t\frac{y}{\|y\|}-z\right) - F^2(y-z) \right] dz \\
	& = & \int_{B_{\|\cdot\|}[y,\varepsilon]} \eta_\varepsilon(y-z)\left[ F^2\left(z+t\frac{y}{\|y\|}\right) - F^2(z) \right] dz \\
	& < & \left( \frac{2r_mr_M^3+12r_M^4}{r_m^2}\right)t
\end{eqnarray*}
for $t\in (0,1)$, where the inequalities are due to Lemma \ref{proposicao F em Rn}. Thus, $\frac{r_m^2}{32}$  is a lower bound and $\frac{2r_mr_M^3+12r_M^4}{r_m^2}$ is an upper bound for $\frac{y^i}{||y||}(\eta_\varepsilon \ast F^2)_{y^i}(y)$.  

\

(3) Consider $y \in \mathbb{R}^n$ such that $\Vert y \Vert \leq 1/2$. 
Then
\begin{eqnarray}
\label{eqnarray_limitacao inferior suavizacao}
\left(\eta_\varepsilon \ast F^2\right)(y) & = & \int \eta_\varepsilon(z) F^2(y-z)dz
\end{eqnarray}
where $\Vert z\Vert \leq \frac{r_m}{8nr_M}$ on the support of $\eta_\varepsilon$.
Then $\Vert y-z\Vert^2 < 1$, what implies $F^2(y-z)< r^2_M$, and 
\[
\left(\eta_\varepsilon \ast F^2\right)(y) < r^2_M
\] 
holds due to \eqref{eqnarray_limitacao inferior suavizacao}.

(4) Consider $y \in \mathbb{R}^n$ such that $\Vert y\Vert \geq \frac{2r_M}{r_m}$. 
The proof follows analogously to Item (3). 
We have that $\Vert y-z\Vert > \frac{r_M}{r_m}$ for $z$ in the support of $\eta_\varepsilon$, what implies $F^2(y-z) > r_M^2$ and
\begin{eqnarray*}
\left(\eta_\varepsilon \ast F^2\right)(y) & = &\int \eta_\varepsilon(z) F^2(y-z)dz > r^2_M,
\end{eqnarray*}
what settles the proof.
\end{proof}


The next lemma is an intermediate step in order to prove several uniform convergences afterward. 


\begin{lem} 
\label{convergencia uniforme de F^epsilon^2} 
$\left(\eta_\varepsilon \ast F^2\right) \rightarrow F^2$ uniformly on compact subsets of  $\mathbb R^n$. 
\end{lem}
\begin{proof}
It follows from Theorem \ref{propriedades da aplicação mollifier}.
\end{proof}


Let $\varepsilon \in (0, \tau).$ Theorem \ref{teo_convexo e epigrafo} and Lemma \ref{lema propriedade de Fepsilon2} implies that 
\begin{equation}
\label{equation_bola aberta Ftilepsilon}
\left\{y \in \mathbb{V}:  (\eta_\varepsilon \ast F^2)(y) \leq r_M^2\right\}
\end{equation}
is bounded, closed and convex. 
Item (3) of Lemma \ref{lema propriedade de Fepsilon2} implies that \eqref{equation_bola aberta Ftilepsilon} contains the origin in its interior. 
Therefore, the positively homogeneous function
\begin{equation} \label{aplicacao F_epsilon til}
\tilde{F}_\varepsilon: \mathbb{R}^n \rightarrow {\mathbb R}
\end{equation}
such that $\tilde{F}_\varepsilon$ equals $r_M$ on $\{ \eta_\varepsilon \ast F^2 = r_M^2\}$  is an asymmetric norm because the triangle inequality follows from the convexity of \eqref{equation_bola aberta Ftilepsilon}. 


Let $(\theta^i) = (\theta^2,\ldots,\theta^n):S_\theta \rightarrow W_\theta \subset \mathbb R^{n-1}$ be a coordinate system on an open subset $S_\theta$ of the Euclidean sphere $S_{\Vert \cdot\Vert}[0,1]$. 
If $r>0$ is the radial coordinate in ${\mathbb R^n},$ then 
\begin{eqnarray*}
	(r,\theta):=(r,\theta^2,\ldots,\theta^n)
\end{eqnarray*}
is a coordinate system on the open cone $C(S_\theta)\subset \mathbb R^n$ of $S_\theta$ with the vertex at the origin. 

Let $y_0\in S_{\tilde{F}_\varepsilon}[0,r_M]$
and
\[
(r_0, \theta_0) := (r_0,\theta_0^2,\ldots,\theta_0^{n})
\]
be the coordinates of $y_0$.
Item (2) of  Lemma \ref{lema propriedade de Fepsilon2} implies that 
\[
\frac{\partial (\eta_\varepsilon \ast F^2)}{\partial r}(r_0,\theta_0) > 0
\] 
and the Implicit Function Theorem states that there exist a smooth function
\begin{eqnarray} \label{phi_epsilon}
	\phi_\varepsilon: W_\theta \rightarrow \left(\frac{1}{2},\frac{2r_M}{r_m}\right)
\end{eqnarray}
satisfying the following conditions:
\begin{enumerate}[(i)]
	\item For every $\theta \in W_\theta$, $\phi_\varepsilon(\theta)$ is the unique value such that 
	\[
	\left(\eta_\varepsilon \ast F^2\right)(\phi_\varepsilon(\theta),\theta)=r_M^2;
	\]
	\item
	\[
	\frac{\partial \phi_\varepsilon(\theta)}{\partial \theta^i} 
	= - \frac{\frac{\partial \left(\eta_\varepsilon \ast F^2\right)(\phi_\varepsilon(\theta),\theta) }{\partial \theta^i}}{\frac{\partial \left(\eta_\varepsilon \ast F^2\right)(\phi_\varepsilon(\theta),\theta) }{\partial r}}
	\]
	for every $\theta \in W_\theta$.
\end{enumerate}

Therefore, the expression of $\tilde{F}_\varepsilon^{ \ 2}$ in terms of $(r,\theta) \in (0,\infty) \times W_\theta$ is given by
\begin{equation} 
	\label{F_epsilon^2 til em coordenadas esfericas} 
	\tilde{F}_\varepsilon^{ \ 2} (r,\theta) = \frac{r^2 .r_M^2}{\phi_\varepsilon^{\ 2}(\theta)}
\end{equation}
because $\tilde{F}_\varepsilon^{\ 2}$ is positively homogeneous of degree 2 with respect to the variable $r$. Thus,
\begin{equation} 
	\label{F_epsilon til em coordenadas esfericas} 
	\tilde{F}_\varepsilon (r,\theta) = \frac{r .r_M}{\phi_\varepsilon(\theta)}
\end{equation}
for every  $(r,\theta) \in (0,\infty) \times W_\theta.$


\begin{pro} \label{F_epsilon é diferenciável} The function $\tilde{F}_\varepsilon: \mathbb{R}^n \rightarrow  \mathbb{R}$ defined in (\ref{aplicacao F_epsilon til}) is smooth in $\mathbb{R}^n\setminus\{0\}$ for every $\varepsilon \in (0, \tau).$
\end{pro}

\begin{proof} It follows from \eqref{F_epsilon til em coordenadas esfericas} and the fact that $\phi_\varepsilon$ defined by (\ref{phi_epsilon}) is smooth.
	
\end{proof}


\begin{teo}
	\label{F_epsilon^2 til converge uniforme}
	$\tilde{F}_\varepsilon^{ \ 2} \rightarrow F^2 $ uniformly on compact subsets of $\mathbb{R}^n$.
\end{teo}

\begin{proof}
The proof is identical to the proof of Proposition 5 of \cite{Fukuoka-Setti} with the variable $x$ dropped.
\end{proof}


\begin{cor}
	\label{F_epsilon til converge uniforme}
	$\tilde{F}_\varepsilon \rightarrow F$ uniformly on compact subsets of $\mathbb{R}^n$.
\end{cor}

\begin{proof} It follows from Theorem \ref{F_epsilon^2 til converge uniforme} and the fact that the square root function is uniformly continuous on compact subsets of $[0,\infty)$.
\end{proof}


\begin{pro} \label{condicao para a hessiana de G epsilon} 
Let $f,g:\mathbb{R}^n\setminus\{0\} \rightarrow \mathbb{R}$ be smooth convex functions with the same level hypersurfaces on the level $r,$ with  $r$  a regular value of $f$ and $g.$ 
Let $\lambda$ be the non-zero continuous function on $f^{-1}(r)$ defined by $\nabla g(y) = \lambda(y) \nabla f(y)$.
Suppose that $f$ is strongly convex with respect to $\sqrt{\frac{\gamma}{2}}\Vert \cdot\Vert$ and that $f^{-1}(r)$ is compact. 
Then $\lambda$ is a positive function.
Moreover,
	\begin{eqnarray*}
		(\Hess g)_{y_0} (v, v) \geq \lambda_m \gamma \Vert v \Vert^2
	\end{eqnarray*}
	for every $y_0 \in f^{-1}(r)$ and $v \in \ker df_{y_0}$, where
	\begin{eqnarray*}
		\lambda_m = \min_{y \in f^{-1}(r)} \lambda(y).
	\end{eqnarray*}
\begin{proof}
Firstly notice that since $f^{-1}(r) = g^{-1}(r)$ and $r$ is a regular value of $f$ and $g,$ then $\nabla g(y) = \lambda(y) \nabla f(y)$ for every $y \in f^{-1}(r),$ where $\lambda$ is a non-zero continuous function defined on $f^{-1}(r).$
Consider the coordinate system $(\tilde{y}^1, \ldots, \tilde{y}^n )$ corresponding to an orthonormal basis $\{v_1, v_2, \ldots, v_n\}$ with the origin at $(0,\ldots, 0)$, where $\ker df_{y_0} = \spann \{v_2, \ldots, v_n\}$.
Without loss of generality, we can choose an orientation for $v_1$ such that $(\nabla f)_{y_0} =(a,0,\ldots,0)$ with  $a > 0,$ that is, $a=\Vert (\nabla f)_{y_0}\Vert$. 
		
Write $y_0=\tilde{y}^i_0v_i.$  By Implicit Function Theorem there are a ball $B=B((\tilde{y}^2_0,\ldots,\tilde{y}^n_0),\delta) \subset \spann\{v_2, \ldots, v_n\}$ and an interval $J=(\tilde{y}^1_0-\mu,\tilde{y}^1_0+\mu)$ with the following properties:
	\begin{itemize}
		\item[(i)] {$ \overline{J} \times B \subset \mathbb{R}^n$ and $\frac{\partial f}{\partial \tilde{y}^1}(\tilde{y}^1,\ldots,\tilde{y}^n) \neq 0 $ for every $(\tilde{y}^1,\ldots,\tilde{y}^n) \in  \overline{J} \times B;$} 
		\item[(ii)] {For every $(\tilde{y}^2,\ldots,\tilde{y}^n) \in B, $  there exists a unique $\tilde{y}^1 = \phi (\tilde{y}^2,\ldots,\tilde{y}^n) \in J$ such that $f(\tilde{y}^1,\tilde{y}^2,\ldots,\tilde{y}^n) =  f(\phi(\tilde{y}^2,\ldots,\tilde{y}^n),\tilde{y}^2,\ldots,\tilde{y}^n) = r.$}
			
The function $\phi:B \longrightarrow J$  is  smooth and the partial derivatives of $f(\phi (\tilde{y}),\tilde{y})=r$ at each point $\hat{y}=(\tilde{y}^2,\ldots,\tilde{y}^n)\in B$  satisfies
		\begin{eqnarray}
		\label{align_derivada primeira fphiytil} 
			\frac{\partial f}{\partial \tilde{y}^i}(\phi(\hat{y}),\hat{y}) + \frac{\partial f}{\partial \tilde{y}^1}(\phi(\hat{y}),\hat{y}) \frac{\partial \phi}{\partial \tilde{y}^i}(\hat{y}) =0,
		\end{eqnarray}
and
	\begin{align}
	\label{align_derivada segunda fphiytil}
	& \frac{\partial^2 f}{\partial (\tilde{y}^1)^2}(\phi(\hat{y}),\hat{y})\frac{\partial \phi}{\partial \tilde{y}^j}(\hat{y})\frac{\partial \phi}{\partial \tilde{y}^i}(\hat{y}) + \frac{\partial^2 f}{\partial \tilde{y}^j \partial \tilde{y}^1}(\phi(\hat{y}),\hat{y})\frac{\partial \phi}{\partial \tilde{y}^i}(\hat{y}) \\
	 + & \frac{\partial f}{\partial \tilde{y}^1}(\phi(\hat{y}), \hat{y}) \frac{\partial^2 \phi}{\partial \tilde{y}^j \partial \tilde{y}^i}(\hat{y}) 
	 + \frac{\partial^2 f}{\partial \tilde{y}^1 \partial \tilde{y}^i}(\phi(\hat{y}),\hat{y})\frac{\partial \phi}{\partial \tilde{y}^j}(\hat{y}) \nonumber \\
	 + & \frac{\partial^2 f}{\partial \tilde{y}^j \partial \tilde{y}^i}(\phi(\hat{y}),\hat{y}) = 0 \nonumber
	\end{align}
for every $i,j=2,\ldots,n.$
	\end{itemize}
		
Applying \eqref{align_derivada primeira fphiytil} at $\hat{y}_0 := (\tilde{y}^2_0,\ldots,\tilde{y}^n_0), $ we obtain  
$$\frac{\partial \phi}{\partial \tilde{y}^i}(\hat{y}_0)=0,$$
because $\frac{\partial}{\partial \tilde{y}^i}  \in \ker df_{y_0}, i=2,\ldots,n.$

Evaluating \eqref{align_derivada segunda fphiytil} at $\hat{y}_0$ we have
	\begin{eqnarray*}
	 \frac{\partial^2 f}{\partial \tilde{y}^j \partial \tilde{y}^i}(\phi(\hat{y}_0),\hat{y}_0)= - \frac{\partial f}{\partial \tilde{y}^1 }(\phi(\hat{y}_0),\hat{y}_0)\frac{\partial^2 \phi}{\partial \tilde{y}^j \partial \tilde{y}^i}(\hat{y}_0), \,\, i,j=2,\ldots, n.
	\end{eqnarray*}
Using the same reasoning for $g$, we have that
	\begin{eqnarray}
		& & \frac{\partial^2 g}{\partial \tilde{y}^j \partial \tilde{y}^i}(\phi(\hat{y}_0),\hat{y}_0) = - \frac{\partial g}{\partial \tilde{y}^1 }(\phi(\hat{y}_0),\hat{y}_0)\frac{\partial^2 \phi}{\partial \tilde{y}^j \partial \tilde{y}^i}(\hat{y}_0) \nonumber \\
		& = & \frac{\frac{\partial g}{\partial \tilde{y}^1 }(\phi(\hat{y}_0),\hat{y}_0)}{\frac{\partial f}{\partial \tilde{y}^1 }(\phi(\hat{y}_0),\hat{y}_0)}\frac{\partial^2 f}{\partial \tilde{y}^j \partial \tilde{y}^i}(\phi(\hat{y}_0),\hat{y}_0).\label{eqnarray_relacao entre hessianas}
	\end{eqnarray}
Since
\[
\nabla g(y_0) = \left(\frac{\partial g}{\partial \tilde{y}^1}(y_0), 0, \ldots, 0 \right)
\]
and
\[
\nabla f(y_0) = \left(\frac{\partial f}{\partial \tilde{y}^1}(y_0), 0, \ldots, 0 \right),
\]
we have that $\nabla g(y_0) = \lambda (y_0) \nabla f(y_0)$, where $\lambda (y_0) = \frac{\frac{\partial g}{\partial \tilde{y}^1}(y_0)}{\frac{\partial f}{\partial \tilde{y}^1}(y_0)}$. 
Therefore,
	\begin{eqnarray*}
		& (\Hess g)_{y_0}(v, v) 
		= \lambda(y_0) (\Hess f)_{y_0}(v, v)
	\end{eqnarray*}
for every $y_0 \in f^{-1}(r)$ and $v \in \ker df_{y_0}$ due to \eqref{eqnarray_relacao entre hessianas}. 
Moreover, $r$ is a regular value of $g$ and $g$ is convex, what implies that $\lambda(y_0) > 0$ for every $y_0 \in f^{-1}(r)$. 
Finally, the strong convexity of $f$ implies 
\[
(\Hess g)_{y_0}(v, v) \geq \lambda_m \gamma \Vert v \Vert^2
\]
for every $v \in \ker df_{y_0}$, what settles the proposition. 
	\end{proof}
\end{pro}


Now we apply Proposition \ref{condicao para a hessiana de G epsilon} to $f=\eta_{\varepsilon}\ast F^2$, $g=\tilde{F}_{\varepsilon}{}^2$ and $r=r_M$ in order to prove the strong convexity of $\tilde{F}_{\varepsilon}{}^2$ and determine its modulus of convexity. 
The strong convexity of $\eta_{\varepsilon}\ast F^2$  follows from Lemma \ref{Suavização de função fortemente convexa} and the convexity of $\tilde{F}_{\varepsilon}{}^2$ holds due to Theorem \ref{teo_composicao de convexas}, the convexity of $\tilde{F}_\varepsilon$ and the relation $\tilde{F}_{\varepsilon}{}^2 = \varphi \circ \tilde{F}_\varepsilon$, where $\varphi$ is the square function defined on $[0,\infty)$.
Set
\begin{eqnarray} \label{constante c_epsilon}
	r_{\varepsilon,m} = \min\left\{\tilde{F}_\varepsilon(y) : y \in S_{\Vert \cdot \Vert }[0,1] \right\}
\end{eqnarray}
and
\begin{eqnarray} \label{constante lambda_m,epsilon}
	\lambda_{\varepsilon,m} = \min_{y \in \left(\eta_\varepsilon \ast F^2\right)^{-1}\left(r_M^2\right)} \lambda_\varepsilon (y),
\end{eqnarray}
where $\lambda_{\varepsilon}(y)$ is defined through the expression 
$\nabla \tilde{F}_\varepsilon^{ \ 2} (y) = \lambda_{\varepsilon} (y) \nabla \left(\eta_\varepsilon \ast F^2\right)(y).$

\begin{teo} \label{F_epsilon é uma norma assimétrica fortemente convexa} Let $F: \mathbb{R}^n \rightarrow  [0,\infty)$ be a strongly convex asymmetric norm with respect to $\sqrt{\frac{\gamma}{2}}\Vert \cdot\Vert$.
Its mollifier smoothing $\tilde{F}_\varepsilon: \mathbb{R}^n \rightarrow  [0,\infty)$ defined in (\ref{aplicacao F_epsilon til}) is a strongly convex asymmetric norm with respect to $\sqrt{\frac{\gamma_\varepsilon}{2}} \Vert \cdot \Vert$ for every $\varepsilon \in (0, \tau),$ where
	\begin{eqnarray*} \label{gamma_epsilon}
		\gamma_\varepsilon 
		= \min\left(2r_{\varepsilon,m}^2,  \gamma \lambda_{\varepsilon, m}\right),
	\end{eqnarray*}
and the constants $r_{\varepsilon,m}$ and $\lambda_{\varepsilon, m}$ are defined as (\ref{constante c_epsilon}) and (\ref{constante lambda_m,epsilon}) respectively.
In particular, $\tilde{F}_{\varepsilon}$ is a Minkowski norm.
\end{teo}

\begin{proof}
	We have that $F^2$ is strongly convex with modulus $\gamma,$ then by Lemma \ref{Suavização de função fortemente convexa} we have that $\left(n_\varepsilon \ast F^2\right)$ also is strongly convex with modulus $\gamma.$	
	The level hypersurface of $\left(n_\varepsilon \ast F^2\right)$ and $\tilde{F}_\varepsilon^{\ 2}$ at $r_M^2$ are the same, $\left(\eta_\varepsilon \ast F^2\right)^{-1}\left(r_M^2\right)$ is compact and $\tilde{F}_\varepsilon$ is convex because it is an asymmetric norm. 
	Then it follows from the Proposition \ref{condicao para a hessiana de G epsilon} that 
	\begin{equation}
	\label{equation_suavizacao fortemente convexa na esfera}
	\big(\tilde{F}_\varepsilon^{\ 2}\big)_{y^iy^j}(y)v^i v^j \geq  \lambda_{\varepsilon, m} \gamma \Vert v \Vert^2,
	\end{equation}
	for every $y \in \left(n_\varepsilon \ast F^2\right)^{-1}\left(r_M^2\right)$ and for every $v \in \ker d\left(n_\varepsilon \ast F^2\right)_y.$
	
	Now, we will prove that
	$$\big(\tilde{F}_\varepsilon^{\ 2}\big)_{y^iy^j}(y)v^i v^j \geq  \gamma_\varepsilon \Vert v \Vert^2,$$
	for every $y \in \left(n_\varepsilon \ast F^2\right)^{-1}\left(r_M^2\right)$ and every $v \in \mathbb{R}^n.$ The general case with $y \in \mathbb{R}^n\setminus\{0\}$ follows from the fact that the function $\tilde{F}_\varepsilon^{\ 2}$ is positively homogeneous of degree 2. 
	Given $y \in \left(n_\varepsilon \ast F^2\right)^{-1}\left(r_M^2\right)$ and $v \in \mathbb{R}^n,$ write $v = \mu y + w,$ where $\mu \in \mathbb{R}$ and $w \in \ker d\big(\tilde{F}_\varepsilon^{\ 2}\big)_y$. 
	Then 
	\begin{eqnarray*}
		\big(\tilde{F}_\varepsilon^{\ 2}\big)_{y^iy^j}(y) v^i v^j & = & \mu^2 \big(\tilde{F}_\varepsilon^{\ 2}\big)_{y^iy^j}(y) y^i y^j + \mu \big(\tilde{F}_\varepsilon^{\ 2}\big)_{y^iy^j}(y) y^i w^j + \\
		& & + \mu \big(\tilde{F}_\varepsilon^{\ 2}\big)_{y^iy^j}(y) w^i y^j + \big(\tilde{F}_\varepsilon^{\ 2}\big)_{y^iy^j}(y) w^i w^j \\
		& = & \mu^2 \big(\tilde{F}_\varepsilon^{\ 2}\big)_{y^iy^j}(y) y^i y^j + \big(\tilde{F}_\varepsilon^{\ 2}\big)_{y^iy^j}(y) w^i w^j
	\end{eqnarray*}
due to Proposition \ref{pro_Minkowski Euler} and because $(\tilde{F}_{\varepsilon})_{y^i}(y)w^i = \left< (\nabla \tilde{F}_{\varepsilon})(y),w\right> = 0$.
Considering \eqref{constante c_epsilon}, \eqref{equation_suavizacao fortemente convexa na esfera} and Proposition \ref{pro_Minkowski Euler} again, we have
	\begin{eqnarray*}
		\big(\tilde{F}_\varepsilon^{\ 2}\big)_{y^iy^j}(y) v^i v^j & \geq & 2 \mu^2 \tilde{F}_\varepsilon^{\ 2}(y) + \lambda_{\varepsilon, m} \gamma \Vert w \Vert^2 \\
		& = & 2 \tilde{F}_\varepsilon^{\ 2}(\mu y) +  \lambda_{\varepsilon, m} \gamma \Vert w \Vert^2 \\
		& \geq & 2 r_{\varepsilon, m}^2 \Vert \mu y \Vert^2  + \lambda_{\varepsilon, m} \gamma \Vert w \Vert^2 \\
		& \geq & \min\left(2 r_{\varepsilon, m}^2, \lambda_{\varepsilon, m}\gamma\right) \Vert v \Vert^2 \\
		& = & \gamma_\varepsilon \Vert v \Vert^2.
	\end{eqnarray*}
Therefore, $\tilde{F}_\varepsilon$ is a strongly convex asymmetric norm with respect to $\sqrt{\frac{\gamma_\varepsilon}{2}} \Vert \cdot \Vert$ for every $\varepsilon \in (0, \tau).$

Finally, in order to see that $\tilde{F}_{\varepsilon}$ is a Minkowski norm, notice that its smoothness in $\mathbb{V} \backslash \{0\}$ is consequence of the smoothness of $\phi_\varepsilon$ and \eqref{F_epsilon til em coordenadas esfericas}, and
the result is consequence of Proposition \ref{proposition_norma limitante inferior}.
\end{proof}

	
In Theorem \ref{F_epsilon2 til são fortemente convexas com a mesma constante},  we will prove that $\{\tilde{F}_\varepsilon\}$ and $F$ are strongly convex with respect to the same norm $\sqrt{\frac{\gamma_0}{2}}\Vert \cdot\Vert$. 
In order to do so, 
$\varepsilon$ must vary in the interval $(0, \tau_0),$ where $\tau_0 \in (0, \tau]$  is given by Lemma \ref{c_epsilon convege para c}.
 
Until the end of this work, we will need progressively smaller intervals of variation of $\varepsilon$. 
The choice of these intervals will prioritize clarity over optimality. 
Therefore, the intervals used along the proofs can vary up and down.


\begin{lem} \label{c_epsilon convege para c}
	There exists $\tau_0 \in (0, \tau]$	such that $r_{\varepsilon, m} \in (r_m/2, 3r_m/2)$ for every $\varepsilon \in (0, \tau_0)$, where $r_m$ is defined in \eqref{constante rho}. 
	
\end{lem}

\begin{proof} It follows from the fact that the minimum function is continuous and from Corollary \ref{F_epsilon til converge uniforme}.
\end{proof}


\begin{teo} \label{F_epsilon2 til são fortemente convexas com a mesma constante} The functions $F^2$ and $\tilde{F}_\varepsilon^{\ 2}:\mathbb{R}^n \rightarrow  [0,\infty)$ are strongly convex with modulus
	\begin{eqnarray} \label{gamma_0}
		\gamma_0 = \min\left(\frac{r_{m}^2}{2},  \frac{\gamma r_m^4}{8r_mr_M^3 + 48r_M^4}\right)  
	\end{eqnarray}
	for every $\varepsilon \in (0, \tau_0),$ where $\tau_0 \in (0, \tau]$ is given by Lemma \ref{c_epsilon convege para c}.
	Equivalently $F$ and $\tilde{F}_{\varepsilon}$ are strongly convex asymmetric norms with respect to $\sqrt{\frac{\gamma_0}{2}}\Vert \cdot\Vert$.
\end{teo}

\begin{proof} The function $F^2$ is strongly convex with modulus $\gamma_0$ because it is strongly convex with modulus $\gamma>\gamma_0$.

We saw in Theorem \ref{F_epsilon é uma norma assimétrica fortemente convexa} that every $\tilde{F}_\varepsilon^{\ 2}$ is strongly convex with modulus 
	\begin{eqnarray*}
	\gamma_\varepsilon =   \min\left(2r_{\varepsilon,m}^2, \gamma \lambda_{\varepsilon, m}\right),
	\end{eqnarray*}
	for every $\varepsilon \in (0, \tau).$ Let us show that $\gamma_\varepsilon > \gamma_0$ for every $\varepsilon \in (0, \tau_0).$ 
	We will prove that $2r_{\varepsilon,m}^2 > \gamma_0$ and $\gamma \lambda_{\varepsilon, m} > \gamma_0$ for every $\varepsilon \in (0, \tau_0).$
	
	$\bullet$ $2r_{\varepsilon,m}^2 > \gamma_0:$ By Lemma \ref{c_epsilon convege para c} we have that $r_{\varepsilon,m}^2> \frac{r_{m}^2}{4}$ for every $\varepsilon \in (0, \tau_0)$. 
	Then 
	\begin{eqnarray*}
		2r_{\varepsilon,m}^2 > \gamma_0,
	\end{eqnarray*}
	for every $\varepsilon \in (0, \tau_0).$
	
	$\bullet$ $\gamma \lambda_{\varepsilon, m} > \gamma_0:$ 
Let us find a positive lower bound for all $\lambda_{\varepsilon, m}$, where $\varepsilon \in (0,\tau)$.

We know that $d(\tilde{F}_{\varepsilon}{}^2)(y)$ and $d(\eta_\varepsilon \ast F^2)(y)$ have the same kernel for every $y \in \left(\eta_\varepsilon\ast F^2\right)^{-1}\left(r_M^2\right)$, what implies that one is a scalar multiple of the other. 
Since $\nabla \big(\tilde{F}_\varepsilon^{\ 2}\big)(y) = \lambda_\varepsilon(y) \nabla \left(\eta_\varepsilon\ast F^2\right)(y)$, we have that $d(\tilde{F}_{\varepsilon}{}^2)(y) =  \lambda_{\varepsilon}(y).d(\eta_\varepsilon \ast F^2)(y)$, and in particular their radial derivatives satisfy
	\begin{eqnarray}
	\label{eqnarray_lambda varepsilon}
		\lambda_\varepsilon(y) & = & \frac{\frac{y^i}{\Vert y\Vert}\frac{\partial \big(\tilde{F}_\varepsilon^{\ 2}\big)(y)}{\partial y^i}}{\frac{y^i}{\Vert y\Vert}\frac{\partial \left(\eta_\varepsilon \ast F^2\right)(y)}{\partial y^i}} =\frac{\frac{\partial \tilde{F}_{\varepsilon}{}^2(r,\theta)}{\partial r}}{\frac{\partial (\eta_{\varepsilon}\ast F^2)(r,\theta)}{\partial r}}.
	\end{eqnarray}
Using the expression for $\tilde{F}_\varepsilon^{\ 2}$ given in (\ref{F_epsilon^2 til em coordenadas esfericas}) and the fact that $\phi_\varepsilon(\theta)$  
and $r=\Vert y\Vert$ are in the interval
$ (1/2, 2r_M/r_m)$ because $y\in \mathcal{A}$, we obtain 
	\begin{eqnarray}
	\label{eqnarray_limitante inferior derivada radial Fquadrado}
		\frac{\partial \tilde{F}_\varepsilon^{\ 2}(r, \theta)}{\partial r} & = & \frac{2r .r_M^2}{\phi_\varepsilon^{\ 2}(\theta)}
	 \geq \frac{2 \cdot \frac{1}{2} \cdot r_M^2}{\left(\frac{2r_M}{r_m}\right)^2}  = \frac{r_m^2}{4}.
	\end{eqnarray}
Using \eqref{eqnarray_limitante inferior derivada radial Fquadrado} and the second inequality of Item (2) of Lemma \ref{lema propriedade de Fepsilon2} on \eqref{eqnarray_lambda varepsilon}, we get
	\begin{eqnarray*}
		\lambda_\varepsilon(y) 
		 > \frac{r_m^2}{4} \cdot \frac{r_m^2}{2r_mr_M^3+12r_M^4} = \frac{r_m^4}{8r_mr_M^3 + 48r_M^4}.
	\end{eqnarray*}
	Note that the last expression does not depend on $y$ or on $\varepsilon,$ so
	\begin{eqnarray*}
		\lambda_{\varepsilon, m} & > & \frac{r_m^4}{8r_mr_M^3 + 48r_M^4}>0,
	\end{eqnarray*}
	for every $\varepsilon \in (0, \tau).$ Then
	\begin{eqnarray*}
		\gamma  \lambda_{\varepsilon, m} >  \frac{\gamma r_m^4}{8r_mr_M^3 + 48r_M^4} \geq \gamma_0,
	\end{eqnarray*}
	for every $\varepsilon \in (0, \tau).$
	
	Therefore, the functions $\tilde{F}_\varepsilon^{\ 2}$ are strongly convex with modulus
		\begin{eqnarray*}
		\gamma_0 = \min\left(\frac{r_m^2}{2},  \frac{\gamma r_m^4}{8r_mr_M^3 + 48r_M^4}\right)
	\end{eqnarray*}
	for every $\varepsilon \in (0, \tau_0).$

	The last assertion is consequence of Theorem \ref{Fortemente convexa - Equivalencia}.
\end{proof}


\section{Convergence of left-invariant mollifier smoothings on Lie groups}
\label{Smoothing F}

Until the end of this work, $G$ is a Lie group endowed with a left-invariant strongly convex $C^0$-Finsler structure $F$. Recall that $F_e = F\vert_{\mathfrak{g}}$.
Consider the mollifier smoothing $\tilde{F}_{\varepsilon}$ of $F_e$ on $(\mathfrak{g}, \left< \cdot, \cdot \right>) \cong \mathbb{R}^n$ defined by \eqref{aplicacao F_epsilon til}. It is a Minkowski norm due to Theorem \ref{F_epsilon é uma norma assimétrica fortemente convexa} and the strong convexity of $F_e$.
The left-invariant mollifier smoothing $F_\varepsilon:TG \rightarrow \mathbb{R}$ of $F$ is defined as
\begin{equation}
\label{equation_define F_varespilon}
	F_\varepsilon(x,y) = \tilde{F}_\varepsilon((dL_{x^{-1}})_x(y)).
\end{equation}
The smoothness of $F_\varepsilon$ on $TG\backslash 0$ is consequence of its left-invariance and of $\tilde{F}_\varepsilon$ being a Minkowski norm.
Therefore, $F_\varepsilon$ is a left-invariant Finsler structure of $G$ for every $\varepsilon \in (0,\tau)$, where $\tau$ is defined by \eqref{Constante tau}, which we call {\em left-invariant mollifier smoothing of $F$}.
In this short section we prove the uniform convergence $F_\varepsilon \rightarrow F$ on compact subsets of $TG$ as $\varepsilon \rightarrow 0$.

In the next result we will use the diffeomorphism $\varphi:G\times \mathfrak{g} \rightarrow TG $ defined by
\begin{equation*} \label{Trivialização X}
	\varphi(x,y)= (x,(dL_x)_e(y)).
\end{equation*}
Its inverse $\varphi^{-1}: TG \rightarrow G\times \mathfrak{g}$ is given by
\begin{equation*} \label{Inversa da Trivialização X}
	\varphi^{-1}(x,y)=(x,(dL_{x^{-1}})_x(y)).    
\end{equation*}


\begin{teo} \label{Convergência uniforme de F epsilon} $F_\varepsilon\rightarrow F$ uniformly on compact subsets of $TG$ as $\varepsilon$ converges to zero. 
\end{teo}

\begin{proof} Let $K_{TG}$ be a compact subset of $TG.$ Given $\mu>0$ we must find $\tilde{\tau} \in (0, \tau]$ such that
	\begin{equation*}
		|F_\varepsilon(x,y) - F(x,y)|  < \mu,
	\end{equation*}
	for every $(x,y) \in K_{TG}$ whenever $\varepsilon \in (0, \tilde{\tau}).$
	
	Since $K_{TG} \subset TG$  is compact, then $K_{\mathfrak{g}} = \pi_2(\varphi^{-1}(K_{TG}))$ is a compact subset of $\mathfrak{g}\cong  \mathbb{R}^n,$ where $\pi_2:G\times \mathfrak{g} \rightarrow \mathfrak{g}$ is given by $\pi_2(x,y)=y.$
	
	Note that
	\begin{equation*}
		F(x,y) = F_e((dL_{x^{-1}})_x (y)) = F_e(\pi_2(\varphi^{-1}(x,y)))
	\end{equation*}
	and
	\begin{equation*}
		F_\varepsilon(x,y) = \tilde{F}_\varepsilon((dL_{x^{-1}})_x (y)) = \tilde{F}_\varepsilon(\pi_2(\varphi^{-1}(x,y))).
	\end{equation*}
	So
	\begin{equation*}
		|F_\varepsilon(x,y) - F(x,y)|  = \left|\tilde{F}_\varepsilon(\pi_2(\varphi^{-1}(x,y))) - F_e(\pi_2(\varphi^{-1}(x,y)))\right|.
	\end{equation*}
	
	By Corollary \ref{F_epsilon til converge uniforme}, for each $\mu>0,$ there exists $\tilde{\tau}\in (0, \tau]$ such that 
	\begin{equation*}
		\left|\tilde{F}_\varepsilon(z) - F_e(z)\right|  < \mu,
	\end{equation*}
	for every $z \in K_{\mathfrak{g}}$ whenever $\varepsilon \in (0, \tilde{\tau})$, in particular for $z=\pi_2(\varphi^{-1}(x,y)).$ This shows that $F_\varepsilon$ converges uniformly to $F$ on $K_{TG}$ as $\varepsilon$ converges to zero.
\end{proof}


\section{The uniform convergences $u_\varepsilon \rightarrow u$ and $\tilde{\mathcal{E}}_\varepsilon \rightarrow \tilde{\mathcal{E}}$ on compact subsets of $\mathfrak{g}^\ast \backslash \{0\}$}

\label{section_convergences and estimates u and E}

Pontryagin extremals on Lie groups endowed with a left-invariant strongly convex $C^0$-Finsler structure can be studied through the initial value problem of its Lie algebra version:
\[
\dot{\mathfrak{a}} = \tilde{\mathcal{E}}(\mathfrak{a}) = \mathfrak{a}([u(\mathfrak{a}), \cdot]), \quad \mathfrak{a}(0) = \mathfrak{a}_0
\]
(see Corollary \ref{corollary_Jessica Ryuichi}).
Let $u_\varepsilon: \mathfrak{g}^\ast \backslash \{0\} \rightarrow S_{\tilde{F}_\varepsilon}[0,1]$ and $\tilde{\mathcal{E}}_\varepsilon: \mathfrak{g}^\ast \backslash \{0\} \rightarrow\mathfrak{g}^\ast$ be the corresponding objects of the left-invariant mollifier smoothing $(G,\tilde{F}_\varepsilon)$ of $(G,F)$.
In this section, we study the uniform convergences $u_\varepsilon \rightarrow u$ and $\tilde{\mathcal{E}}_\varepsilon \rightarrow \tilde{\mathcal{E}}$ on compact subsets of $\mathfrak{g}^\ast \backslash \{0\}$.
In addition, we prove that given a compact subset $\Kgzero \subset \mathfrak{g}^\ast \backslash \{0\}$, there exist $\tau_4>0$ such that the families of mappings $\{u\}\cup \{u_\varepsilon\}_{\varepsilon\in (0,\tau_4)}$ and $\{\tilde{\mathcal{E}}\} \cup \{\tilde{\mathcal{E}_\varepsilon}\}_{\varepsilon \in (0,\tau_4)}$ are Lipschitz maps on $\Kgzero$, with each family having the same Lipschitz constants.

Recall that $u(\mathfrak{a})$ is the unique point that maximizes $\mathfrak{a}$ in $S_{F_e}[0,1]$, that is,
\[
u(\mathfrak{a}) 
= \frac{d\big(F_e\big)_{\ast}{}^{2}(\mathfrak{a})}{F_e\big(d\big({F}_e\big)_\ast^{\ 2}(\mathfrak{a})\big)} 
= \frac{d\big(F_e\big)_{\ast}{}^{2}(\mathfrak{a})}{2\big(F_e\big)_\ast(\mathfrak{a})}
\]
(see Proposition \ref{Formula u alpha}).

In the same way, we have that
	\begin{eqnarray} \label{Expressão de u_epsilon(alpha)}
	u_\varepsilon(\mathfrak{a}) = \frac{d\big(\tilde{F}_\varepsilon\big)_\ast^{\ 2}(\mathfrak{a})}{\tilde{F}_\varepsilon\big(d\big(\tilde{F}_\varepsilon\big)_\ast^{\ 2}(\mathfrak{a})\big)} = \frac{d\big(\tilde{F}_\varepsilon\big)_\ast^{\ 2}(\mathfrak{a})}{2\big(\tilde{F}_\varepsilon\big)_\ast(\mathfrak{a})}.
\end{eqnarray}


\begin{pro}\label{Convergencia norma dual}  $\big(\tilde{F}_\varepsilon\big)_\ast \rightarrow (F_e)_\ast$ uniformly on compact subsets of $\mathfrak{g}^*\setminus\{0\}$ when $\varepsilon \rightarrow 0.$
\end{pro}

\begin{proof}
	Let $K_{\mathfrak{g}^*\setminus \{0\}}$ a compact subset of $\mathfrak{g}^*\setminus \{0\}.$ We will prove that $\big(\tilde{F}_\varepsilon\big)_\ast \rightarrow (F_e)_\ast$ uniformly on $K_{\mathfrak{g}^*\setminus \{0\}}$ when $\varepsilon \rightarrow 0.$ 
	
	Given $\mu>0,$ we must find $\tilde{\tau} \in (0, \tau]$ such that 
	\begin{eqnarray*}
		\big\vert\big(\tilde{F}_\varepsilon\big)_\ast(\mathfrak{a}) - (F_e)_\ast(\mathfrak{a})\big\vert < \mu, \, \forall \, \mathfrak{a} \in K_{\mathfrak{g}^*\setminus \{0\}} \mbox{ and } \forall \, \varepsilon \in (0,\tilde{\tau}).
	\end{eqnarray*}
	
	Since $S_{F_e}[0,1]$ is compact, the functions $F_e$ and $\tilde{F}_\varepsilon$ are continuous and $\tilde{F}_\varepsilon(y) \neq 0$ for every $y \in S_{F_e}[0,1]$, it follows from Corollary \ref{F_epsilon til converge uniforme} that
	\begin{eqnarray*}
		\frac{F_e}{\tilde{F}_\varepsilon} = \frac{1}{\tilde{F}_\varepsilon} \rightarrow 1,
	\end{eqnarray*}
	 uniformly on $S_{F_e}[0,1]$ when $\varepsilon \rightarrow 0.$ So, there exists $\tilde{\tau} \in (0, \tau],$ such that
	\begin{eqnarray*}
	\left\vert	\frac{1}{\tilde{F}_\varepsilon(y)} - 1 \right\vert < \frac{\mu}{M_\ast}, \, \forall \, y \in S_{F_e}[0,1] \mbox{ and } \forall \, \varepsilon \in (0,\tilde{\tau}),
	\end{eqnarray*}
	where $M_\ast = \max \left\{ (F_e)_\ast(\mathfrak{b}) : \mathfrak{b} \in \Kgzero \right\}.$
	It follows that for every $\mathfrak{a} \in K_{\mathfrak{g}^*\setminus \{0\}}$ and $\varepsilon \in (0,\tilde{\tau})$, we have
	\begin{eqnarray*}
		\big\vert\big(\tilde{F}_\varepsilon\big)_\ast(\mathfrak{a}) - (F_e)_\ast(\mathfrak{a})\big\vert & = & \left| \max_{y \in S_{F_e}[0,1]} \frac{\mathfrak{a}(y)}{\tilde{F}_\varepsilon(y)} - \max_{y \in S_{F_e}[0,1]} \frac{\mathfrak{a}(y)}{F_e(y)}   \right| \\
		& = & \max_{y \in S_{F_e}[0,1]} \mathfrak{a}(y) \left\vert \frac{1}{\tilde{F}_\varepsilon(y)} - 1 \right\vert \\
		& < & \mu,
	\end{eqnarray*}
what settles the proof.
\end{proof}


\begin{lem} \label{Convergência das curvas de níveis} Let $K_{\mathfrak{g}^*\setminus \{0\}}$ a compact subset of $\mathfrak{g}^*\setminus\{0\}.$ Given $\mu>0,$ there exists $\tilde{\tau} \in (0, \tau]$ such that
	\begin{eqnarray*}
		\dist\left( \left\{ \frac{\mathfrak{a}}{\big(\tilde{F}_\varepsilon\big)_\ast(\mathfrak{a})} =1 \right\}, \left\{ \frac{\mathfrak{a}}{(F_e)_\ast(\mathfrak{a})}=1 \right\} \right)< \mu,
	\end{eqnarray*}
	for every $\mathfrak{a} \in K_{\mathfrak{g}^*\setminus \{0\}}$ and $\varepsilon \in (0, \tilde{\tau}).$	
\end{lem}

\begin{proof}	
	Given $\mathfrak{a} \in K_{\mathfrak{g}^*\setminus \{0\}},$ we have that 
	\begin{eqnarray*}
		\dist\left( \left\{ \frac{\mathfrak{a}}{\big(\tilde{F}_\varepsilon\big)_\ast(\mathfrak{a})} =1 \right\}, \left\{ \frac{\mathfrak{a}}{(F_e)_\ast(\mathfrak{a})}=1 \right\} \right)  = \frac{\big\vert \big(\tilde{F}_\varepsilon\big)_\ast (\mathfrak{a}) - (F_e)_\ast (\mathfrak{a})\big\vert }{\Vert \mathfrak{a} \Vert_{\ast} }
	\end{eqnarray*}
	due to \eqref{Distancia entre hiperplanos}.
	
	By Proposition \ref{Convergencia norma dual}, we have that $\big(\tilde{F}_\varepsilon\big)_\ast  \rightarrow (F_e)_\ast$ uniformly on $K_{\mathfrak{g}^*\setminus \{0\}}$ when $\varepsilon \rightarrow 0.$ Then, given $\mu> 0$, there exists  $\tilde{\tau} \in (0, \tau]$ such that
	\begin{eqnarray*}
		\big\vert \big(\tilde{F}_\varepsilon\big)_\ast (\mathfrak{a}) - (F_e)_\ast (\mathfrak{a})\big\vert < \mu \min \left\{ \big\Vert \mathfrak{b} \big\Vert_{\ast} : \mathfrak{b} \in K_{\mathfrak{g}^{\ast}\setminus \{0\}} \right\},
	\end{eqnarray*}
	for every $\mathfrak{a} \in K_{\mathfrak{g}^*\setminus \{0\}}$ and $\varepsilon \in (0, \tilde{\tau})$. 
	
	Then
	\begin{eqnarray*}
	 & & \dist\left( \left\{ \frac{\mathfrak{a}}{\big(\tilde{F}_\varepsilon\big)_\ast(\mathfrak{a})} =1 \right\}, \left\{ \frac{\mathfrak{a}}{(F_e)_\ast(\mathfrak{a})}=1 \right\} \right) \\   & = & \frac{\big\vert \big(\tilde{F}_\varepsilon\big)_\ast (\mathfrak{a}) - (F_e)_\ast (\mathfrak{a})\big\vert }{\Vert \mathfrak{a} \Vert_{\ast} }
	 < \frac{ \mu \min \left\{ \big\Vert \mathfrak{b} \big\Vert_{\ast} : \mathfrak{b} \in K_{\mathfrak{g}\setminus \{0\}} \right\} }{\Vert \mathfrak{a} \Vert_{\ast} }  \leq \mu,
	\end{eqnarray*}
	for every $\mathfrak{a} \in K_{\mathfrak{g}^*\setminus \{0\}}$ and $\varepsilon \in (0, \tilde{\tau})$, and the lemma is proved.
	
\end{proof}


\begin{lem} \label{As bolas contem uma bola euclidiana e estao contidas em uma bola euclidiana} \rm There are positive constants $\tau_1 \in (0, \tau]$, $r$ and $R$, such that
	\begin{eqnarray}
		& & B_{\Vert \cdot \Vert}[0, r] \subset B_{F_e}[0,1] \subset B_{\Vert \cdot \Vert}[0, R] \ {\rm and} \label{eqnarray_bolas encaixadas} \\
		& & B_{\Vert \cdot \Vert}[0, r] \subset B_{\tilde{F}_\varepsilon}[0,1] \subset B_{\Vert \cdot \Vert}[0, R], \label{eqnarray_bolas encaixadas epsilon}
	\end{eqnarray} 
	for every $\varepsilon \in (0, \tau_1).$
\end{lem}

\begin{proof} This result is a direct consequence of the uniform convergence  $\tilde{F}_{\varepsilon} \rightarrow F_e$  on $S_{F_e}[0,1]$ (see Corollary \ref{F_epsilon til converge uniforme}).
\end{proof}

\begin{lem} \label{Lema para mostrar a convergência de u varepsilon} Let $F_e$ be a strongly convex asymmetric norm on $\mathfrak{g}$. 
Let $K_{\mathfrak{g}^*\setminus \{0\}}$ be a compact subset of $\mathfrak{g}^*\setminus \{0\}$.  
For every $\mathfrak{a} \in K_{\mathfrak{g}^*\setminus \{0\}}$,
set $D_{u(\mathfrak{a})}=B_{\Vert \cdot \Vert} [u(\mathfrak{a}), r] \cap S_{F_e}[0,1]$ and 
${D_{u(\mathfrak{a})}}^\complement=S_{F_e}[0,1]-D_{u(\mathfrak{a})}$, where $r$ is defined in Lemma \ref{As bolas contem uma bola euclidiana e estao contidas em uma bola euclidiana}. The following relations hold:
	\begin{enumerate}
		\item $d_m := \inf\left\{\dist \left( {D_{u(\mathfrak{a})}}^\complement, \left\{ \frac{\mathfrak{a}}{(F_e)_\ast(\mathfrak{a}) } = 1\right\} \right): \, \mathfrak{a} \in  	K_{\mathfrak{g}^*\setminus \{0\}} \right\}> 0.$
		\item There exists $\tau_2 \in (0, \tau],$ such that
		\begin{eqnarray*}
			\left\Vert \frac{y}{\tilde{F}_\varepsilon(y)} - y \right\Vert < \frac{d_m}{4},
		\end{eqnarray*}
		for every $y \in S_{F_e}[0,1]$ and $\varepsilon \in (0, \tau_2).$
		\item For every $\mathfrak{a} \in K_{\mathfrak{g}^*\setminus \{0\}}$, we have 
		\begin{eqnarray*}
			\dist \left( \frac{y}{\tilde{F}_\varepsilon(y)}, \left\{ \frac{\mathfrak{a}}{(F_e)_\ast(\mathfrak{a}) } = 1\right\} \right) > \frac{3d_m}{4},
		\end{eqnarray*}
		for every $y \in {D_{u(\mathfrak{a})}}^\complement$ and $\varepsilon \in (0, \tau_2).$ 
	\end{enumerate}
	
\end{lem}
\begin{proof}
	
	\
	
	(1) Consider $\mathfrak{a} \in K_{\mathfrak{g}^*\setminus \{0\}}.$ 
	Notice that $F_e(z) \geq \frac{\mathfrak{a}(z)}{(F_e)_{\ast}(\mathfrak{a})}$ for every $z \in S_F[0,1]$, with equality iff $z=u(\mathfrak{a})$.
	Then $F_e(z) \geq F_e(u(\mathfrak{a})) +\frac{\mathfrak{a}}{(F_e)_\ast(\mathfrak{a})}(z-u(\mathfrak{a}))$ because $F_e(u(\mathfrak{a})) = \frac{\mathfrak{a}(u(\mathfrak{a}))}{(F_e)_{\ast}(\mathfrak{a})}=1$.
	Squaring both sides, we have $F_{e}{}^2(z)\geq F_{e}{}^2(u(\mathfrak{a})) + \frac{2\mathfrak{a} }{(F_e)_{\ast}(\mathfrak{a})}(z-u(\mathfrak{a}))$, that is, $\frac{2\mathfrak{a}}{(F_e)_\ast(\mathfrak{a})} \in \partial F_{e}{}^2(u(\mathfrak{a}))$.
	By Proposition \ref{proposition_para mostrar que S_F_e está contida em uma esfera euclidiana}, there exists $\tilde{y}(\mathfrak{a}) \in \mathfrak{g}$ and $\mathcal{R}>0$ (which doesn't depend on $\mathfrak{a}$), so that
		\begin{eqnarray*}
			& S_{F_e}[0,1] \subset B_{\Vert \cdot \Vert} [\tilde{y}(\mathfrak{a}), \mathcal{R}], \quad  \quad u(\mathfrak{a}) \in S_{\Vert \cdot \Vert} [\tilde{y}(\mathfrak{a}), \mathcal{R}] \\ & \text{ and } T_{u(\mathfrak{a})}(S_{\Vert \cdot\Vert}[\tilde{y}(\mathfrak{a}),\mathcal{R}]) = \left\{\frac{\mathfrak{a}}{(F_e)_{\ast}(\mathfrak{a})} = 1 \right\}.
		\end{eqnarray*}
		So
		\begin{eqnarray*}
			& & \dist\left({D_{u(\mathfrak{a})}}^\complement, \left\{\frac{\mathfrak{a}}{(F_e)_\ast(\mathfrak{a}) } = 1\right\} \right) \\
			& = & \dist\left(S_{F_e}[0,1] - B_{\Vert \cdot \Vert} [u(\mathfrak{a}), r] , \left\{\frac{\mathfrak{a}}{(F_e)_\ast(\mathfrak{a}) } = 1\right\} \right) \\
			&\geq &	\dist\left(S_{\Vert \cdot \Vert} [\tilde{y}(\mathfrak{a}), \mathcal{R}] - B_{\Vert \cdot \Vert} [u(\mathfrak{a}), r],  T_{u(\mathfrak{a})}(S_{\Vert \cdot\Vert}[\tilde{y}(\mathfrak{a}),\mathcal{R}]) \right).
		\end{eqnarray*}
		Note that the last distance above is positive and it doesn't depend on $ \mathfrak{a} \in K_{\mathfrak{g}^*\setminus \{0\}}$. Therefore 
		\begin{eqnarray*}
			d_m := \inf\left\{\dist \left( {D_{u(\mathfrak{a})}}^\complement, \left\{ \frac{\mathfrak{a}}{(F_e)_\ast(\mathfrak{a}) } = 1\right\} \right): \, \mathfrak{a} \in  K_{\mathfrak{g}^*\setminus \{0\}} \right\} > 0.
		\end{eqnarray*}
		
		\
		
		(2) As $S_{F_e}[0,1]$ is compact, then $\frac{1}{\tilde{F}_\varepsilon} \rightarrow \frac{1}{F_e}$ uniformly on $S_{F_e}[0,1].$ Set 
		\begin{eqnarray*}
			\mu = \frac{d_m}{4\max\limits_{z \in S_{F_e}[0,1]} \Vert z \Vert} > 0.
		\end{eqnarray*} 
		 Then there exists $\tau_2 \in (0, \tau],$ such that
		\begin{eqnarray*}
			\left| \frac{1}{\tilde{F}_\varepsilon(y)} - \frac{1}{F_e(y)} \right| < \mu,
		\end{eqnarray*}
		for every $y \in S_{F_e}[0,1]$ and $\varepsilon \in (0, \tau_2)$. Then, for every $\varepsilon \in (0, \tau_2)$ and  $y \in S_{F_e}[0,1],$ we have that
		\begin{eqnarray*}
			\Bigg\Vert  \frac{y}{\tilde{F}_\varepsilon(y)}  - y \Bigg\Vert & = & \Bigg\Vert  \frac{1}{\tilde{F}_\varepsilon(y)} y - \frac{1}{F_e(y)}y \Bigg\Vert  \\
			& < & \mu \Vert y \Vert \\
			& = & \frac{d_m}{4\max\limits_{z \in S_{F_e}[0,1]} \Vert z \Vert} \Vert y \Vert \\
			& \leq & \frac{d_m}{4},
		\end{eqnarray*}
		which completes the proof of Item (2).
		
		\
		
		(3) Consider $\mathfrak{a} \in K_{\mathfrak{g}^*\setminus \{0\}}.$ Given $y \in {D_{u(\mathfrak{a})}}^\complement,$ we have 
		\begin{eqnarray*}
			& & \dist\left(\frac{y}{\tilde{F}_\varepsilon(y)}, \left\{ \frac{\mathfrak{a}}{(F_e)_\ast(\mathfrak{a})} = 1 \right\} \right) \\
			& \geq & \dist\left( y, \left\{ \frac{\mathfrak{a}}{(F_e)_\ast(\mathfrak{a})} = 1 \right\}\right) -\dist\left( \frac{y}{\tilde{F}_\varepsilon(y)}, y\right)\\
			& > & d_m - \frac{d_m}{4} \\
			& = & \frac{3d_m}{4},
		\end{eqnarray*}
		for every $\varepsilon \in (0, \tau_2)$. In the second relation, we used Items (1) and (2).
	
\end{proof}


\begin{teo} \rm \label{Convergência de u alpha varepsilon} $u_{\varepsilon} \rightarrow u$ uniformly on compact subsets of $\mathfrak{g}^*\setminus \{0\}.$
\end{teo}

\begin{proof} Let $K_{\mathfrak{g}^*\setminus \{0\}}$ be a compact subset of $\mathfrak{g}^*\setminus \{0\}.$ Consider $\mu>0.$ Let $r$ and $R$ as defined in Lemma \ref{As bolas contem uma bola euclidiana e estao contidas em uma bola euclidiana}. 
We can suppose, without loss of generality, that $ r < \mu/2.$
Let $D_{u(\mathfrak{a})}$ and $d_m$ as given in Lemma \ref{Lema para mostrar a convergência de u varepsilon}. 

Consider 
\begin{eqnarray*} \label{Conjunto Z}
	A = B_{\Vert \cdot \Vert}[0, R] - B_{\Vert \cdot \Vert}(0, r) \subset \mathfrak{g}.
\end{eqnarray*}
Since $\frac{1}{\tilde{F}_\varepsilon(y)} \rightarrow \frac{1}{F_e(y)}$ uniformly on $A$, there exists 
\begin{eqnarray}\label{Constante tau3}
	\tau_3\in (0, \tau]
\end{eqnarray} 
such that
\begin{eqnarray}
\label{eqnarray_majorado por mu}
	\Bigg\vert \frac{1}{\tilde{F}_\varepsilon(y)} - \frac{1}{F_e(y)} \Bigg\vert < \frac{\mu}{2 \max\limits_{z \in A} \Vert z \Vert},
\end{eqnarray}
for every $ y \in A$ and $ \varepsilon \in (0, \tau_3)$. 
	
	Since $K_{\mathfrak{g}^*\setminus \{0\}}$ is compact, then by Lemma \ref{Convergência das curvas de níveis}, there exists $\tilde{\tau}>0,$ such that
	\begin{eqnarray} \label{desigualdade das curvas de níveis menor que d/2}
		\dist \left(\left\{\frac{\mathfrak{a}}{\big(\tilde{F}_\varepsilon\big)_\ast(\mathfrak{a})} = 1\right\}, \left\{\frac{\mathfrak{a}}{(F_e)_\ast(\mathfrak{a})} = 1\right\} \right) < \frac{d_m}{2},
	\end{eqnarray}
	for every $\mathfrak{a} \in K_{\mathfrak{g}^*\setminus \{0\}}$ and $\varepsilon \in (0, \tilde{\tau}).$ We can assume that $\tilde{\tau}\leq\min\{\tau_1, \tau_2, \tau_3\}$, where $\tau_1$ and $\tau_2$ are given in Lemmas \ref{As bolas contem uma bola euclidiana e estao contidas em uma bola euclidiana} and \ref{Lema para mostrar a convergência de u varepsilon} respectively, and $\tau_3$ is defined in (\ref{Constante tau3}).
	
	For every $\mathfrak{a} \in K_{\mathfrak{g}^*\setminus \{0\}}$ and $\varepsilon \in (0, \tilde{\tau}),$ we have
	\begin{eqnarray*}
		\Vert u_\varepsilon(\mathfrak{a}) - u(\mathfrak{a}) \Vert & \leq & \Bigg\Vert  u_\varepsilon(\mathfrak{a}) - \frac{u_\varepsilon(\mathfrak{a})}{F_e(u_\varepsilon(\mathfrak{a}))} \Bigg\Vert + \Bigg\Vert  \frac{u_\varepsilon(\mathfrak{a})}{F_e(u_\varepsilon(\mathfrak{a}))} - u(\mathfrak{a}) \Bigg\Vert\\
		& < &  \Bigg\Vert  \frac{u_\varepsilon(\mathfrak{a})}{\tilde{F}_\varepsilon(u_\varepsilon(\mathfrak{a}))} - \frac{u_\varepsilon(\mathfrak{a})}{F_e(u_\varepsilon(\mathfrak{a}))} \Bigg\Vert + r \\
		& = & \Bigg\vert  \frac{1}{\tilde{F}_\varepsilon(u_\varepsilon(\mathfrak{a}))} - \frac{1}{F_e(u_\varepsilon(\mathfrak{a}))} \Bigg\vert \Vert u_\varepsilon(\mathfrak{a}) \Vert + r \\
		& < & \frac{\mu}{2 \max\limits_{z \in A} \Vert z \Vert} \Vert u_\varepsilon(\mathfrak{a}) \Vert + r < \frac{\mu}{2}+ \frac{\mu}{2} = \mu.
	\end{eqnarray*}
	Note that $u_\varepsilon(\mathfrak{a}) \in A$ due to  \eqref{eqnarray_bolas encaixadas epsilon} and because $u_\varepsilon(\mathfrak{a}) \in S_{\tilde{F}_\varepsilon}[0,1]$.  
	For the second inequality, we use the fact that $\tilde{\tau}$ is less than or equal to $\tau_2$ to ensure that $\frac{u_\varepsilon(\mathfrak{a})}{F_e(u_\varepsilon(\mathfrak{a}))} \in D_{u(\mathfrak{a})},$ otherwise we would have $\frac{u_\varepsilon(\mathfrak{a})}{F_e(u_\varepsilon(\mathfrak{a}))} \in {D_{u(\mathfrak{a})}}^\complement$ 
	and
	\begin{eqnarray*}
		\dist\left(u_\varepsilon(\mathfrak{a}), \left\{ \frac{\mathfrak{a}}{(F_e)_\ast(\mathfrak{a})} = 1 \right\}  \right) \geq \frac{3d_m}{4}
	\end{eqnarray*}
	due to Item (3) of the Lemma \ref{Lema para mostrar a convergência de u varepsilon}, which contradicts (\ref{desigualdade das curvas de níveis menor que d/2}). 
	The fourth relation is due to $\tilde{\tau} \leq \tau_3$ and \eqref{eqnarray_majorado por mu}.
With these remarks, we conclude the proof.
\end{proof}


\begin{teo} \rm \label{Convergência E fresco til} $\tilde{\mathcal{E}}_\varepsilon \rightarrow \tilde{\mathcal{E}}$ uniformly on compact subsets of $\mathfrak{g}^*\setminus\{0\}$.
\end{teo}	

\begin{proof} Let $K_{\mathfrak{g}^*\setminus\{0\}}$ be a compact subset of $\mathfrak{g}^*\setminus\{0\}$. 
	 Given $\mu>0,$ we must find $\tilde{\tau} \in (0, \tau]$ such that
	\begin{eqnarray*}
		\big\Vert \tilde{\mathcal{E}}_\varepsilon(\mathfrak{a}) - \tilde{\mathcal{E}}(\mathfrak{a}) \big\Vert_* < \mu,
	\end{eqnarray*}
	for every $\mathfrak{a} \in K_{\mathfrak{g}^*\setminus\{0\}}$ and $\varepsilon \in (0, \tilde{\tau})$.
	
	Let $\{e_1, \ldots, e_n\}$ be an orthonormal basis of $\mathfrak{g}.$ Given $\mathfrak{a} \in K_{\mathfrak{g}^*\setminus\{0\}}$, we have
	\begin{eqnarray}
		\big\Vert \tilde{\mathcal{E}}_\varepsilon(\mathfrak{a}) - \tilde{\mathcal{E}}(\mathfrak{a}) \big\Vert_* & = & \big\Vert \mathfrak{a}([u_\varepsilon(\mathfrak{a}), \cdot]) - \mathfrak{a}([u(\mathfrak{a}), \cdot]) \big\Vert_* \nonumber\\
		& = & \big\Vert \mathfrak{a}([u_\varepsilon(\mathfrak{a}) - u(\mathfrak{a}), \cdot]) \big\Vert_* \nonumber \\
		& = & \max_{y \in S_{\Vert \cdot \Vert }[0,1]} \big\vert \mathfrak{a}([u_\varepsilon(\mathfrak{a}) - u(\mathfrak{a}), y])\big\vert \nonumber\\
		& \leq & \Vert \mathfrak{a}\Vert_{\ast} \max_{y \in S_{\Vert \cdot \Vert }[0,1]} \big\vert [u_\varepsilon(\mathfrak{a}) - u(\mathfrak{a}), y^ie_i]\big\vert \nonumber \\
		& = & \Vert \mathfrak{a}\Vert_{\ast}\max_{y \in S_{\Vert \cdot \Vert }[0,1]} \big\vert y^i \big\vert \big\vert [u_\varepsilon(\mathfrak{a}) - u(\mathfrak{a}), e_i]\big\vert \nonumber\\
		& \leq & n \Vert \mathfrak{a}\Vert_{\ast} \max_{ i \in \{1, \ldots, n\} } \big\vert [u_\varepsilon(\mathfrak{a}) - u(\mathfrak{a}), e_i]\big\vert. \label{Desigualdade1}
	\end{eqnarray}
	In the last inequality, note that $\big\vert y^i \big\vert \leq 1, $ for every $i \in \{1, \ldots, n\}$.
	
	For every $e_i \in \mathfrak{g},$ $i \in \{1, \ldots, n\}$, we have the uniform convergence
	$$[u_\varepsilon(\mathfrak{a}), e_i] \rightarrow [u(\mathfrak{a}), e_i]$$
	on $K_{\mathfrak{g}^*\setminus\{0\}}$ when $\varepsilon \rightarrow 0$, because the function $[\cdot , e_i]: \mathfrak{g} \rightarrow \mathbb{R}$ is continuous and $u_\varepsilon \rightarrow u$ uniformly on $K_{\mathfrak{g}^*\setminus\{0\}}$ due to Theorem \ref{Convergência de u alpha varepsilon}. Then there is $\tilde{\tau} \in (0, \tau]$ such that
	\begin{eqnarray} \label{Deseigualdade2}
		\big\Vert [u_\varepsilon(\mathfrak{a}), e_i]- [u(\mathfrak{a}), e_i] \big\Vert & < & \frac{\mu}{n \max\limits_{\mathfrak{b} \in K_{\mathfrak{g}^*\setminus\{0\}}} \Vert \mathfrak{b} \Vert_\ast},
	\end{eqnarray}
	for every $i \in \{1, \ldots, n\}$, $ \mathfrak{a} \in K_{\mathfrak{g}^*\setminus\{0\}}$ and $\varepsilon \in (0, \tilde{\tau})$. 
	
	Using (\ref{Desigualdade1}) and (\ref{Deseigualdade2}), we have that
	\begin{eqnarray*}
		\big\Vert \tilde{\mathcal{E}}_\varepsilon(\mathfrak{a}) - \tilde{\mathcal{E}}(\mathfrak{a}) \big\Vert_* & \leq & n \Vert \mathfrak{a} \Vert_\ast \max_{ i \in \{1, \ldots, n\} }  \big\Vert[u_\varepsilon(\mathfrak{a}) - u(\mathfrak{a}), e_i]\big\Vert \\
		& < & n \Vert \mathfrak{a} \Vert_\ast \frac{\mu}{n \max\limits_{\mathfrak{b} \in K_{\mathfrak{g}^*\setminus\{0\}}} \Vert \mathfrak{b} \Vert_\ast} \leq \mu,
	\end{eqnarray*}
	for every $ \mathfrak{a} \in K_{\mathfrak{g}^*\setminus\{0\}}$ and $\varepsilon \in (0, \tilde{\tau}),$ what settles the proof.
\end{proof}


\begin{teo} \label{u_epsilon é lipschitz} Let $K_{\mathfrak{g}^\ast \setminus \{0\}}$ be a compact subset of $\mathfrak{g}^\ast \setminus \{0\}.$ Then there exists $\tau_4 \in (0, \min\{\tau_0, \tau_1, \tau_2, \tau_3\}]$ and $\mathcal{K}_1 > 0$ such that $\Vert u(\mathfrak{a}) - u(\mathfrak{b})\Vert \leq \mathcal{K}_1.\Vert \mathfrak{a} - \mathfrak{b}  \Vert_{\ast}$ and 
$\Vert u_{\varepsilon}(\mathfrak{a}) - u_{\varepsilon}(\mathfrak{b})\Vert \leq \mathcal{K}_1.\Vert \mathfrak{a} - \mathfrak{b}  \Vert_{\ast}$ for every $\varepsilon \in (0,\tau_4)$ and $\mathfrak{a}, \mathfrak{b} \in K_{\mathfrak{g}^\ast \setminus \{0\}}$.
$\tau_0,$ $\tau_1$ and $\tau_2$ are given in Lemmas \ref{c_epsilon convege para c}, \ref{As bolas contem uma bola euclidiana e estao contidas em uma bola euclidiana} and \ref{Lema para mostrar a convergência de u varepsilon} respectively and $\tau_3$ is defined in (\ref{Constante tau3}).
\end{teo}

\begin{proof} By Lemma \ref{u_alpha é lipschitz}, we have that $u(\mathfrak{a})$ is Lipschitz on $K_{\mathfrak{g}^\ast \setminus \{0\}}$, with Lipschitz constant
	\begin{eqnarray*}
		\mathcal{K} = \frac{1}{2} \left(C_1^2 C_3 L_1 + C_1^2 C_2 \frac{4}{\gamma_0}\right),
	\end{eqnarray*}
	where $C_1, C_2,$ $C_3$, $L_1$ and $\gamma_0$ are defined by (\ref{Constante C1}), (\ref{Constante C2}), (\ref{Constante C3}), (\ref{Constante L1}) and \eqref{gamma_0} respectively.
	Here we replaced $\gamma$ by $\gamma_0$ in order to simplify the proof.
	
	Analogously, considering $\tau_0$ given in Lemma \ref{c_epsilon convege para c}, for each $\varepsilon \in (0, \tau_0),$ $u_\varepsilon(\mathfrak{a})$ is Lipschitz on $K_{\mathfrak{g}^\ast \setminus \{0\}}$ with Lipschitz constant
	\begin{eqnarray*}
		\mathcal{K}(\varepsilon) = \frac{1}{2} \left(C_1^2(\varepsilon) C_3(\varepsilon) L_1(\varepsilon) + C_1^2(\varepsilon) C_2(\varepsilon) \frac{4}{\gamma_0}\right),
	\end{eqnarray*}
	where
	\begin{eqnarray*}
		C_1(\varepsilon) & = & \max\left\{ \frac{1}{\big(\tilde{F}_\varepsilon\big)_\ast(\mathfrak{a})} : \, \mathfrak{a} \in K_{\mathfrak{g}^\ast \setminus \{0\}} \right\},  \\
		C_2(\varepsilon) & = & \max\left\{ \big(\tilde{F}_\varepsilon\big)_\ast(\mathfrak{a}) : \, \mathfrak{a} \in K_{\mathfrak{g}^\ast \setminus \{0\}} \right\},  \\
		C_3(\varepsilon) & = & \max\left\{\big\Vert d\big(\tilde{F}_\varepsilon\big)_\ast^{\ 2}(\mathfrak{a}) \big\Vert : \, \mathfrak{a} \in K_{\mathfrak{g}^\ast \setminus \{0\}} \right\} \text{ and }  \\
		L_1(\varepsilon) & = & \max\left\{ \big(\tilde{F}_\varepsilon\big)_\ast(\mathfrak{a}) : \, \mathfrak{a} \in S_{\Vert \cdot \Vert_\ast}[0,1]  \right\}.  
	\end{eqnarray*}
	By Proposition \ref{Convergencia norma dual} and the fact that the maximum function is continuous, we have that $C_1(\varepsilon), C_2(\varepsilon)$ and $L_1(\varepsilon)$ converge uniformly to $C_1, C_2$ and $L_1,$ respectively.
	We also have that due to (\ref{Expressão de u_epsilon(alpha)}), Proposition \ref{Convergencia norma dual} and Theorem \ref{Convergência de u alpha varepsilon}, that $C_3(\varepsilon)$ converges uniformly to $C_3.$ Then $\mathcal{K}(\varepsilon)$ converges uniformly to $\mathcal{K}$.
	It follows that there exists
	\begin{eqnarray*} \label{Constante tau_4}
	 \tau_4 \in \left(0, \min\{\tau_0, \tau_1, \tau_2, \tau_3\} \right]
	\end{eqnarray*}
	such that 
	$$\mathcal{K}(\varepsilon) < \mathcal{K}+1,$$
	for every $\varepsilon \in \left(0, \tau_4\right).$
	
	Therefore,
	\begin{eqnarray} \label{Constante K} 
	\mathcal{K}_1 = \mathcal{K}+1
	\end{eqnarray}
	is a common Lipschitz constant for the maps $u$ and $u_\varepsilon$ restricted to $K_{\mathfrak{g}^\ast \setminus \{0\}}$ when $\varepsilon \in \left(0, \tau_4\right).$
\end{proof}


\begin{teo} \label{a[u(a),.] é lipschitz} Let $K_{\mathfrak{g}^\ast \setminus \{0\}}$ be a compact subset of $\mathfrak{g}^\ast \setminus \{0\}.$
		Then there exists $\tilde{\mathcal{K}}>0$ such that 
		\begin{equation}
		\nonumber
		\Vert \tilde{\mathcal{E}}(\mathfrak{a}) - \tilde{\mathcal{E}}(\mathfrak{b})\Vert_\ast \leq \tilde{\mathcal{K}}\Vert \mathfrak{a} - \mathfrak{b} \Vert_\ast
		\end{equation}
		and
		\begin{equation}
		\nonumber 
		\Vert \tilde{\mathcal{E}_\varepsilon}(\mathfrak{a}) - \tilde{\mathcal{E}_\varepsilon}(\mathfrak{b})\Vert_\ast \leq \tilde{\mathcal{K}}\Vert \mathfrak{a} - \mathfrak{b} \Vert_\ast
		\end{equation}
		for every $\mathfrak{a},\mathfrak{b} \in \Kgzero$ and $\varepsilon \in (0,\tau_4)$, where $\tau_4$ is defined in Theorem \ref{u_epsilon é lipschitz}.
		\end{teo}

\begin{proof} 
The proof for $\tilde{\mathcal{E}}_{\varepsilon}(\mathfrak{a})$ is analogous to the proof of Theorem \ref{E til e Lipschitz} for $\tilde{\mathcal{E}}(\mathfrak{a})$.
Let $\hat{C}$ and $\tilde{C}$ as defined in Theorem \ref{E til e Lipschitz}, and $\{e_1, \ldots, e_n\}$ be an orthonormal basis of $\mathfrak{g}$.
	Consider $\mathfrak{a}, \mathfrak{b} \in K_{\mathfrak{g}^\ast \setminus \{0\}}$ and $\varepsilon \in \left(0, \tau_4\right)$. 
	Following the proof of Theorem \ref{E til e Lipschitz} until \eqref{eqnarray_estimativa Lipschitz u} for $\tilde{\mathcal{E}}_{\varepsilon}$, we have
	
	\begin{eqnarray}
		\big\Vert \tilde{\mathcal{E}}_\varepsilon(\mathfrak{a}) - \tilde{\mathcal{E}}_\varepsilon(\mathfrak{b}) \big\Vert_\ast & \leq &  \sum_{i,j=1}^n  \tilde{C} \big\Vert u_{\varepsilon}(\mathfrak{a})-u_{\varepsilon}(\mathfrak{b}) \big\Vert  \left\Vert \sum_{k=1}^n c_{ij}{}^k e_k \right\Vert \nonumber \\
		& + & \sum_{i,j=1}^n \big\Vert u_{\varepsilon}(\mathfrak{b})\big\Vert \Vert \mathfrak{a}-\mathfrak{b} \Vert_\ast \left\Vert \sum_{k=1}^n c_{ij}{}^k e_k \right\Vert. \nonumber
		\end{eqnarray}
Therefore
		\begin{eqnarray}
		\label{•}
		\big\Vert \tilde{\mathcal{E}}_\varepsilon(\mathfrak{a}) - \tilde{\mathcal{E}}_\varepsilon(\mathfrak{b}) \big\Vert_\ast  \leq	\left( \tilde{C} (\mathcal{K} + 1)  + R  \right)  \hat{C} n^3 \big\Vert \mathfrak{a}-\mathfrak{b} \big\Vert_\ast,\nonumber
	\end{eqnarray}
	where $R$ is defined in the Lemma \ref{As bolas contem uma bola euclidiana e estao contidas em uma bola euclidiana} and $\mathcal{K} + 1$ is a Lipschitz constant of $u_\varepsilon$ defined in \eqref{Constante K} for $\varepsilon \in (0, \tau_4)$. Then $\left( \tilde{C} (\mathcal{K} + 1) + R  \right)  \hat{C} n^3 $ is a Lipschitz constant for $\tilde{\mathcal{E}}_\varepsilon$ for $\varepsilon \in \left(0, \tau_4\right).$	
	This same constant is a Lipschitz constant for $\tilde{\mathcal{E}}(\mathfrak{a})$
because 
\[
\frac{1}{r_m}=\max_{y\in S_{F_e}[0,1]} \Vert y\Vert \leq R
\]
due to \eqref{align_maximo em SF eh 1 sobre rm} and \eqref{eqnarray_bolas encaixadas}.
	 Therefore,
	\begin{eqnarray*}
	\tilde{\mathcal{K}} := \left( \tilde{C} (\mathcal{K}+1) + R  \right)  \hat{C} n^3
	\end{eqnarray*} 
	is a Lipschitz constant for $\tilde{\mathcal{E}}(\mathfrak{a})$ and $\tilde{\mathcal{E}}_\varepsilon(\mathfrak{a})$, $\varepsilon \in \left(0, \tau_4\right).$
\end{proof}


\section{Uniform convergence of Pontryagin extremals}

\label{a epsilon convege para a e x epsilon converge para x}

Let $F_e : \mathfrak{g} \rightarrow  [0,\infty)$ be a strongly convex asymmetric norm with respect to $\sqrt{\frac{\gamma}{2}}\Vert \cdot \Vert.$ 
Consider $\mathfrak{a}_0 \in \mathfrak{g}^\ast \backslash \{0\}$. 
Let 
\begin{eqnarray*}
(u(\mathfrak{a}(t)), \mathfrak{a}(t)) \in S_{F_e}[0,1] \times \mathfrak{g}^*\setminus\{0\},
\end{eqnarray*}
$t \in \mathbb{R},$ be the unique Pontryagin extremal of $\big(G, F \big)$ such that $\mathfrak{a}(0) = \mathfrak{a}_0$ (see Theorem \ref{a(t) esta definida para toda a reta}). 
Recall that $\tilde{F}_\varepsilon : \mathfrak{g} \rightarrow  [0,\infty)$ is the mollifier smoothing of $F_e$ defined in (\ref{aplicacao F_epsilon til}), with $\varepsilon \in (0,\tau),$ where $\tau$ is defined in (\ref{Constante tau}). 
For every $\varepsilon \in (0, \tau)$, the Pontryagin extremal of $\big(G, \tilde{F}_\varepsilon \big)$ satisfying $\mathfrak{a}_\varepsilon (0)=\mathfrak{a}_0$ will be denoted by
\begin{eqnarray*}
	(u_\varepsilon (\mathfrak{a}_\varepsilon(t)), \mathfrak{a}_\varepsilon(t)) \in S_{\tilde{F}_\varepsilon[0,1]} \times \mathfrak{g}^*\setminus\{0\},
\end{eqnarray*}
$t \in \mathbb{R}.$


In Subsection \ref{subsection_convergencia em algebras}, we prove the uniform convergence 
\begin{eqnarray} \label{a epsilon converge uniformemente para a}
	(u_\varepsilon (\mathfrak{a}_\varepsilon(t)),\mathfrak{a}_\varepsilon(t)) \rightarrow (u(\mathfrak{a}(t)),\mathfrak{a}(t))
\end{eqnarray}
on compact intervals of $\mathbb{R}$ when $\varepsilon \rightarrow 0.$ 

In Subsection \ref{subsection_convergencia de extremais de Pontryagin no grupo}, we consider that the Lie group $G$ is connected.
Let $(x_0, \alpha_0) \in T^\ast G\backslash 0$ and $(x(t), \alpha(t))$ be the Pontryagin extremal of $(G,F)$ such that $(x(0), \alpha(0)) = (x_0, \alpha_0)$.
In the same way, let $(x_\varepsilon(t), \alpha_\varepsilon(t))$ be the Pontryagin extremals of $(G,F_\varepsilon)$ such that $(x_\varepsilon(0), \alpha_\varepsilon(0)) = (x_0, \alpha_0)$.
We prove that $(x_\varepsilon(t), \alpha_\varepsilon(t))$ converges uniformly to $(x(t), \alpha(t))$ on compact intervals as $\varepsilon \rightarrow 0$.


\subsection{Pontryagin extremals on $(\mathfrak{g}\backslash \{0\}) \times (\mathfrak{g}^\ast\backslash \{0\}$)}
\label{subsection_convergencia em algebras}

\begin{lem} \label{As bolas duais contem uma bola euclidiana e estao contidas em uma bola euclidiana} Let $\mathfrak{a}_0 \in \mathfrak{g}^\ast \backslash \{0\}$. 
The following statements holds:
	\begin{enumerate}
		\item There are $\tau_{B^\ast} \in (0,\tau_4]$ and positive constants $r^\ast$ and $R^\ast$ such that
		\begin{eqnarray*}
			& & B_{\Vert \cdot \Vert_\ast}[0, r^\ast] \subset B_{(F_e)_\ast}(0,(F_e)_\ast (\mathfrak{a}_0)) \ssubset B_{\Vert \cdot \Vert_\ast}(0, R^\ast);
		\end{eqnarray*}
and	
		\begin{eqnarray*}
			& & B_{\Vert \cdot \Vert_\ast}[0, r^\ast] \subset B_{(\tilde{F}_\varepsilon)_\ast}(0,(\tilde{F}_{\varepsilon})_\ast (\mathfrak{a}_0)) \ssubset B_{\Vert \cdot \Vert_\ast}(0, R^\ast)
		\end{eqnarray*}
for every $\varepsilon \in (0,\tau_{B^\ast})$. 
		\item Set $B^\ast := B_{\Vert \cdot \Vert_\ast}(0, R^\ast)-B_{\Vert \cdot \Vert_\ast}[0, r^\ast]$. For every $t \in \mathbb{R}$ and $\varepsilon \in (0, \tau_{B^\ast})$, we have $\mathfrak{a}(t) \in B^\ast$ and $\mathfrak{a}_{\varepsilon}(t) \in B^\ast$.	
		\end{enumerate}

\end{lem}
\begin{proof}
	Item (1): 
This result is a direct consequence of the uniform convergence  $(\tilde{F}_{\varepsilon})_\ast \rightarrow (F_e)_\ast$  on $S_{F_e}[0,1]$ as $\varepsilon \rightarrow 0$ (see Proposition \ref{Convergencia norma dual}).
	
	Item (2): $\mathfrak{a}(t)$ and $\mathfrak{a}_{\varepsilon}(t)$ are defined on $\mathbb{R} $ due to Theorem \ref{a(t) esta definida para toda a reta}. The inclusions $\mathfrak{a}(t) \in B^\ast$ and $\mathfrak{a}_{\varepsilon}(t)\in B^\ast$ for every $t\in \mathbb{R}$ and $\varepsilon \in (0,\tau_{B^\ast})$ follow from Item (1) and (\ref{É constante pelo PMP}).
\end{proof}


In what follows, we split the proof of \eqref{a epsilon converge uniformemente para a} in two parts: 
$\mathfrak{a}_\varepsilon(t) \rightarrow \mathfrak{a}(t)$ and $u_\varepsilon (\mathfrak{a}_\varepsilon(t)) \rightarrow u(\mathfrak{a}(t))$.

The key result for the proof of the convergence $\mathfrak{a}_\varepsilon(t) \rightarrow \mathfrak{a}(t)$ in \eqref{a epsilon converge uniformemente para a} is Theorem 7 of Filippov's book \cite{Filippov}.
Let us explain it according to our setting.
This result deals with non-autonomous Carath\'eodory ordinary differential equations $\dot{x}(t)=f(t,x;\mu)$ depending on the parameter $\mu$, which lies in some metric space. 
In our case, $\mu = \varepsilon \in [0, \tau_{B^\ast})$, where $\tau_{B^\ast}$ is given in Lemma  \ref{As bolas duais contem uma bola euclidiana e estao contidas em uma bola euclidiana}, and $\dot{x}=f(t,x;\mu)$ is the family of ordinary differential equations 
\begin{equation}
\label{equation_edo campo geodesico estendido}
\dot{\mathfrak{a}} = \mathfrak{a} (\left[u_{\varepsilon}(\mathfrak{a}), \cdot \right]).
\end{equation}
If $\varepsilon>0$, then \eqref{equation_edo campo geodesico estendido} is the equation corresponding to $(G,F_\varepsilon)$. 
If $\varepsilon = 0$, then $u_\varepsilon = u$ and \eqref{equation_edo campo geodesico estendido} is the equation corresponding to $(G,F)$.
The bounded domain $B$ given in \cite{Filippov} is the open subset $B^\ast \subset \mathfrak{g}^\ast$ given in Lemma \ref{As bolas duais contem uma bola euclidiana e estao contidas em uma bola euclidiana}.
Whenever we need estimates that hold in compact subsets of $\mathfrak{g}^\ast\backslash \{0\}$, we will set $\Kgzero = \overline{B^\ast}$.

For $\varepsilon>0$, we use the subindex $\varepsilon$ for solutions $\mathfrak{a}_\varepsilon(t)$ of \eqref{equation_edo campo geodesico estendido} satisfying the initial condition $\mathfrak{a}_\varepsilon(0) = \mathfrak{a}_0$. 
For $\varepsilon=0$, we simply denote this solution by $\mathfrak{a}(t)$. 
These notations differ from Filippov's notation, but this difference will not cause any misunderstandings.
Although \eqref{equation_edo campo geodesico estendido} is an autonomous family of differential equations, we denote 
\begin{equation}
\label{equation_notacao ftalphavarepsilon}
f(t,\mathfrak{a};\varepsilon) := \tilde{\mathcal{E}}_\varepsilon (\mathfrak{a}) = \mathfrak{a}([u_\varepsilon (\mathfrak{a}), \cdot])
\end{equation} 
in order to apply Theorem 7 of \cite{Filippov} more straightforwardly.

Now we present the aforementioned theorem adapted to our situation:


\begin{teo} \label{Teorema do Filippov} Let $\mathfrak{a}_0 \in \mathfrak{g}^\ast \backslash \{0\}$ and consider $B^\ast$ and $\tau_{B^\ast}$ as in Lemma \ref{As bolas duais contem uma bola euclidiana e estao contidas em uma bola euclidiana}. 
Let $\varepsilon \in [0,\tau_{B^\ast})$, $t \in [0,t_0]$, $\mathfrak{a} \in B^\ast$ and $f(t,\mathfrak{a};\varepsilon)$ defined by \eqref{equation_notacao ftalphavarepsilon}.
Suppose that
	
	\begin{enumerate}
		\item the function $f(t, \mathfrak{a}; \varepsilon)$ be measurable in $t$ for constant $\mathfrak{a}, \varepsilon$;
		\item $|f(t, \mathfrak{a}; \varepsilon)| \leq m(t, \varepsilon)$, the function $m(t, \varepsilon)$ being summable in $t$;
		\item there exist a summable function $l(t)$ and a monotone function $\psi(\rho) \to 0$ for $\rho \to 0$ such that for each $\rho > 0$, if $\Vert \mathfrak{a} - \mathfrak{b} \Vert_\ast \leq \rho$, then 
		\begin{equation*}
			\Vert f(t, \mathfrak{a}; \varepsilon) - f(t, \mathfrak{b}; \varepsilon)\Vert_{\ast} \leq l(t) \psi(\rho)
		\end{equation*}
		for almost all $t\in [0,t_0]$.
		\item for each $\mathfrak{a} \in B^\ast$, we have the uniform convergence 
		\begin{equation*}
			\int_{0}^{t} f(s, \mathfrak{a}; \varepsilon)\, ds \to \int_{0}^{t} f(s, \mathfrak{a}; 0)\, ds
		\end{equation*}
		as $\varepsilon \rightarrow 0$ on the segment $t \in [0,t_0]$;
		\item the solution $\mathfrak{a}(t)$ of the problem 
		\begin{equation} \label{eq:16}
			\dot{\mathfrak{a}} = f(t, \mathfrak{a}; \varepsilon), \quad \mathfrak{a}(0) = \mathfrak{a}_0
		\end{equation}
		 for $\varepsilon = 0$ is unique for $t \geq 0$ and lie in the domain $B^\ast$ for $t \in [0,t_0]$.
	\end{enumerate}
	Then for $\varepsilon$ sufficiently near $0$, the solutions $\mathfrak{a}_\varepsilon(t)$ of the problem (\ref{eq:16}) on the interval $[0,t_0]$ exist and converge uniformly to $\mathfrak{a}(t)$ as $\varepsilon \to 0$.
\end{teo}


Let us prove these items:

\begin{enumerate}
	\item The function $f(t,\mathfrak{a} ; \varepsilon)$ be measurable in $t$ for constant $ \mathfrak{a}$ and $\varepsilon$: In fact, the function $f$ does not depend on $t$.
	\item $\big\Vert \tilde{\mathcal{E}}_\varepsilon(\mathfrak{a}) \big\Vert_\ast \leq m(t,\varepsilon)$, the function $m(t, \varepsilon)$ is integrable in $t$: 
	\begin{eqnarray*}
	\big\Vert \tilde{\mathcal{E}}_\varepsilon(\mathfrak{a}) \big\Vert_\ast & = & \big\Vert \mathfrak{a}[u_\varepsilon(\mathfrak{a}), \cdot] \big\Vert_\ast = \max_{y \in S_{\Vert \cdot \Vert}[0,1]} \big\vert \mathfrak{a}[u_\varepsilon(\mathfrak{a}), y] \big\vert \\
	 & \leq & \max_{y \in S_{\Vert \cdot \Vert}[0,1]} \big\Vert \mathfrak{a} \big\Vert_\ast \big\Vert [u_\varepsilon(\mathfrak{a}), y] \big\Vert.
	\end{eqnarray*}
	 Notice that
	 \begin{itemize}
	 \item $\Vert\mathfrak{a}\Vert_\ast < R^\ast$ for every $\mathfrak{a} \in B^\ast$;
	 \item $\Vert u_\varepsilon(\mathfrak{a})\Vert \leq R$ for every $\varepsilon \in [0,\tau_{B^\ast})$ and $\mathfrak{a}\in B^\ast$ due to Lemma \ref{As bolas contem uma bola euclidiana e estao contidas em uma bola euclidiana};
	 \item $\Vert y\Vert=1$ because $y \in S_{\Vert \cdot \Vert}[0,1]$.
	 \end{itemize}
	 Therefore there exist $C>0$ such that $\Vert \tilde{\mathcal{E}}_\varepsilon (\mathfrak{a})\Vert_\ast < C$ for every $\varepsilon \in [0, \tau_{B^\ast})$ and $\mathfrak{a} \in B^\ast$ because $[\cdot, \cdot]$ is a continuous function, and it is enough to set $m(t,\varepsilon)=C$.

	\item
	There exist a summable function $l(t)$ and a monotone function $\psi(\rho) \to 0$ for $\rho \to 0$ such that for each $\rho > 0$, if $\Vert \mathfrak{a} - \mathfrak{b} \Vert_\ast \leq \rho$, then 		\begin{equation*}
			\big\Vert \tilde{\mathcal{E}}_\varepsilon(\mathfrak{a}) - \tilde{\mathcal{E}}_\varepsilon(\mathfrak{b}) \big\Vert_\ast = \Vert f(t, \mathfrak{a}; \varepsilon) - f(t, \mathfrak{b}; \varepsilon)\Vert_{\ast} \leq l(t) \psi(\rho)
		\end{equation*}
		for almost all $t\in [0,t_0]$:
	This item is a direct consequence of Theorem \ref{a[u(a),.] é lipschitz}, with $l(t)\equiv \tilde{\mathcal{K}}$ being the Lipschitz constant for $\Kgzero = \overline{B^\ast}$ and $\psi(\rho) = \rho$.
	\item For each $\mathfrak{a} \in B^\ast$ 
	\begin{eqnarray*}
		t \mapsto \int_{0}^t f(s, \mathfrak{a} ; \varepsilon)ds
	\end{eqnarray*}
	converges uniformly in $t \in [0,t_0]$
	to
	\begin{eqnarray*}
		t \mapsto \int_{0}^t f(s, \mathfrak{a}; 0) ds:
	\end{eqnarray*}
	
	 We have
	\begin{eqnarray*}
		\int_{0}^t f(s, \mathfrak{a} ; \varepsilon)ds 
		= \int_{0}^t \tilde{\mathcal{E}}_\varepsilon(\mathfrak{a}) ds 
		= t. \tilde{\mathcal{E}}_\varepsilon(\mathfrak{a}).
	\end{eqnarray*}
	In the same way, $\int_{0}^t f(s, \mathfrak{a}; 0) ds = t.\tilde{\mathcal{E}}(\mathfrak{a}).$ So, it suffices to show that
	\begin{eqnarray*}
		\tilde{\mathcal{E}}_\varepsilon(\mathfrak{a}) \rightarrow \tilde{\mathcal{E}}(\mathfrak{a}),
	\end{eqnarray*}
	which was proved in Theorem \ref{Convergência E fresco til}.
	\item	
	The solution $\mathfrak{a}(t)$ of initial value problem
	\begin{eqnarray} \label{PVI 16}
		\dot{\mathfrak{a}} = \mathfrak{a}[u(\mathfrak{a}), \cdot], \quad \mathfrak{a}(0) = \mathfrak{a}_0
	\end{eqnarray}
	is unique and lie in the domain $B^\ast$ for every $t \in \mathbb{R}:$ By Picard-Lindel\"of Theorem and Theorem \ref{a(t) esta definida para toda a reta} there is a unique solution $\mathfrak{a}(t)$ of (\ref{PVI 16}) defined for every $t \in \mathbb{R}.$ By Item (2) of Lemma \ref{As bolas duais contem uma bola euclidiana e estao contidas em uma bola euclidiana}, we have that $\mathfrak{a}(t) \in B^\ast$ for every $t \in \mathbb{R}.$
	
\end{enumerate}

	Therefore, the hypotheses of Theorem \ref{Teorema do Filippov} are satisfied for $f(t, \mathfrak{a}; \varepsilon) = \mathfrak{a}[u_\varepsilon(\mathfrak{a}), \cdot]$, and we can conclude that the unique solutions $t \in \mathbb{R} \mapsto \mathfrak{a}_{\varepsilon}(t)$ of the initial value problems \begin{eqnarray*}
	\dot{\mathfrak{a}} = \mathfrak{a} [u_\varepsilon(\mathfrak{a}), \cdot], \quad \mathfrak{a} (0) = \mathfrak{a}_0, \,\, \varepsilon \in [0,\tau_{B^\ast}),
\end{eqnarray*}
satisfy the uniform convergence
	\begin{eqnarray} \label{a_epsilon converge para a}
	\mathfrak{a}_\varepsilon(t) \rightarrow \mathfrak{a}(t)
\end{eqnarray}
on $[0,t_0]$ when $\varepsilon \rightarrow 0$.

Let us extend this result to the uniform convergence of \eqref{a_epsilon converge para a} to an arbitrary compact interval $[a,b] \subset \mathbb{R}$.
In fact, if we consider the unique solutions $t\in \mathbb{R} \mapsto \mathfrak{\i} (t)$ and $t\in \mathbb{R} \mapsto \mathfrak{\i}_\varepsilon (t)$ of the initial value problems
\begin{eqnarray*}
	\dot{\mathfrak{a}} = -\mathfrak{a} [u(\mathfrak{a}), \cdot], \quad \mathfrak{a} (0) = \mathfrak{a}_0, \,\, \varepsilon \in [0,\tau_{B^\ast})
\end{eqnarray*}
and
\begin{eqnarray*}
	\dot{\mathfrak{a}} = -\mathfrak{a} [u_\varepsilon(\mathfrak{a}), \cdot], \quad \mathfrak{a} (0) = \mathfrak{a}_0, \,\, \varepsilon \in [0,\tau_{B^\ast})
\end{eqnarray*}
respectively, direct calculations show that 
\begin{equation}
\label{relacao a e i}
\mathfrak{a}(t) = \mathfrak{\i}(-t) \text{ and }\mathfrak{a}_\varepsilon(t) = \mathfrak{\i}_\varepsilon(-t).
\end{equation}
Moreover, it is straightforward that $-f(t,\mathfrak{a};\varepsilon) = -\mathfrak{a}([u_\varepsilon(\mathfrak{a}),\cdot])$ also satisfies the conditions of Theorem \ref{Teorema do Filippov}.
Therefore, we have the uniform convergence
\begin{eqnarray*}
\mathfrak{\i}_\varepsilon(t) \rightarrow \mathfrak{\i}(t)
\end{eqnarray*}
on $[0,t_0]$ when $\varepsilon \rightarrow 0$, what implies that the uniform convergence \eqref{a_epsilon converge para a} holds on $[-t_0,0]$ due to \eqref{relacao a e i}. 
Since the uniform convergence \eqref{a_epsilon converge para a} holds on any interval $[-t_0,t_0]$, it also holds on any interval $[a,b] \subset \mathbb{R}$. 

Now we prove that $u_\varepsilon (\mathfrak{a}_\varepsilon(t)) \rightarrow u (\mathfrak{a}(t))$ uniformly on $[a,b].$ Given $t \in [a,b]$ we have
\begin{eqnarray}
	\big\Vert u_\varepsilon(\mathfrak{a}_\varepsilon(t)) - u(\mathfrak{a}(t)) \big\Vert	& \leq & \big\Vert u_\varepsilon(\mathfrak{a}_\varepsilon(t)) - u_\varepsilon(\mathfrak{a}(t)) \big\Vert + \big\Vert u_\varepsilon(\mathfrak{a}(t)) - u(\mathfrak{a}(t)) \big\Vert \nonumber \\
	& \leq & \mathcal{K}_1 \big\Vert \mathfrak{a}_\varepsilon(t) - \mathfrak{a}(t) \big\Vert_\ast + \big\Vert u_\varepsilon(\mathfrak{a}(t)) - u(\mathfrak{a}(t)) \big\Vert.\label{eqnarray_estimativa uespilont menos ut}
\end{eqnarray}
The relation $\big\Vert u_\varepsilon(\mathfrak{a}_\varepsilon(t)) - u_\varepsilon(\mathfrak{a}(t)) \big\Vert \leq \mathcal{K}_1 \big\Vert \mathfrak{a}_\varepsilon(t) - \mathfrak{a}(t) \big\Vert_\ast$ holds due to Theorem \ref{u_epsilon é lipschitz} and Item (2) of Lemma \ref{As bolas duais contem uma bola euclidiana e estao contidas em uma bola euclidiana}. 
The uniform convergence of \eqref{eqnarray_estimativa uespilont menos ut} to zero follows from \eqref{a_epsilon converge para a}, Theorem \ref{Convergência de u alpha varepsilon} applied to $\Kgzero = \overline{B^\ast}$, and Item (2) of Lemma \ref{As bolas duais contem uma bola euclidiana e estao contidas em uma bola euclidiana}.
Therefore, the following result was proved:
\begin{teo}
\label{theorem_convergencia uniforme nas algebras}
Let $G$ be a Lie group endowed with a left-invariant strongly convex $C^0$-Finsler structure $F$ and consider the left-invariant mollifier smoothing $F_\varepsilon$ of $F$ defined by \eqref{equation_define F_varespilon}.
Let $[a,b]\subset \mathbb{R}$ and consider $\mathfrak{a}_0 \in \mathfrak{g}^\ast \backslash \{0\}$.
Suppose that $(u(\mathfrak{a}(t)), \mathfrak{a}(t))$ and $(u_\varepsilon(\mathfrak{a}_\varepsilon(t)), \mathfrak{a}_\varepsilon(t))$, $\varepsilon \in (0,\tau_{B^\ast})$, $t\in \mathbb{R}$, are the Pontryagin extremals of $(G,F)$ and $(G,F_\varepsilon)$ in $\mathfrak{g}\backslash  \{0\}\times B^\ast$ respectively satisfying $\mathfrak{a}(0) = \mathfrak{a}_\varepsilon(0) = \mathfrak{a}_0$.
Then $(u_\varepsilon(\mathfrak{a}_\varepsilon(t)), \mathfrak{a}_\varepsilon(t))$ converges uniformly to $(u(\mathfrak{a}(t)), \mathfrak{a}(t))$ on $[a,b]$ as $\varepsilon \rightarrow 0$.
\end{teo}


\subsection{Pontryagin extremals on $T^\ast G\backslash 0$}
\label{subsection_convergencia de extremais de Pontryagin no grupo}

Now we present the version of Theorem \ref{theorem_convergencia uniforme nas algebras} on $T^\ast G\backslash 0$.
The key result for the proof of this result is Lemma 4.3.2 of \cite{Fritz},
which elements we introduce in the sequel according to our settings.

Let $\{e_1, \ldots, e_n\}$ be an orthonormal basis of $\mathfrak{g}$.
Denote the left-invariant vector fields on $G$ corresponding to $e_i$ by $x \in G \mapsto X_i(x) = d(L_x)_e(e_i)$. 

\begin{lem}
\label{Lemma_existencia e unicidade de extremal no grupo}
Let $G$ a Lie group, $x_0 \in G$ and $\mathfrak{u}(t): \mathbb{R} \rightarrow \mathfrak{g}$ a measurable function satisfying $\Vert \mathfrak{u}(t)\Vert < \tilde{R}$ for some $\tilde{R}>0$.
Then the initial value problem 
\begin{equation}
\label{equation_EDO para u mensuravel}
\dot{x}= \mathfrak{u}^j(t)X_j(x),\quad x(0)=x_0
\end{equation}
on $G$ admits a unique solution $x(t)$ defined on $\mathbb{R}$. 
\end{lem}

\begin{proof}
If we represent \eqref{equation_EDO para u mensuravel} in a coordinate system $(x^1, \ldots, x^n)$ on $U \ssubset G$, then there exist smooth functions $b_j^k:U \rightarrow \mathbb{R}^n$, $j,k \in \{1, \ldots, n\}$, such that 
\[
\dot{x}^i = \mathfrak{u}^j(t) b_j^i(x).
\]
It is straightforward that it satisfies the Carath\'eodory conditions for the local existence and uniqueness of solutions (see \cite{Filippov}).
Let $x(t):(a,b) \rightarrow G$ be the maximal solution of \eqref{equation_EDO para u mensuravel}.
We will suppose that $b<\infty$ and show that this leads to a contradiction. 
Consider the left-invariant Riemannian metric $x \mapsto \left< \cdot, \cdot \right>_x$ on $G$ such that $\{X_1(x), \ldots, X_n(x)\}$ is an orthonomal frame of $G$.
Denote the distance function with respect to 
$x \mapsto \left< \cdot , \cdot \right>_x$ by $d_G$.
Notice that $(G,d_G)$ is a complete metric space due to the Hopf-Rinow Theorem.
Moreover 
\begin{equation}
\label{equation dG Lipschitz}
d_G(x(t),x(s)) \leq \tilde{R} \vert t-s \vert
\end{equation} 
because $\mathfrak{u}(t) \in B_{\Vert \cdot\Vert}[0,\tilde{R}]$ and $\Vert \dot{x}(t) \Vert \leq \tilde{R}$ for almost every $t \in (a,b)$.
Therefore, if $t_i$ is a sequence in $(a,b)$ converging to $b$, then it is a Cauchy sequence, what implies that $x(t_i)$ is a Cauchy sequence in $(G,d_G)$.
Then $x(t_i)$ converges to a point in $G$ and $x(t)$ can be extended beyond $t=b$, what gives a contradiction. 
Therefore $b=\infty$.
In the same way we prove that $a=-\infty$, and we conclude that the maximal interval of definition of $x(t)$ is $\mathbb{R}$.
\end{proof}

Consider the diffeomorphism $T^\ast G \rightarrow G \times \mathfrak{g^\ast}$ given by $(x, \alpha) \mapsto$ $(x, d(L^\ast_x)_e(\alpha))$.
Define the metric $d_{T^\ast G}: T^\ast G \times T^\ast G \rightarrow \mathbb{R}$ given by 
\begin{align*}
& d_{T^\ast G}((x_1,\alpha_1),(x_2,\alpha_2)) \\ = &  \max\{ d_G(x_1, x_2), \Vert d(L^\ast_{x_1})_{e}(\alpha_1) - d(L^\ast_{x_2})_{e}(\alpha_2)\Vert_\ast \},
\end{align*}
which is the maximum norm on $G\times \mathfrak{g}^\ast$.
Let $(x_0, \alpha_0) \in T^\ast G\backslash 0$ and consider the Pontryagin extremals $t \in \mathbb{R} \mapsto (x(t),\alpha(t))$ and $t \in \mathbb{R} \mapsto (x_\varepsilon(t), \alpha_\varepsilon(t))$ of $(G,F)$ and $(G,F_\varepsilon)$ respectively satisfying $(x(0),\alpha(0)) = (x_\varepsilon(0),\alpha_\varepsilon(0)) = (x_0, \alpha_0)$.
According to Corollary \ref{corollary_Jessica Ryuichi}, there exist a unique $\mathfrak{a}(t)$ corresponding to $(x(t), \alpha(t))$.
Analogously, for each $\varepsilon \in (0,\tau)$, there exist a unique $\mathfrak{a}_\varepsilon(t)$ corresponding to $(x_\varepsilon(t),$ $\alpha_\varepsilon(t))$.
Then $(u_\varepsilon(\mathfrak{a}_\varepsilon(t)), \mathfrak{a}_\varepsilon(t))$ converges uniformly to $(u(\mathfrak{a}(t)), \mathfrak{a}(t))$ on $[a,b]$ as $\varepsilon \rightarrow 0$ due to 
Theorem \ref{theorem_convergencia uniforme nas algebras}.
We will prove that $(x_\varepsilon(t),\alpha_\varepsilon(t))$ converges uniformly to $(x(t),\alpha(t))$ on $[a,b]$ when $\varepsilon \rightarrow 0$. 

Define $\mathcal{U}= \{\mathfrak{u}:\mathbb{R} \rightarrow B_{\Vert \cdot\Vert}[0,R] \subset \mathfrak{g}: \,\, \mathfrak{u} \text{ is measurable}\}$, where $R$ is defined in Lemma \ref{As bolas contem uma bola euclidiana e estao contidas em uma bola euclidiana}. 
Then every measurable function $\mathfrak{u}(t)$ defined on $S_{F_e}[0,1]$ or else on $S_{\tilde{F}_\varepsilon}[0,1]$ lies in $\mathcal{U}$.

The Pontryagin extremal $(x(t), \alpha(t))$ corresponds to the control $u(\mathfrak{a}(t))$, what implies, $\dot{x}(t)=d(L_{x(t)})_e (u(\mathfrak{a}(t)))$.
If we write $u(\mathfrak{a}(t)) = u^j(\mathfrak{a}(t)) e_j$, we have that
\begin{eqnarray} \label{Sistema de controle afim 2}
	\dot{x}(t) = u^j(\mathfrak{a}(t)) X_j(x(t)), \quad x(0) = x_0.
\end{eqnarray}
Analogously, the Pontryagin extremals $(x_\varepsilon(t), \alpha_\varepsilon(t))$ of $(G,F_\varepsilon)$ correspond to the controls $u_\varepsilon(\mathfrak{a}_\varepsilon(t))$, and we have
\begin{eqnarray} \label{Sistema de controle afim 3}
	\dot{x}_\varepsilon(t) = u_\varepsilon^j(\mathfrak{a}_\varepsilon(t))X_j(x_\varepsilon(t)), \quad x_\varepsilon(0) = x_0.
\end{eqnarray}

Then we can consider (\ref{Sistema de controle afim 2}) and (\ref{Sistema de controle afim 3}) as differential equations of the affine control system
\begin{eqnarray*}\label{sistema de controle}
	\dot{x}(t) = \mathfrak{u}^j(t) X_j (x(t)),\quad \mathfrak{u}(t)= \mathfrak{u}^i(t)e_i \in B_{\Vert \cdot\Vert}[0,R].
\end{eqnarray*}

Lemma \ref{Lemma_existencia e unicidade de extremal no grupo} states that for every $(x_0, \mathfrak{u}) \in G \times \mathcal{U}$, there is a unique solution $t\in \mathbb{R} \mapsto$ $\phi(t,x_0, \mathfrak{u}),$ of the initial value problem $\dot{x} = \mathfrak{u}^j(t) X_j(x)$, $\phi(0, x_0, \mathfrak{u}) = x_0.$ Furthermore, Lemma 4.3.2 of \cite{Fritz} states that
\begin{eqnarray*}
	\phi: \mathbb{R} \times G \times \mathcal{U} \rightarrow G
\end{eqnarray*}
is continuous when $\mathcal{U}\subset L_\infty(\mathbb{R}, \mathbb{R}^n) = \left( L_1(\mathbb{R}, \mathbb{R}^n) \right)^*$ is endowed with the weak$^*$ topology. 
Since the weak$^*$ topology is coarser than the strong topology, we have that $\phi: \mathbb{R} \times G \times \mathcal{U} \rightarrow G$ is continuous when $\mathcal{U}$ is equipped with the topology of the essential supremum norm, which is the topology of $\mathcal{U}$ that we will consider. 

Now we are in position to state the main theorem of this work.

\begin{teo} 
\label{theorem_principal no fibrado cotangente} 
Let $G$ be a Lie group endowed with a left-invariant strongly convex $C^0$-Finsler structure $F$.
Let $F_\varepsilon$, $\varepsilon \in (0,\tau_{B^\ast})$, be the left-invariant mollifier smoothing of $F$ defined by \eqref{equation_define F_varespilon}.
Consider $(x_0, \alpha_0) \in T^\ast G\backslash 0$ and let $(x(t),\alpha(t))$ and $(x_\varepsilon(t), \alpha_\varepsilon(t))$ be the Pontryagin extremals of $(G,F)$ and $(G,F_\varepsilon)$ respectively satisfying $(x(0), \alpha(0)) = (x_\varepsilon(0), \alpha_\varepsilon(0)) = (x_0, \alpha_0)$. 
If $[a,b]$ is a compact interval, then we have the uniform convergence 
\[
(x_\varepsilon(t), \alpha_\varepsilon(t)) \mapsto (x(t), \alpha(t))
\] 
on $[a,b]$ as $\varepsilon \rightarrow 0$.
\end{teo}

\begin{proof} 
Firstly we prove the uniform convergence $x_\varepsilon(t) \rightarrow x(t)$ on $[a,b]$ as $\varepsilon \rightarrow 0$.
Let $\mu > 0$. We will prove that there exist $\tau_{\mu} \in (0,\tau_{B^\ast})$ such that if $\varepsilon \in (0,\tau_{\mu})$, then $d_G(x_\varepsilon(t),x(t))< \mu$ for every $t \in [a,b]$.
Notice that if $t,s \in [a,b]$ are such that $\vert t - s\vert < \frac{\mu}{3R}$, then 
\begin{equation}
\label{equation_distancia G xvarepsilon t e s}
d_G(x_\varepsilon(t), x_\varepsilon(s)) < \frac{\mu}{3}
\end{equation} 
for every $\varepsilon \in [0,\tau_{B^\ast})$ due to \eqref{equation dG Lipschitz}.

Consider a partition $\{t_0 =a, t_1, \ldots,$ $ t_N = b\}$ such that for every $t \in [a,b]$, there exist a $t_i$ such that $\vert t-t_i\vert < \frac{\mu}{3R}$.
For every $t_i$, $i=0,1, \ldots, N$, consider $\phi^{-1}(B_{d_G}(x(t_i),\frac{\mu}{3})) \cap \left( \{t_i\} \times \{x_0\} \times \mathcal{U} \right)$, which can be identified with an open subset of $\mathcal{U}$.
Of course $u(\mathfrak{a}(t))$ lies in this open subset.
Then there exist $r_{\mu,i} > 0$ such that for every $\tilde{\mathfrak{u}}(t) \in \mathcal{U}$ satisfying $\esssup_{t\in \mathbb{R}} \Vert u(\mathfrak{a}(t)) -  \tilde{\mathfrak{u}}(t) \Vert < r_{\mu,i}$, we have that 
\[
d_G(\phi(t_i,x_0,\tilde{\mathfrak{u}}(t)),\phi(t_i,x_0, u(\mathfrak{a}(t)))) < \frac{\mu}{3}.
\]

Set $r_{\mu} = \min_{i=0,\ldots,N} r_{\mu,i}$.
Since $u_\varepsilon(\mathfrak{a}_\varepsilon(t))$ converges uniformly to $u(\mathfrak{a}(t))$ on $[a,b]$ when $\varepsilon \rightarrow 0$, then there exist a $\tau_{\mu} \in (0,\tau_{B^\ast})$ such that $\Vert  u_\varepsilon(\mathfrak{a}_\varepsilon(t)) - u(\mathfrak{a}(t))\Vert < r_{\mu}$ for every $t \in [a,b]$ and $\varepsilon \in (0,\tau_{\mu})$, in which case we have
\begin{align}
\nonumber
& d_G(\phi(t_i,x_0, u_\varepsilon(\mathfrak{a}_\varepsilon(t))),\phi(t_i,x_0, u(\mathfrak{a}(t)))) & \\ 
& = d_G(x_\varepsilon(t_i),x(t_i)) < \frac{\mu}{3} & \label{align_estimativa ponto amostral}
\end{align}
for every $i\in \{0, \ldots, N\}$.

Now consider $t\in [a,b]$ and $\varepsilon \in (0,\tau_{\mu})$.
Let $i \in \{0,\ldots, N\}$ such that $\vert t-t_i \vert < \frac{\mu}{3R}$.
Then
\begin{eqnarray*}
d_G(x_\varepsilon(t), x(t)) & = & d_G(x_\varepsilon(t),x_\varepsilon(t_i))+d_G(x_\varepsilon(t_i),x(t_i))+d_G(x(t_i),x(t)) \\
& < & \frac{\mu}{3} + \frac{\mu}{3}+ \frac{\mu}{3} = \mu
\end{eqnarray*}
due to \eqref{equation_distancia G xvarepsilon t e s} and \eqref{align_estimativa ponto amostral}, what proves the uniform convergence $x_\varepsilon(t) \rightarrow x(t)$ on the interval $[a,b]$ as $\varepsilon \rightarrow 0$.

In order to prove the uniform convergence $(x_\varepsilon(t),\alpha_\varepsilon(t)) \rightarrow (x(t), \alpha(t))$ on $[a,b]$ as $\varepsilon \rightarrow 0$, notice that
\begin{eqnarray*}
& \max\limits_{t\in [a,b]} \left\{ d_{T^\ast G}((x_\varepsilon(t), \alpha_\varepsilon(t)),(x(t), \alpha(t))) \right\} \\ 
& = \max\limits_{t\in [a,b]} \left\{ \max \{d_G(x_\varepsilon(t), x(t)), \Vert d(L^\ast_{x_\varepsilon(t)})_{e}(\alpha_\varepsilon(t)) -d(L^\ast_{x(t)})_{e}(\alpha(t))\Vert_\ast \} \right\} \\
& = \max\limits_{t\in [a,b]} \left\{ \max \{d_G(x_\varepsilon(t), x(t)), \Vert \mathfrak{a}_\varepsilon(t) - \mathfrak{a}(t) \Vert_\ast \}\right\},
\end{eqnarray*}
which converges zero as $\varepsilon \rightarrow 0$ due to Theorem \ref{theorem_convergencia uniforme nas algebras} and the uniform convergence $x_\varepsilon(t) \rightarrow x(t)$ on the interval $[a,b]$ as $\varepsilon \rightarrow 0$.
\end{proof}


\end{document}